\documentclass[11pt]{article}

\usepackage{amssymb,amsmath,amsthm,amsfonts}
\usepackage{subfigure, hyperref, multirow, bm, color, stmaryrd, marginnote, graphicx}

\numberwithin{equation}{section}
\allowdisplaybreaks

\newcommand{\R}{\mathbb{R}}
\newcommand{\Z}{\mathbb{Z}}

\newcommand{\A}{\mathcal{A}}
\newcommand{\E}{\mathcal{E}}
\newcommand{\D}{\mathcal{D}}
\newcommand{\sR}{\mathcal{R}}

\newcommand{\bars}{\overline s}
\newcommand{\bu}{\bm{u}}
\newcommand{\bw}{\bm{w}}
\newcommand{\bx}{\bm{x}}
\newcommand{\X}{\bm{X}}
\newcommand{\be}{\bm{e}}
\newcommand{\bv}{\bm{v}}
\newcommand{\p}{\partial}
\newcommand{\ts}{\thinspace}
\newcommand{\dive}{{\rm{div}}}
\newcommand{\SB}{{\rm SB}}
\newcommand{\ws}{\wt{s}}
\newcommand{\wR}{\wt{\bm{R}}}
\newcommand{\abs}[1]{\left\lvert #1 \right\rvert}
\newcommand{\norm}[1]{\left\lVert #1 \right\rVert}
\newcommand{\wh}[1]{\widehat{#1}}
\newcommand{\wt}[1]{\widetilde{#1}}
\newcommand{\mc}[1]{\mathcal{#1}}

\newtheorem{theorem}{Theorem}[section]
\newtheorem{lemma}[theorem]{Lemma}
\newtheorem{proposition}[theorem]{Proposition}
\newtheorem{corollary}[theorem]{Corollary}
\theoremstyle{definition}
\newtheorem{remark}[theorem]{Remark}
\theoremstyle{definition}
\newtheorem{definition}{Definition}[section]
\allowdisplaybreaks

\begin{document}
\title{Theoretical justification and error analysis for slender body theory with free ends}
\author{Yoichiro Mori, Laurel Ohm, Daniel Spirn
\footnote{This research was supported in part by NSF grant DMS-1620316 and DMS-1516978, awarded to Y.M., by NSF GRF grant 00039202 and a Torske Kubben Fellowship, awarded to L.O., and by NSF grant DMS-1516565, awarded to D.S. The authors also thank the IMA where most of this work was performed. }\\ \textit{\small School of Mathematics, University of Minnesota, Minneapolis, MN 55455}}
\date{\today}

\maketitle

\begin{abstract}
Slender body theory is a commonly used approximation in computational models of thin fibers in viscous fluids, especially in simulating the motion of cilia or flagella in swimming microorganisms. In \cite{closed_loop}, we developed a PDE framework for analyzing the error introduced by the slender body approximation for closed-loop fibers with constant radius $\epsilon$, and showed that the difference between our closed-loop PDE solution and the slender body approximation is bounded by an expression proportional to $\epsilon|\log\epsilon|$. Here we extend the slender body PDE framework to the free endpoint setting, which is more physically relevant from a modeling standpoint but more technically demanding than the closed loop analysis. The main new difficulties arising in the free endpoint setting are defining the endpoint geometry, identifying the extent of the 1D slender body force density, and determining how the well-posedness constants depend on the non-constant fiber radius. Given a slender fiber satisfying certain geometric constraints at the filament endpoints and a one-dimensional force density satisfying an endpoint decay condition, we show a bound for the difference between the solution to the slender body PDE and the slender body approximation in the free endpoint setting. The bound is a sum of the same $\epsilon|\log\epsilon|$ term appearing in the closed loop setting and an endpoint term proportional to $\epsilon$, where $\epsilon$ is now the maximum fiber radius. 
\end{abstract}

\maketitle

\tableofcontents
\section{Introduction}
Consider a microorganism swimming through a viscous fluid, propelled by a slender flagellum or an array of thin cilia. Understanding how these thin structures interact with the surrounding fluid to propel the microorganism forward may aid in the  design of artificial cilia for use in microfluidic lab-on-a-chip devices, which facilitate rapid chemical reactions by mixing fluids on small scales \cite{van2009printed, hanasoge2018microfluidic, khaderi2011microfluidic}, or in self-propulsion of nanorobots for targeted drug delivery \cite{pak2011high, kanevsky2010modeling, peyer2013bio, kim2016fabrication}. With these applications in mind, it is important to create accurate, theoretically sound mathematical models describing the behavior of slender fibers immersed in fluid. \\

To model the interaction between a thin filament and a viscous fluid in $\R^3$, we begin with the Stokes equations describing the velocity $\bu$ and pressure $p$ of the fluid: 
\begin{equation}\label{stokesEQ}
- \Delta \bu +\nabla p =0, \quad \dive \ts \bu = 0, 
\end{equation}
along with suitable boundary conditions prescribed over the fiber surface. However, from a practical perspective, parameterizing and prescribing boundary conditions over the entire slender body surface is very computationally intensive, especially since we are often interested in simulating tens \cite{tornberg2004simulating} or even hundreds \cite{clague1997numerical} of individual fibers. \\  

This is where slender body theory becomes useful. The basic idea behind slender body theory is to exploit the thinness of the fiber: instead of being modeled as a fully three-dimensional object, the filament is treated as a one-dimensional force density along a curve in $\R^3$. The early pioneers of slender body theory, including Hancock \cite{hancock1953self}, Cox \cite{cox1970motion}, and Batchelor \cite{batchelor1970slender}, treated the filament as a curve of point sources only. Later, Keller and Rubinow \cite{keller1976slender} and Johnson \cite{johnson1980improved} added higher order corrections to this basic slender body theory to develop an improved slender body theory, yielding an approximation to the fiber centerline velocity \eqref{SBT_asymp_free} that has been used as a basis for many computational models of thin fibers, including \cite{tornberg2004simulating, gotz2000interactions, li2013sedimentation, chattopadhyay2009effect, spagnolie2011comparative}, and \cite{smith2007discrete}. Johnson in particular introduced the use of a doublet correction, forming what we will refer to as the classical slender body approximation \eqref{SBT_free0}. \\

However, until recently, there has been a lack of rigorous theoretical justification for slender body theory. In particular, given force data prescribed only along a one-dimensional curve in $\R^3$, it was not immediately obvious how to state the problem to which slender body theory is an approximation as a well-posed PDE. It was therefore unclear how to characterize and analyze the error introduced by approximating a truly three-dimensional fiber as a one-dimensional object. As a starting point, we consider only the static problem of imposing a time-independent force along the fiber. \\

In \cite{closed_loop}, we propose a rigorous error analysis framework for slender body theory given that the fiber is a closed loop with constant radius $\epsilon$. We define the \emph{slender body PDE} as a Stokes boundary value problem with only partial Dirichlet and partial Neumann information specified at each point along the fiber surface. The partial Dirichlet information, which we term the \emph{fiber integrity condition}, allows us to make sense of total force data (the partial Neumann information) prescribed only on a one-dimensional curve in $\R^3$. After tracking the $\epsilon$-dependence in various constants arising in the well-posedness theory for slender body PDE, we compare this PDE solution to the closed-loop slender body approximation along the actual fiber surface and derive an error estimate. We ultimately show that, given a few regularity assumptions on the fiber centerline and the force density along the fiber, the error in the slender body approximation is bounded by an expression proportional to $\epsilon\abs{\log\epsilon}$. \\

The closed loop analysis presents a good first step toward placing slender body theory on a solid theoretical foundation. However, each of the above-mentioned modeling applications concern thin filaments that are not closed loops, but instead have free endpoints. The closed loop problem is easier to write down and analyze, but the free endpoint problem is much more important from a modeling perspective. In this paper, we extend the slender body PDE framework designed in \cite{closed_loop} to derive an error estimate for the slender body approximation for a fiber with free ends. The resulting error bound is the same order as the closed loop estimate, provided the prescribed force decays sufficiently at the fiber endpoints.


\subsection{Free endpoint slender body geometry}\label{geometry_section}
We begin by introducing some geometric considerations for a slender body with free endpoints. Let $\X_{\text{ext}}: [-3/2,3/2] \to \R^3$ denote the coordinates of a curve $\Gamma_0\in\R^3$, parameterized by arc length $\varphi$. We assume $\X_{\text{ext}}\in \mc{C}^2(-3/2,3/2)$ so that the curvature $\kappa(\varphi)$ is well defined for each $\varphi\in(-3/2,3/2)$. \\

As in the closed loop setting (see \cite{closed_loop}), the curve $\Gamma_0$ is assumed to be non-self-intersecting: there exists $c_\Gamma>0$ such that
\begin{equation}\label{no_intersect}
\inf_{\varphi_1\neq \varphi_2} \frac{\abs{\X_\text{ext}(\varphi_1) - \X_\text{ext}(\varphi_2)}}{\abs{\varphi_1-\varphi_2}} \ge c_\Gamma.
\end{equation} 

The region around $\Gamma_0$ can be described via a $\mc{C}^1$ orthonormal frame along $\X_\text{ext}$. We first define the tangent vector
\[ \be_{\rm t}(\varphi) = \frac{d\X_\text{ext}}{d\varphi}(\varphi). \]

Choosing any unit vector $\be_{n_1}(0)$ normal to the tangent vector $\be_{\rm t}(0)$, we define the vector field $\be_{n_1}(\varphi)\in \mc{C}^1$ via parallel transport along $\X_\text{ext}$ such that $\be_{n_1}(\varphi)\perp\be_{\rm t}(\varphi)$ at each $\varphi\in\X_\text{ext}$. Define $\be_{n_2}(s) =\be_{\rm t}(s) \times \be_{n_1}(s)$. Then the triple $\be_{\rm t},\be_{n_1},\be_{n_2}$ forms a $\mc{C}^1$ orthonormal frame known as a \emph{Bishop frame} \cite{bishop1975there}, and satisfies the ODE 
\begin{equation}\label{bishop_ODE}
\frac{d}{d\varphi} \begin{pmatrix}
\be_{\rm t}(\varphi) \\
\be_{n_1}(\varphi) \\
\be_{n_2}(\varphi)
\end{pmatrix} = \begin{pmatrix}
0 & \kappa_1(\varphi) & \kappa_2(\varphi) \\
-\kappa_1(\varphi) & 0 & 0 \\
-\kappa_2(\varphi) & 0 & 0
\end{pmatrix} \begin{pmatrix}
\be_{\rm t}(\varphi) \\
\be_{n_1}(\varphi) \\
\be_{n_2}(\varphi)
\end{pmatrix}.
\end{equation}
Here $\kappa_1$ and $\kappa_2$ are continuous functions of $\varphi$ satisfying the the relation
\begin{equation}\label{curvature_eq}
\kappa(\varphi) = \sqrt{\kappa_1^2(\varphi)+\kappa_2^2(\varphi)},
\end{equation}
where $\kappa(\varphi)$ is the curvature of $\X_\text{ext}(\varphi)$. We define
\begin{equation}\label{kappamax_free}
\kappa_{\max} := \max_{\varphi\in (-3/2,3/2)} \abs{\kappa(\varphi)}.
\end{equation} 
Note that the Bishop frame is similar to the frame constructed in the closed loop setting (\cite{closed_loop}, Lemma 1.1), but in the free endpoint case we may take the coefficient $\kappa_3=0$ as we no longer require periodicity of the frame. \\

We define cylindrical unit vectors with respect to the Bishop frame:
\begin{align*}
\be_{\rho}(\varphi,\theta) &:= \cos\theta\be_{n_1}(\varphi)+ \sin\theta\be_{n_2}(\varphi) \\
\be_\theta(\varphi,\theta) &:= -\sin\theta \be_{n_1}(\varphi)+\cos\theta\be_{n_2}(\varphi).
\end{align*}

Now, for some
\begin{equation}\label{rmax}
r_{\max} = r_{\max}(c_\Gamma,\kappa_{\max})
\end{equation}
we may parameterize points $\bx$ with dist$(\bx,\Gamma_0)\le r_{\max}$ as a tube about $\Gamma_0$:
\begin{equation}\label{tube_free}
\bx = \X_\text{ext}(\varphi)+\rho \be_{\rho}(\varphi,\theta).
\end{equation}

To define a free-end slender body, we consider a subset of the curve $\X_\text{ext}(\varphi)$. Let $0<\epsilon\le r_{\max}/4$ and define 
\begin{equation}\label{SB_centerline}
\X(\varphi) = \{ \X_\text{ext}(\varphi) \ts : \ts -\eta_\epsilon \le \varphi \le \eta_\epsilon \}.
\end{equation} 
 Here $1<\eta_\epsilon< \frac{3}{2}$ is a value close to 1; in particular,
\begin{equation}\label{eta_epsilon}
c_{\eta,0} \epsilon^2 \le \eta_\epsilon-1 \le c_\eta \epsilon^2, \quad c_{\eta,0}, c_\eta >0.
\end{equation}

\begin{figure}[!h]
\centering
\includegraphics[scale=0.65]{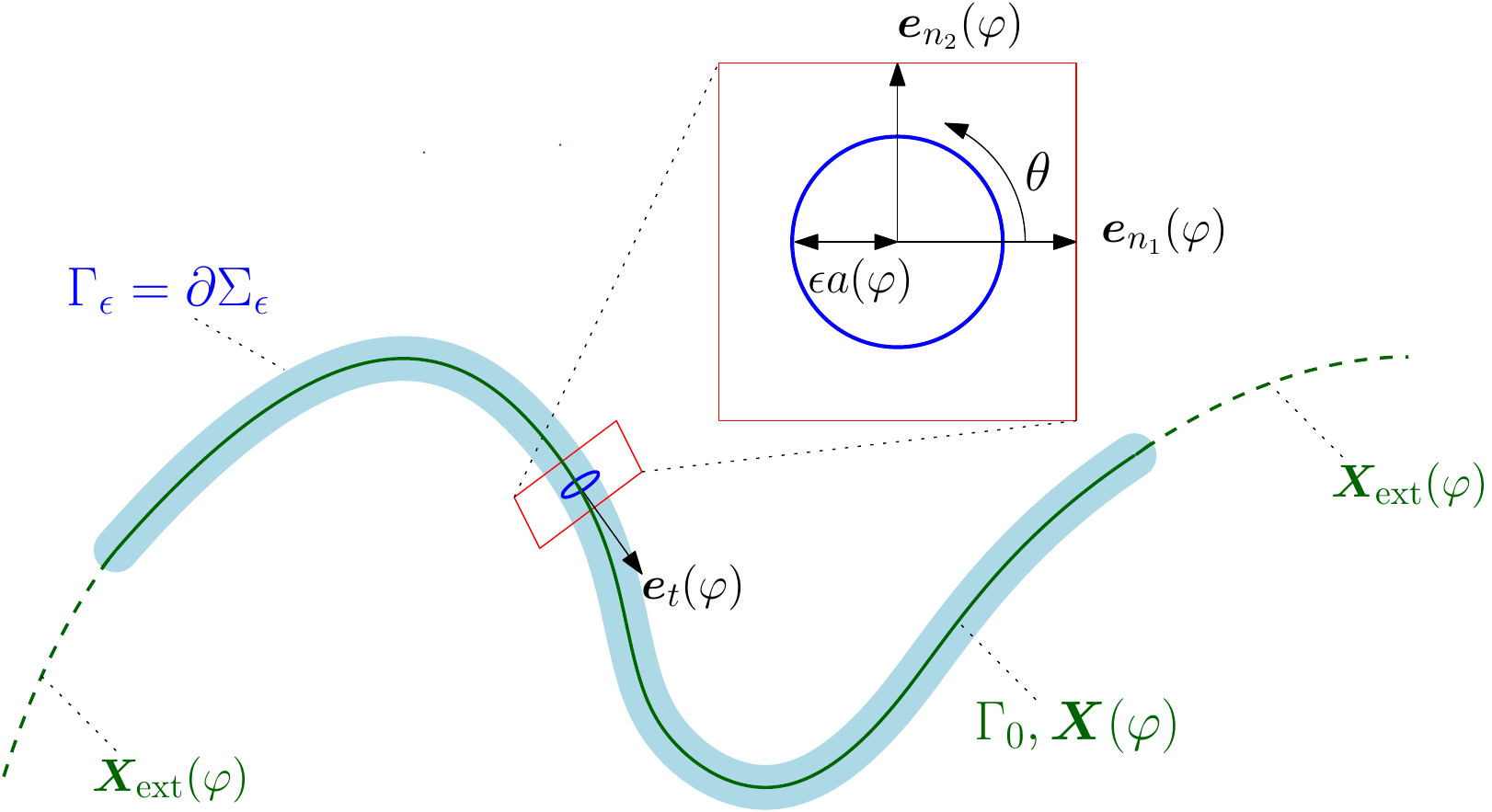} \\
\caption{The geometry of the free-endpoint slender body is specified with respect to a $\mc{C}^1$ Bishop frame \eqref{bishop_ODE} along the fiber centerline $\X(\varphi)$. We assume that the fiber centerline $\X(\varphi)$ can be considered as a subset of a longer curve $\X_\text{ext}(\varphi)$ satisfying the non-self-intersection condition \eqref{no_intersect}. }
\label{fig:free_geom}
\end{figure}

We then define a free-endpoint slender body $\Sigma_\epsilon$ by 
\begin{equation}\label{SB_free}
\Sigma_\epsilon = \big\{\bx\in\R^3 \ts: \ts \bx=\X(\varphi)+\rho\be_\rho(\varphi,\theta), \quad \rho <\epsilon a(\varphi) \big\}.
\end{equation}

Here the radius function $a(\varphi)$ is defined as follows.
\begin{definition}\label{admissible_a}
We say that $a: [-\eta_\epsilon, \eta_\epsilon] \to \R$ is an \emph{admissible slender body radius function} if the following conditions hold: 
\begin{enumerate}
\item (Smoothness) The radius function $a$ is in $\mc{C}^2(-\eta_\epsilon,\eta_\epsilon)$.
\item (Spheroidal endpoints) There exists $0<\delta_a<1$ such that for $\varphi\in (-\eta_\epsilon,-1+\delta_a] \cup [1-\delta_a,\eta_\epsilon)$, we have 
\[ \abs{a(\varphi)-\sqrt{\eta_\epsilon^2-\varphi^2}} \le c_a \epsilon^2\sqrt{\eta_\epsilon^2-\varphi^2}. \]
In particular, $a(\pm \eta_\epsilon)=0$ and we require the following monotonicity in endpoint decay:
\[ a(\varphi_1)\le a(\varphi_2) \quad \text{ if } 1-\delta_a<\abs{\varphi_2}\le\abs{\varphi_1}<\eta_\epsilon. \] 
\item (Non-vanishing away from endpoints) We have $0<a(\varphi) \le 1$ for $\varphi\in (-\eta_\epsilon,\eta_\epsilon)$, and there exists a constant $0<a_0< 1$ such for $\varphi\in (-1+\delta_a,1-\delta_a)$, $a(\varphi)\ge a_0$.  
\item (Controlled derivative) The product $a(\varphi)a'(\varphi)$ is bounded for each $\varphi\in [-\eta_\epsilon,\eta_\epsilon]$. We define the constant
\begin{equation}\label{bar_ca}
\bar c_a := \sup_{\varphi\in (-\eta_\epsilon,\eta_\epsilon)} \abs{a(\varphi)a'(\varphi)}.
\end{equation}
\end{enumerate}
Note that due to the extent $\eta_\epsilon$ of the slender body, the radius function $a(\varphi)$ depends on $\epsilon$. However, to avoid clutter, we will not indicate this dependence in our notation. 
\end{definition}
 
An important example of an admissible radius function $a(\varphi)$ is that of the prolate spheroid with foci at $\pm1$. In this case, $\eta_\epsilon=\sqrt{1+\epsilon^2}$ and the radius function and its derivative are specified for all $\varphi\in(-\eta_{\epsilon},\eta_{\epsilon})$ by 
\begin{align*}
a(\varphi) &= \frac{1}{\sqrt{1+\epsilon^2}}\sqrt{1+\epsilon^2-\varphi^2}, \\
a'(\varphi) &= \frac{-1}{\sqrt{1+\epsilon^2}}\frac{\varphi}{\sqrt{1+\epsilon^2-\varphi^2}}.
\end{align*}
Classical slender body theory has been developed and studied particularly for the case of slender prolate spheroids, or at least spheroidal endpoints \cite{johnson1980improved, chwang1975hydromechanics, kim1987general}. We will often use the prolate spheroid as a motivating example, but note that our methods apply to slightly more general slender body geometries.  \\

We parameterize the slender body surface $\Gamma_\epsilon= \p\Sigma_\epsilon$ as
\begin{equation}\label{Gamma_free}
\Gamma_\epsilon(\varphi,\theta)=\X(\varphi)+\epsilon a(\varphi)\be_\rho(\varphi,\theta).
\end{equation}

The surface element on $\Gamma_\epsilon$ is given by
\begin{equation}\label{surfel_free}
dS = \mc{J}_\epsilon(\varphi, \theta) d\theta d\varphi
\end{equation}
where the Jacobian factor $\mc{J}_\epsilon$ is defined as
\begin{equation}\label{jac_free}
\begin{split}
\mc{J}_\epsilon(\varphi, \theta) &:= \epsilon a(\varphi)\sqrt{(1-\epsilon a(\varphi)\wh\kappa(\varphi,\theta))^2 + \epsilon^2(a'(\varphi))^2} ; \\
\wh\kappa(\varphi,\theta) &:= \kappa_1(\varphi)\cos\theta+\kappa_2(\varphi)\sin\theta.
\end{split}
\end{equation}

As in the closed loop setting, we also define a tubular neighborhood of the slender body centerline ending in spherical caps about the fiber endpoints: 
\begin{equation}\label{mc_O} 
\begin{aligned}
\mathcal{O} &:= \bigg\{\bx \in \Omega_\epsilon \ts : \ts \bx= \X(\varphi)+\rho \be_\rho(\varphi,\theta); \ts \abs{\varphi} \le 1, \ts 0\le \theta<2\pi, \ts \epsilon a(\varphi) < \rho< r_{\max}  \bigg\} \\
&\hspace{2cm} \bigcup \bigg\{\bx \in \Omega_\epsilon \ts : \ts \bx= \X_\text{ext}(\varphi)+\rho \be_\rho(\varphi,\theta); \ts 1<\abs{\varphi} < 1+r_{\max},\\
&\hspace{5cm}  0\le \theta<2\pi, \ts \epsilon a(\varphi) < \rho< \sqrt{r_{\max}^2-(\abs{\varphi}-1)^2}  \bigg\}.
\end{aligned}
\end{equation}
Here $r_{\max}$ is given by \eqref{rmax}, and we define $a(\varphi)\equiv0$ for $\abs{\varphi}\ge \eta_\epsilon$. See Figure \ref{fig:free_geom1} for a depiction of the region $\mc{O}$. Note that the region $\mc{O}$ is $\mc{C}^1$ and, since $r_{\max}$ does not depend on $\epsilon$, the diameter of $\mc{O}$ is bounded independent of $\epsilon$.

\begin{figure}[!h]
\centering
\includegraphics[scale=0.65]{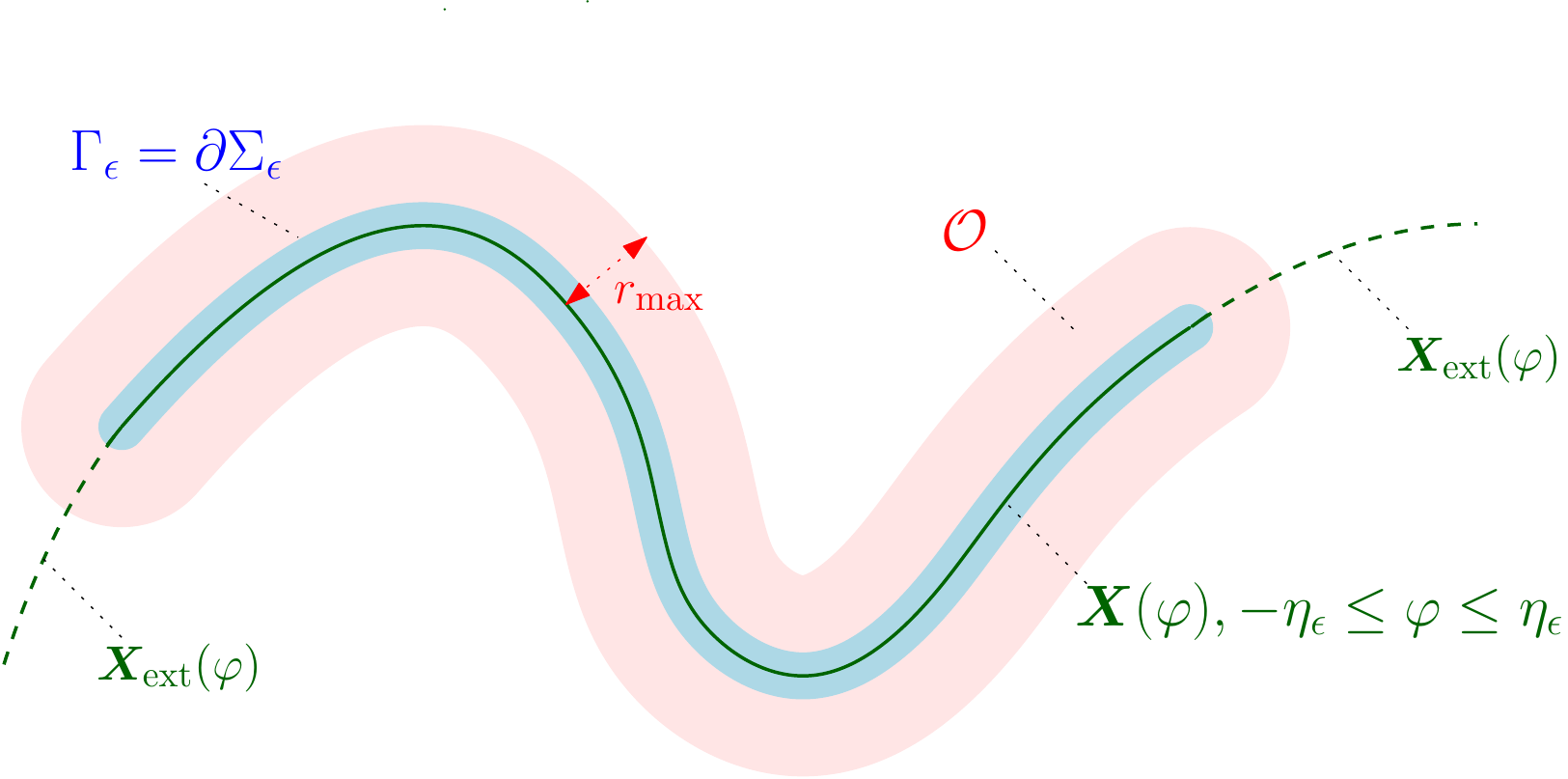} \\
\caption{The neighborhood $\mc{O}$ of the slender body is given by a tubular region about the slender body centerline, with spherical caps at the endpoints. }
\label{fig:free_geom1}
\end{figure}

\subsection{Slender body theory and effective centerline}
The main idea behind slender body theory is to exploit the thin geometry of the filament described in Section \ref{geometry_section} and approximate the fiber as a force density distributed along a one-dimensional curve in three-dimensional space. \\

To build this approximation, we will need the \emph{Stokeslet}: the free-space Green's function for the Stokes equations in $\R^3$. The Stokeslet describes the velocity $\bu$ due to a point source at $\bx_0\in \R^3$ of strength $\bm{f}$, i.e. 
\begin{equation}\label{green_stokes} 
-\Delta\bu+\nabla p=\bm{f}\delta(\bx-\bx_0), \quad \dive\ts \bu =0. 
\end{equation}
It can be shown (see \cite{pozrikidis1992boundary,childress1981mechanics}) that the solution to \eqref{green_stokes} is given by 
\[ \bu = \frac{1}{8\pi} \mc{S}(\bx-\bx_0) \bm{f}, \quad p = \frac{1}{4\pi}\frac{(\bx-\bx_0)\cdot\bm{f}}{\abs{\bx-\bx_0}^3}, \]
where the Stokeslet $\mc{S}$ is defined 
\begin{equation}\label{stokeslet}
\mc{S}(\bx)=\frac{{\bf I}}{\abs{\bx}}+\frac{\bx\bx^{\rm T}}{\abs{\bx}^3}.
\end{equation}

The most basic version of slender body theory, developed by \cite{hancock1953self, batchelor1970slender, cox1970motion}, approximates the force-per-cross-section $\bm{f}(s)$ along the filament by integrating Stokeslets of strength $\bm{f}(s)$ along the the fiber centerline. Later, \cite{keller1976slender, johnson1980improved} improved on this basic model by proposing an additional constraint. To alert the surrounding fluid to the fact that the fiber is in fact a solid, three-dimensional object, \cite{keller1976slender} and \cite{johnson1980improved} require the velocity of the filament to be uniform over each cross section. This ensures the structural integrity of the fiber, as each cross section is constrained to maintain its circular shape over time. In particular, the filament may bend along its centerline, but the radius $\epsilon a(\varphi)$ remains fixed over time at each $\varphi$.  \\


In \cite{closed_loop} we refer to this constant-on-cross-sections constraint as the \emph{fiber integrity condition}. To enforce the fiber integrity condition, we will need to define the \emph{doublet}, given by 
\begin{equation}\label{doublet}
\mc{D}(\bx)= \frac{1}{2}\Delta\mc{S}(\bx) =\frac{{\bf I}}{\abs{\bx}^3}-\frac{3\bx\bx^{\rm T}}{\abs{\bx}^5} .
\end{equation}

On each cross section we then need to add a doublet correction to the Stokeslet term, with a coefficient calculated to cancel the leading order angular dependence in the cross sectional velocity. It can be shown (see \cite{johnson1980improved} and the heuristic in Section 1.2.2 of \cite{closed_loop}) that this coefficient must be proportional to the square of the cross sectional radius $\epsilon^2 a^2$ multiplied by the Stokeslet force $\bm{f}$. \\

However, since the Stokeslet and doublet are singular at $\bx_0$, we will encounter issues in defining, let alone determining the value of, the slender body approximation along $\Gamma_\epsilon$ anywhere the fiber centerline and the actual slender body surface coincide. Clearly this presents a potential problem for the slender body approximation at the fiber endpoints. To address this issue, many authors \cite{johnson1980improved, chwang1975hydromechanics, kim1987general} require that, for a spheroidal slender body, the Stokeslets and doublets are distributed only between the generalized foci of the body (i.e. the foci of the straightened body) at $\varphi=\pm1$. For the (slightly more general) slender fibers that we consider, the notion of generalized focus is still well-defined due to the spheroidal endpoint requirement of Definition \ref{admissible_a}. \\

We define the {\it effective centerline} to be the portion of the fiber centerline $\X$ lying between the generalized foci of the slender body at $\varphi=\pm 1$: 
\begin{equation}\label{effective_centerline}
\text{Effective centerline } = \big\{ \X(s) \ts : \ts -1 \le s \le 1 \big\}.
\end{equation}
Hereafter, we use the variables $s,t \in (-1,1)$ to parameterize the effective centerline. \\

Thus the classical slender body approximation is given by an integral expression over the effective centerline:
\begin{align}\label{SBT_free0}
\bu^{\SB}(\bx) &=\frac{1}{8\pi}\int_{-1}^1 \bigg( \mc{S}(\bm{R})+\frac{\epsilon^2a^2(t)}{2}\mc{D}(\bm{R}) \bigg)\bm{f}(t) \ts dt; \; 
\bm{R}=\bx-\X(t),
\end{align}
where $\mc{S}$ and $\mc{D}$ are as defined in \eqref{stokeslet} and \eqref{doublet}, respectively. The corresponding slender body approximation to the pressure is given by 
 \begin{equation}\label{SBp_free}
p^{\SB}(\bx) = \frac{1}{4\pi}\int_{-1}^1 \frac{\bm{R}\cdot {\bm f}(t)}{|\bm{R}|^3} \ts dt. 
\end{equation}

Importantly, due to the extent of the effective centerline \eqref{effective_centerline}, the slender body approximations \eqref{SBT_free0} and \eqref{SBp_free} are well-defined at every point on the actual slender body surface $\Gamma_\epsilon$. In particular, the expressions \eqref{SBT_free0} and \eqref{SBp_free} do not blow up at the actual fiber endpoints at $\pm \eta_\epsilon$ since the effective centerline does not intersect $\Gamma_\epsilon$. \\

An asymptotic expansion of \eqref{SBT_free0} about $\epsilon=0$ results in an approximation for the motion of the fiber centerline, which we term $\bu^{\SB}_{\rm C}(s)$. Defining the centerline difference
\begin{equation}\label{RC}
\bm{R}_{\rm C}(s,t)=\X(s)-\X(t), \quad s,t\in (-1,1),
\end{equation}
along the effective centerline of the fiber, we have that $\bu^{\SB}_{\rm C}(s)$ is given by
\begin{equation}\label{SBT_asymp_free}
\begin{aligned}
8\pi \ts \bu^{\SB}_{\rm C}(s) &= \big[({\bf I}- 3\be_{\rm t}(s)\be_{\rm t}(s)^{\rm T})+({\bf I}+\be_{\rm t}(s)\be_{\rm t}(s)^{\rm T})L(s)\big]{\bm f}(s) \\
&\qquad + \int_{-1}^1 \left[ \left(\frac{{\bf I}}{|\bm{R}_{\rm C}|}+ \frac{\bm{R}_{\rm C}\bm{R}_{\rm C}^{\rm T}}{|\bm{R}_{\rm C}|^3}\right){\bm f}(t) - \frac{{\bf I}+\be_{\rm t}(s)\be_{\rm t}(s)^{\rm T} }{|s-t|} {\bm f}(s)\right] \ts dt.
\end{aligned}
\end{equation}
Here $L(s) = \log\big( \frac{2(1-s^2)+2\sqrt{(1-s^2)^2+4\epsilon^2a^2(s)}}{\epsilon^2a^2(s)}\big)$. Note that this expression for $L(s)$ agrees with G\"otz \cite{gotz2000interactions} away from the fiber endpoints, but has been adjusted to be well-defined up to the ends (recall that $a(\pm1) =O(\epsilon)$, by Definition \ref{admissible_a}). The expression would also agree with Tornberg-Shelley \cite{tornberg2004simulating} if the radius function $a(s)$ were instead spheroidal with endpoints at exactly $\pm1$ ($a^2(s)=1-s^2$), rather than at $\pm \eta_\epsilon$. However, to ensure that the slender body approximation \eqref{SBT_free0} is well-defined at the fiber endpoints, we use a radius function as defined in Definition \ref{admissible_a}. \\

\begin{figure}[!h]
\centering
\includegraphics[scale=0.6]{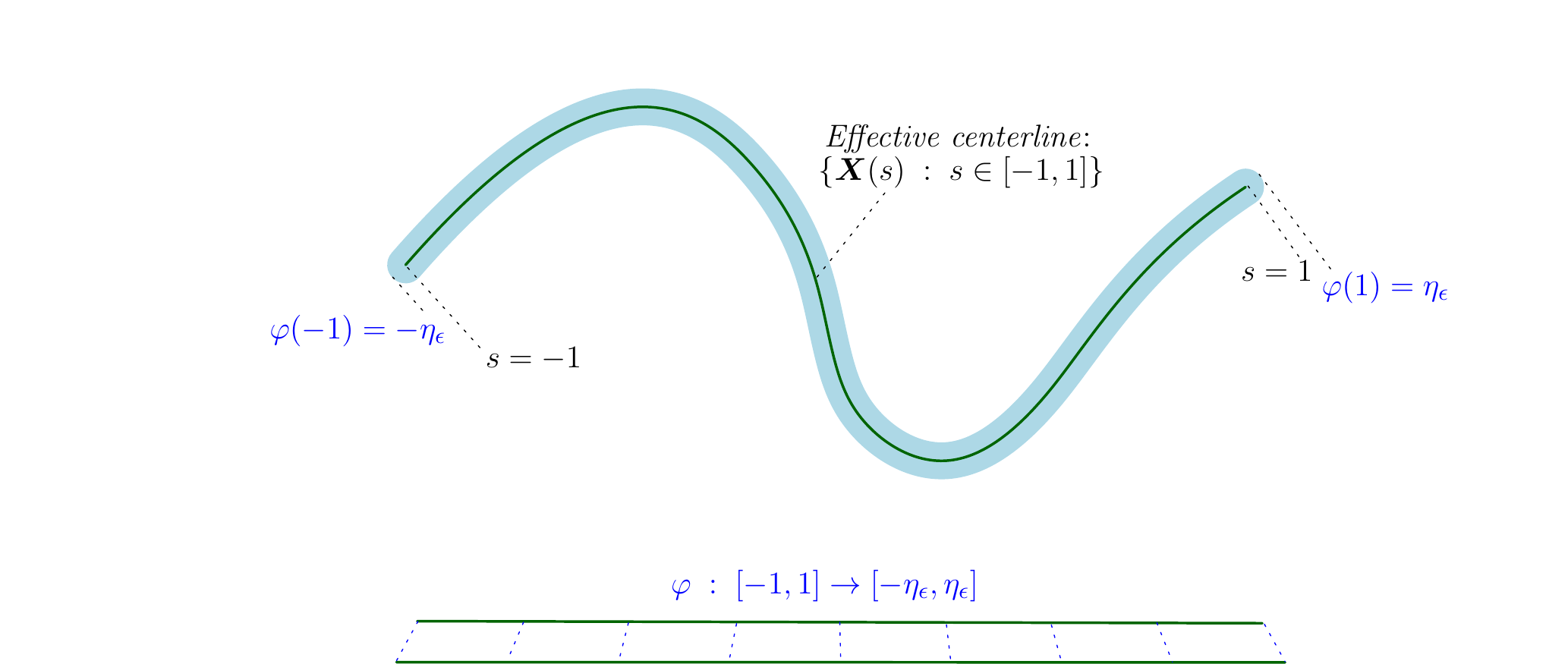} \\
\caption{We think of the arclength parameter $\varphi$ as a function $\varphi(s)$ taking points $s\in[-1,1]$ along the effective centerline to the true centerline, $\varphi(s)\in [-\eta_\epsilon,\eta_\epsilon]$. This allows us to avoid singularities in the slender body approximation \eqref{SBT_free0} at the fiber endpoints. }
\label{fig:free_geom2}
\end{figure}

Now, before we can write down a PDE defining a notion of true solution to the free endpoint slender body problem, we need to be able to reconcile the discrepancy between the slender body force distribution along the effective centerline and the true extent of the actual fiber surface. \\

We instead consider the arc length variable $\varphi$ as a function $\varphi(s)$ that takes $s\in[-1,1]$ to the true centerline, $-\eta_\epsilon \le \varphi(s) \le \eta_\epsilon$. More specifically, we give the following definition of $\varphi(s)$:  
\begin{definition}\label{varphi_def}
We say that $\varphi: [-1,1] \to [-\eta_\epsilon,\eta_\epsilon]$ is an \emph{admissible stretch function} if
\begin{enumerate}
\item (Close to $s$) The function $\varphi\in \mc{C}^1(-1,1)$ is bijective onto $[-\eta_\epsilon,\eta_\epsilon]$ and there exists a constant $c_\varphi>0$ such that $\abs{\varphi(s)-s}\le c_\varphi\epsilon^2$ and $\abs{\varphi'(s)-1}\le c_{\varphi}\epsilon^2$. 
\item (Symmetry) For each $s\in [-1,1]$, $\varphi(-s) = -\varphi(s)$.
\end{enumerate}
\end{definition}

One possibility for the form of $\varphi(\cdot)$ is simply the uniform stretch
\[ \varphi(s) = \eta_\epsilon s; \]
however, it may be useful to be able to specify more complicated functions.

\subsection{Free endpoint slender body PDE}
Using the above definition of effective centerline \eqref{effective_centerline}, we proceed to formulate the free endpoint slender body PDE as a boundary value problem in $\Omega_\epsilon = \R^3\backslash \overline{\Sigma_\epsilon}$. Following a similar outline to the closed loop case \cite{closed_loop}, we show that it is possible to reconstruct a three-dimensional Stokes flow given only a one-dimensional force density $\bm{f}(s)$ and the fiber integrity condition.  \\

Let $\Sigma_\epsilon$ be a slender body with free endpoints as described in Section \ref{geometry_section}, with radius function $a(\varphi)$ as in Definition \ref{admissible_a}, and consider an admissible stretch function $\varphi=\varphi(s)$ as in Definition \ref{varphi_def}. Letting $\bm{\sigma}= \nabla \bu+(\nabla\bu)^{\rm T} + p{\bf I}$ denote the stress tensor, we define the free endpoint slender body PDE as
\begin{equation}\label{free_PDE}
\begin{aligned}
-\Delta \bu + \nabla p &= 0, \ts\ts \dive \ts\bu = 0 \quad \text{in }\Omega_\epsilon = \R^3\backslash\overline{\Sigma_\epsilon} \\
\int_0^{2\pi} ( \bm{\sigma} {\bm n})\big|_{(\varphi(s),\theta)} \ts \mc{J}_\epsilon(\varphi(s),\theta) \varphi'(s) d\theta &= \bm{f}(s) \hspace{1.6cm} \text{ on }\Gamma_\epsilon\\
\bu \big|_{\Gamma_\epsilon} &= \bu\big(\varphi(s)\big) \text{ (unknown but independent of }\theta) \\
\bu\to 0 \ts &\text{ as } |\bx| \to \infty.
\end{aligned}
\end{equation}
Here the Jacobian factor $\mc{J}_\epsilon$ is as defined in \eqref{jac_free}, and the unit normal $\bm{n}$ to $\Gamma_\epsilon$, directed into the slender body, is given by 
\begin{equation}\label{unit_normal} 
{\bm n}(\varphi,\theta) = -\frac{1}{\sqrt{1+\epsilon^2(a')^2}} \be_{\rho}(\varphi,\theta) + \frac{\epsilon a'(\varphi)}{\sqrt{1+\epsilon^2(a')^2}} \be_{\rm t}(\varphi),
\end{equation}
where we have accounted for the variation in the radius function $a(\varphi)$ along the fiber. \\

As in the closed loop setting \cite{closed_loop}, the slender body PDE for free endpoints is a boundary value problem with partial Dirichlet and partial Neumann information specified everywhere along the fiber surface, resulting in a PDE describing semirigid fiber motion. In contrast to the closed loop case, the free-endpoint slender body force data $\bm{f}(s)$, $s\in (-1,1)$, is specified only along the effective centerline \eqref{effective_centerline} but is defined such that the data is stretched to encompass each cross section along the actual fiber surface, $\varphi(s)\in (-\eta_\epsilon,\eta_\epsilon)$. The fiber integrity condition $\bu\big|_{\Gamma_\epsilon}=\bu(\varphi(s))$ is still required to hold along the entire length of the fiber, $-\eta_{\epsilon}< \varphi(s) < \eta_{\epsilon}$. \\

The free endpoint PDE \eqref{free_PDE} gives rise to a natural energy balance that is very similar to the closed loop setting: 
\begin{equation}\label{energy_balance}
\begin{aligned}
\int_{\Omega_\epsilon} 2\abs{\E(\bu)}^2 \ts d\bx &= \int_{\Gamma_\epsilon} \bu(\varphi(s))\cdot( \bm{\sigma}\bm{n})\big|_{(\varphi(s),\theta)} \mc{J}_\epsilon(\varphi(s),\theta) \varphi'(s) \ts d\theta ds \\
&= \int_{-1}^1 \bu(\varphi(s))\cdot \bm{f}(s) \ts ds.
\end{aligned}
\end{equation}
Again, the dissipation per unit time (left hand side) balances the power (force times velocity) exerted by the slender body (right hand side). \\

Using the energy balance \eqref{energy_balance}, we show a well-posedness result for the free endpoint slender body PDE in the function space $D^{1,2}(\Omega_\epsilon)$ (see \eqref{D12_definition}). In order to do so, we will need to impose a decay condition on the force $\bm{f}$ at the fiber endpoints. We define the weighted space $L^2_a(-1,1)$ and corresponding norm $\norm{\cdot}_{L^2_a(-1,1)}$ as
\begin{equation}\label{first_decay_space}
\begin{aligned}
L^2_a(-1,1) &= \bigg\{ \bm{g}\in L^2(-1,1) \ts : \ts \bm{g}(s)\abs{\log(1-s^2)}^{1/2} \in L^2(-1,1)\bigg\}, \\ 
 \norm{\bm{g}}_{L^2_a(-1,1)}^2 &=\int_{-1}^1\abs{\bm{g}(s)}^2 \abs{\log(1-s^2)} \ts ds.
\end{aligned}
\end{equation}

Note that, due to the spheroidal endpoint condition of Definition \ref{admissible_a}, we have 
\begin{equation}\label{first_decay_space2}
 \bm{g}\in L^2_a(-1,1) \text{ if and only if } \bm{g}(s)\abs{\log\big(a(\varphi(s))\big)}^{1/2} \in L^2(-1,1).
 \end{equation}
 Therefore, we will frequently use the condition \eqref{first_decay_space2} as an alternative characterization of the space $L^2_a(-1,1)$, although this characterization implicitly depends on $\epsilon$. \\


Using the definition of $L^2_a(-1,1)$, we show:
\begin{theorem}\emph{(Well-posedness of free-endpoint slender body PDE)}\label{free_end_theorem}
Let $\Omega_\epsilon = \R^3\backslash\overline{\Sigma_\epsilon}$ for a free-endpoint slender body $\Sigma_\epsilon$ as specified in Section \ref{geometry_section}. Given $\bm{f}\in L^2_a(-1,1)$, there exists a unique weak solution $(\bu,p)\in D^{1,2}(\Omega_\epsilon)\times L^2(\Omega_\epsilon)$ to \eqref{free_PDE} satisfying 
\begin{equation}\label{free_bound1}
\norm{\bu}_{D^{1,2}(\Omega_\epsilon)} + \norm{p}_{L^2(\Omega_\epsilon)} \le \abs{\log \epsilon}^{1/2} C\norm{\bm{f}}_{L^2_a(-1,1)},
\end{equation}
where the constant $C$ depends on the constants $c_\Gamma$, $\kappa_{\max}$, $\delta_a$, $c_a$, $\bar c_a$, $a_0$, $c_\varphi$, $c_\eta$, and $c_{\eta,0}$, but not on $\epsilon$. 
\end{theorem}

The free-endpoint PDE formulation \eqref{free_PDE} results in an analogous energy estimate to that in the closed-loop setting: the force data $\bm{f}$ specified along a one-dimensional curve controls the fluid velocity and pressure in three dimensions. However, unlike the closed loop setting, the vanishing radius at the fiber endpoints necessitates a decay requirement \eqref{first_decay_space} in $\bm{f}$. This is due to the fact that the unweighted $L^2$ trace inequality diverges logarithmically at the fiber endpoints, and therefore we must include a small weight in the trace to ensure well-posedness. We explore this phenomenon in detail in Appendix \ref{trace_section} and derive the logarithmic $\epsilon$-dependence appearing in the estimate \eqref{free_bound1}. In addition to the trace inequality, showing the $\epsilon$-dependence in \eqref{free_bound1} will require a new proof of $\epsilon$-independence in the Korn inequality for a thin fiber with free, spheroidal endpoints. This is accomplished by way of an $\epsilon$-independent extension operator into the interior of the fiber, which we construct in Appendix \ref{extension_op}. Given this extension operator, the $\epsilon$-independence of the Korn and Sobolev inequalities follow easily in Appendix \ref{korn_stab}. Finally, in Appendix \ref{press_stab} we verify the $\epsilon$-independence in the pressure inequality \eqref{freeP_boundE}, but this is essentially the same proof as in the closed loop setting (\cite{closed_loop}, Appendix A.2.5). \\

Theorem \ref{free_end_theorem} and the estimate \eqref{free_bound1} provide a framework for estimating the difference between the true solution $(\bu,p)$ to \eqref{free_PDE} and the slender body approximation $(\bu^{\SB},p^{\SB})$ \eqref{SBT_free0}. First, however, we will need to impose a stronger decay condition on the prescribed force $\bm{f}$ at the fiber endpoints. This is due to the manner of obtaining the residual estimates in Section \ref{residuals_free}. To this end, we define the space $\mc{C}_a(-1,1)$, along with the norm $\norm{\cdot}_{\mc{C}_a}$, as follows: 

\begin{equation}\label{decay_condition}
\begin{aligned}
\mc{C}_a(-1,1) &= \bigg\{ \bm{g}\in \mc{C}(-1,1) \ts :\ts \frac{\abs{\bm{g}(s)}}{\sqrt{1-s^2}} < \infty, \ts s\in (-1,1)\bigg\}, \\
\norm{\bm{g}}_{\mc{C}_a(-1,1)} &= \sup_{s\in(-1,1)}\frac{\abs{\bm{g}(s)}}{\sqrt{1-s^2}}.
\end{aligned}
\end{equation}

As in the case of the space $L^2_a(-1,1)$, we note that 
\begin{equation}\label{decay_condition2}
 \bm{g}\in \mc{C}_a(-1,1) \text{ if and only if } \sup_{s\in(-1,1)}\frac{\abs{\bm{g}(s)}}{a(\varphi(s))} <\infty,
 \end{equation}
 due to the spheroidal endpoint condition of Definition \ref{admissible_a}. It will be convenient to use this characterization of the space $\mc{C}_a(-1,1)$ in the residual estimates of Section \ref{residuals_free}, but we again note that this characterization implicitly depends on $\epsilon$. \\
 
We will require the prescribed force $\bm{f}(s)$ to belong to the space $\mc{C}^1(-1,1)\cap \mc{C}_a(-1,1)$, where, as usual, the norm $\norm{\cdot}_{\mc{C}^1(-1,1)}$ is given by
\begin{align*}
\norm{\bm{g}}_{\mc{C}^1(-1,1)} := \displaystyle \sup_{s\in (-1,1)}\abs{\bm{g}(s)}+\displaystyle \sup_{s\in (-1,1)}\abs{\bm{g}'(s)}.
\end{align*}

We show the following error estimate for the slender body approximation \eqref{SBT_free0} for a fiber with free endpoints.

\begin{theorem}\emph{(Free-endpoint slender body theory error estimate)}\label{free_err_theorem}
Let $\Omega_\epsilon=\R^3\backslash\overline{\Sigma_\epsilon}$ for a free-endpoint slender body $\Sigma_\epsilon$ as defined in Section \ref{geometry_section}. Given a force $\bm{f}\in \mc{C}^1(-1,1)\cap \mc{C}_a(-1,1)$, let $(\bu,p)$ denote the true solution to the free endpoint PDE \eqref{free_PDE} and let $(\bu^{\SB},p^{\SB})$ denote the corresponding slender body approximation \eqref{SBT_free0}. Then the difference $\bu^{\SB}-\bu$, $p^{\SB}-p$ satisfies 
\begin{equation}\label{err_est_free}
\norm{\bu^{\SB}-\bu}_{D^{1,2}(\Omega_\epsilon)} + \norm{p^{\SB}-p}_{L^2(\Omega_\epsilon)} \le C \big(\epsilon\abs{\log\epsilon}\norm{\bm{f}}_{\mc{C}^1(-1,1)}+\epsilon\norm{\bm{f}}_{\mc{C}_a(-1,1)} \big) ,
\end{equation}
where $C$ depends on the constants $c_\Gamma$, $\kappa_{\max}$, $\delta_a$, $c_a$, $\bar c_a$, $c_\varphi$, $c_\eta$, and $c_{\eta,0}$, but not on $\epsilon$. \\
In addition, along the effective centerline $(s\in(-1,1))$, the difference between the true fiber velocity ${\rm Tr}(\bu)(s)$ and the slender body centerline approximation $\bu^{\SB}_{\rm C}(s)$ \eqref{SBT_asymp_free} satisfies 
\begin{equation}\label{center_err_free}
\norm{{\rm Tr}(\bu) - \bu^{\SB}_{\rm C}}_{L^2(-1,1)} \le C \big( \epsilon \abs{\log\epsilon}^{3/2} \norm{\bm{f}}_{\mc{C}^1(-1,1)} + \epsilon \abs{\log\epsilon}^{1/2} \norm{\bm{f}}_{\mc{C}_a(-1,1)} \big) ,
\end{equation}
where $C$ again depends on $c_\Gamma$, $\kappa_{\max}$, $\delta_a$, $c_a$, $\bar c_a$, $a_0$, $c_\varphi$, $c_\eta$, and $c_{\eta,0}$, but not on $\epsilon$.
\end{theorem}

Note that the $\norm{\bm{f}}_{\mc{C}^1(-1,1)}$ terms of estimates \eqref{err_est_free} and \eqref{center_err_free} are identical to the closed loop error estimates in \cite{closed_loop}. However, an additional $\norm{\bm{f}}_{\mc{C}_a(-1,1)}$ is needed in the free end case due to effects near the fiber endpoints. As in the closed loop setting, to prove Theorem \ref{free_err_theorem}, we derive residual estimates for the slender body force $\bm{f}^{\SB}-\bm{f}$ and the $\theta$-dependence in the slender body surface velocity $\bu^{\SB}\big|_{\Gamma_\epsilon}(\varphi,\theta)$. The new difficulties in the residual estimation procedure arising from the free ends are detailed in Section \ref{residuals_free}. Most of the technical lemmas involved in the estimation must be altered from their closed-loop analogues to account for the free endpoints and lack of symmetry in integration in arclength along the fiber. In particular, the new technical lemmas for the free endpoint setting will necessitate a stronger decay condition for $\bm{f}$ at the fiber endpoints than the condition \eqref{first_decay_space} needed to simply show well-posedness of the free endpoint PDE. In Section \ref{err_est_sec}, we use the existence framework established in Section \ref{well_posed} to derive the error estimate \eqref{err_est_free} and the centerline estimate \eqref{center_err_free} stated in Theorem \ref{free_err_theorem}. \\

\begin{remark}{\bf Alternative ways to formulate the slender body PDE.}\label{alternatives}
The formulation of the slender body PDE \eqref{free_PDE} is just one of various possible ways to construct a notion of true solution to the slender body problem. We make note of two particularly important alternatives.

\begin{enumerate}
\item {\bf Different force definition: }
Instead of using the stretch function $\varphi(s)$ (Definition \ref{varphi_def}), we may reconcile the discrepancy between the slender body force extent and the actual fiber length in the following way. Given a force $\bm{f}(s)$, $s\in [-1,1]$, we may instead construct a function $\bm{f}_\epsilon(\varphi;\epsilon)$, $\varphi\in [-\eta_\epsilon,\eta_\epsilon]$, satisfying $\abs{\bm{f}_\epsilon(\varphi;\epsilon)-\bm{f}(s)} \le C\epsilon$. We may then replace the total force boundary condition in \eqref{free_PDE} with
\begin{equation}\label{new_force}
\int_0^{2\pi} ( \bm{\sigma} {\bm n})\big|_{(\varphi,\theta)} \ts \mc{J}_\epsilon(\varphi,\theta) d\theta = \bm{f}_\epsilon(\varphi;\epsilon). 
\end{equation}
Note that there is still a degree of ambiguity in this notion of the slender body force, as the choice of $\bm{f}_\epsilon(\varphi;\epsilon)$ is not unique. Indeed, we may interchange these two constructions of the slender body force as follows. Considering $\varphi$ as the stretch function $\varphi(s)$, we multiply both sides of \eqref{new_force} by $\varphi'(s)$ to obtain 
\begin{align*}
\int_0^{2\pi} ( \bm{\sigma} {\bm n})\big|_{(\varphi(s),\theta)} \ts \mc{J}_\epsilon(\varphi(s),\theta) \varphi'(s) d\theta &= \bm{f}_\epsilon(\varphi(s);\epsilon)\varphi'(s) := \bm{f}(s),
\end{align*}
where we have defined $\bm{f}_\epsilon(\varphi(s);\epsilon)\varphi'(s)$ to be the prescribed force data $\bm{f}(s)$.

\item {\bf Different radius function:}
From a modeling perspective, it may be more useful to consider fibers with uniform radius up to the fiber endpoints. For the sake of analysis, we may approximate this scenario via hemispherical caps. We note that a similar error estimate to Theorem \ref{free_err_theorem} should hold for these fibers, too, given a stronger decay condition on the prescribed force at the fiber endpoints.  \\

Consider a free-end slender body with centerline $\X(\varphi)$ as in \eqref{SB_centerline}, but for $\varphi\in [-1-\epsilon,1+\epsilon]$ rather than $\pm \eta_\epsilon$. Let the radius function $a(\varphi)$ be given by 
\begin{equation}\label{rad_a}
\begin{aligned}
a(\varphi) =  \begin{cases}
1, & \varphi\in[-1,1] \\
\frac{1}{\epsilon}\sqrt{\epsilon^2-(\varphi-1)^2}, & 1< \abs{\varphi}\le 1+\epsilon
\end{cases}; \quad a'(\varphi) =  \begin{cases}
0, & \varphi\in[-1,1] \\
\frac{-(\varphi-1)}{\epsilon\sqrt{\epsilon^2-(\varphi-1)^2}}, & 1< \abs{\varphi}\le 1+\epsilon.
\end{cases}
\end{aligned}
\end{equation}
Technically this surface is only $\mc{C}^{1,1}$, but this should not greatly affect the analysis. If needed, the definition \eqref{rad_a} can be accompanied by an appropriate smoothing near $\varphi=\pm 1$ to ensure that the slender body surface is $\mc{C}^2$. We continue to define the effective centerline as in \eqref{effective_centerline}, but instead of the stretch function given by Definition \ref{varphi_def}, we consider $\varphi: [-1,1] \to [-1-\epsilon,1+\epsilon]$ satisfying $\abs{\varphi(s)-s} \le C\epsilon$. \\

The inequalities of Section \ref{stability0} and the subsequent well-posedness arguments of Section \ref{exist_uniq} can be adapted  to this scenario. The same function space $L^2_a(-1,1)$ \eqref{first_decay_space} should work for the force $\bm{f}(s)$ in this scenario. In fact, the proof of existence of an $\epsilon$-independent extension operator (Lemma \ref{extension_free} in Appendix \ref{stability}), used to show the $\epsilon$-independent Korn inequality \eqref{free_korn}, is simpler in this setting due to the uniform radius. The main difference with this choice of $a(\varphi)$ lies in the residual estimates of Section \ref{residuals_free}; in particular, Lemmas \ref{est_free2} and \ref{est_free3}. Obtaining an analogous bound when the radius function $a$ is given by \eqref{rad_a} will require that the prescribed force $\bm{f}(s)$ decays at least like $1-s^2$ at the fiber endpoints. Although this decay requirement is stronger than that needed for the prolate spheroid, this may not be a significant additional limitation on physically relevant choices for $\bm{f}$. With this modification to Lemmas \ref{est_free2} and \ref{est_free3}, the velocity and force residual estimates proceed as in Sections \ref{vel_resid} and \ref{force_resid}. 

\end{enumerate}
\end{remark}

\section{Well-posedness of free-endpoint slender body PDE}\label{well_posed}
This section is devoted to the proof of Theorem \ref{free_end_theorem}. We begin by defining our notion of a weak solution to \eqref{free_PDE}. In Section \ref{stability0}, we note the $\epsilon$-dependence of key inequalities arising in the well-posedness theory for \eqref{free_PDE}, and in Section \ref{exist_uniq}, we use these inequalities to prove the existence and uniqueness claims of Theorem \ref{free_end_theorem} as well as the estimate \eqref{free_bound1}. \\

As in the closed loop setting \cite{closed_loop}, we look for a velocity $\bu$ satisfying \eqref{free_PDE} in a suitable sense with Stokeslet-like decay at $\infty$ ($\sim \abs{\bx}^{-1}$). We therefore consider velocities $\bu$ belonging to the homogeneous Sobolev space $D^{1,2}(\Omega_\epsilon)$, defined as
\begin{equation}\label{D12_definition}
 D^{1,2}(\Omega_\epsilon) = \{ \bu\in L^6(\Omega_\epsilon) \ts : \ts \nabla \bu\in L^2(\Omega_\epsilon) \}.
 \end{equation}
Recalling the Sobolev inequality 
\begin{equation}\label{sobolev_ineq}
\|\bu\|_{L^6(\Omega_\epsilon)} \le c_S\|\nabla \bu\|_{L^2(\Omega_\epsilon)}
\end{equation}
in the exterior domain $\Omega_\epsilon\subset \R^3$, we have that $D^{1,2}(\Omega_\epsilon)$ is a Hilbert space with norm
\begin{equation}\label{D12_norm}
\|\bu\|_{D^{1,2}(\Omega_\epsilon)} \equiv \|\nabla \bu \|_{L^2(\Omega_\epsilon)}.
\end{equation}
For a more detailed exploration of the function space $D^{1,2}$, see \cite{galdi2011introduction}, Chapter II.6 - II.10. 

We define $D^{1,2}_0(\Omega_\epsilon)$ as the closure of $\mc{C}_0^{\infty}(\Omega_\epsilon)$ in $D^{1,2}(\Omega_\epsilon)$.  \\

Following the same outline as in \cite{closed_loop}, we proceed to define a notion of weak solution to \eqref{free_PDE}. First we recall the definition of the admissible set
\begin{equation}\label{Adiv_def}
\A_\epsilon^{\dive} = \{\bu\in D^{1,2}(\Omega_\epsilon) \ts : \ts \dive\ts \bu =0, \ts \bu|_{\Gamma_\epsilon} = \bu(\varphi) \},
\end{equation}
where the boundary condition $\bu|_{\Gamma_\epsilon} = \bu(\varphi)$ is independent of the angle $\theta$ but otherwise unspecified; i.e. $\bu\in \A_\epsilon^{\dive}$ satisfies
\begin{equation}\label{theta_indep}
 \int_{\Gamma_\epsilon} \bu \frac{\p\phi}{\p\theta} dS =0
 \end{equation}
for any $\phi\in \mc{C}^\infty(\Gamma_\epsilon)$. As noted in the closed loop setting, the trace operator $\rm{Tr}$ on $\A_\epsilon^{\dive}$ functions will be considered as both a function on $\Gamma_\epsilon$ and on the centerline $(-\eta_\epsilon,\eta_\epsilon)$. We show in Section \ref{trace_section} that functions $\bu\in \A_\epsilon^{\dive}$ satisfy the free-endpoint weighted trace inequality 
\begin{equation}\label{free_trace}
\begin{aligned}
\frac{1}{\sqrt{2\pi \epsilon(1+\epsilon\kappa_{\max}+\epsilon\bar c_a)}}\norm{{\rm Tr}(\bu) \abs{\log(a)}^{-1/2}}_{L^2(\Gamma_\epsilon)} &\le \norm{{\rm Tr}(\bu) \abs{\log(a)}^{-1/2}}_{L^2(-\eta_\epsilon,\eta_\epsilon)} \\
& \le c_T \norm{\nabla \bu}_{L^2(\Omega_\epsilon)}. 
\end{aligned}
\end{equation}
Here the weight $\abs{\log(a(\varphi))}^{-1/2}$ is needed to avoid logarithmic divergence of the trace near the fiber endpoints. The $\epsilon$-dependence of the constant $c_T$ is explored in detail in Section \ref{trace_section}. \\

Note that $\A_\epsilon^{\dive}$ is nontrivial: taking any constant function defined on $\Gamma_\epsilon$, we can solve the corresponding Stokes Dirichlet boundary value problem in $\Omega_\epsilon$ to obtain a solution that then belongs to $\A_\epsilon^{\dive}$. Furthermore, from \eqref{theta_indep}, we can see that $\A_\epsilon^{\dive}$ is a closed subspace of $D^{1,2}(\Omega_\epsilon)$. \\


With the space $\A_\epsilon^{\dive}$ as above, we define a weak solution to \eqref{free_PDE} as follows:
\begin{definition}{(Weak solution to free-endpoint slender body PDE)} 
A weak solution $\bu\in \A_\epsilon^{\dive}$ to \eqref{free_PDE} satisfies
\begin{equation}\label{free_weak_sto}
\int_{\Omega_\epsilon} 2 \ts \E(\bu):\E(\bv) \ts d\bx - \int_{-1}^1 \bv(\varphi(s))\cdot \bm{f}(s) \ts ds = 0
\end{equation}
for all $\bv \in \A_\epsilon^{\dive}$. 
\end{definition} 

We can formally verify that weak solutions of the slender body Stokes system \eqref{free_PDE} satisfy \eqref{free_weak_sto}. Note that, away from $\Gamma_\epsilon$, the Stokes equations in $\Omega_\epsilon$ can be rewritten as 
\[\dive \ts \bm{\sigma} = 0, \; \dive\ts \bu = 0; \quad \bm{\sigma} = \nabla\bu +(\nabla\bu)^{\rm T} - p{\bf I}. \]
Let $\mc{C}^\infty(\overline \Omega_\epsilon)$ denote smooth functions supported up to the slender body surface $\Gamma_\epsilon$ but away from $\infty$, and assume $\bu\in \A_\epsilon^{\dive}\cap \mc{C}^\infty(\overline \Omega_\epsilon)$ satisfies the slender body Stokes equation \eqref{free_PDE}, so $\dive\ts \bm{\sigma} = 0$ in $\Omega_\epsilon$. Multiplying this equation by $\bv\in \A_\epsilon^{\dive}$ and integrating by parts, we have 
\begin{align*}
0 &= -\int_{\Omega_\epsilon} \dive \ts  \bm{\sigma} \cdot \bv \ts d\bx = \int_{\Omega_\epsilon}  \bm{\sigma} : \nabla \bv \ts d\bx - \int_{\Gamma_\epsilon} \bv\cdot (\bm{\sigma}{\bm n}) \ts dS \\
&= \int_{\Omega_\epsilon} \big(2\ts \E(\bu): \nabla \bv - p\ts \dive\ts \bv\big) \ts d\bx - \int_{-\eta_\epsilon}^{\eta_\epsilon}\int_0^{2\pi} \bv(\varphi) \cdot( \bm{\sigma}{\bm n})\big|_{\varphi,\theta)} \mc{J}_\epsilon(\varphi,\theta) \ts d\theta d\varphi \\
&= \int_{\Omega_\epsilon} 2\ts\E(\bu): \E(\bv) \ts d\bx - \int_{-1}^1 \bv(\varphi(s)) \cdot\int_0^{2\pi}( \bm{\sigma}{\bm n})\big|_{(\varphi(s),\theta)} \mc{J}_\epsilon(\varphi(s),\theta) \varphi'(s) \ts d\theta ds \\
&= \int_{\Omega_\epsilon} 2\ts\E(\bu): \E(\bv) \ts d\bx - \int_{-1}^1 \bv(\varphi(s)) \cdot\bm{f}(s) \ts ds.
\end{align*}

In the following section, we prove the existence of a unique weak solution $\bu\in \A_\epsilon^{\dive}$ to \eqref{free_PDE}. We will also show the existence of a unique pressure $p\in L^2(\Omega_\epsilon)$ corresponding to a weak solution $\bu$ to \eqref{free_PDE}. Defining the space
\begin{equation}\label{Aspace_def}
\A_\epsilon := \{\bv\in D^{1,2}(\Omega_\epsilon) \ts : \ts \bv|_{\Gamma_\epsilon}=\bv(\varphi) \},
\end{equation}
where we have removed the divergence-free restriction, we will show
\begin{lemma}{\emph{(Existence of pressure)}}\label{free_pressure_ex}
Given $\bu\in \A_\epsilon^{\dive}$ satisfying \eqref{free_weak_sto}, there exists a unique corresponding pressure $p\in L^2(\Omega_\epsilon)$ satisfying
\begin{equation}\label{weak_free_p}
 \int_{\Omega_\epsilon}\big( 2 \ts\mathcal{E}(\bu):\mathcal{E}(\bv) - p\ts\dive\ts \bv\big) \ts d\bx - \int_{-1}^1 \bv(\varphi(s))\cdot\bm{f}(s) \ts ds = 0
\end{equation}
for any $\bv\in \A_\epsilon$. 
\end{lemma}

Note that if $(\bu,p)\in (\A_\epsilon^{\dive}\cap \mc{C}_0^{\infty}(\overline \Omega_\epsilon)) \times \mc{C}_0^{\infty}(\overline \Omega_\epsilon)$ satisfies \eqref{weak_free_p}, then, integrating by parts,
\begin{align*}
 0 &= -\int_{\Omega_{\epsilon}}\left(2\ts\dive(\E(\bu))\cdot \bv - \nabla p\cdot \bv\right) \ts d\bx + \int_{\Gamma_{\epsilon}}\left(2\ts \E(\bu){\bm n} - \nabla p\ts{\bm n}\right)\cdot \bv \ts dS  - \int_{-1}^1 \bv(\varphi(s))\cdot{\bm f}(s) \ts ds \\
 &= -\int_{\Omega_{\epsilon}} (\Delta\bu-\nabla p) \cdot\bv \ts d\bx + \int_{-1}^1\int_0^{2\pi}(\bm{\sigma}{\bm n})\cdot \bv(\varphi) \ts \mathcal{J}_{\epsilon}(\varphi,\theta)\ts d\theta d\varphi  - \int_{-1}^1 \bv(\varphi(s))\cdot{\bm f}(s) \ts ds \\
 &= \int_{\Omega_{\epsilon}} (-\Delta\bu+\nabla p) \cdot\bv \ts d\bx + \int_{-1}^1\bv(\varphi(s))\cdot\bigg(\int_0^{2\pi}(\bm{\sigma}{\bm n}) \ts \mathcal{J}_{\epsilon}(\varphi(s),\theta) \varphi'(s) \ts d\theta - {\bm f}(s)\bigg) ds.
\end{align*}
Thus $(\bu,p)$ actually satisfies \eqref{free_PDE} pointwise almost everywhere, and any smooth enough $(\bu,p)$ satisfying \eqref{weak_free_p} is in fact a classical solution to the free-endpoint slender body PDE.

\subsection{Key inequalities and $\epsilon$-dependence}\label{stability0}
Using the notion \eqref{free_weak_sto} of weak solution to \eqref{free_PDE}, we can easily obtain existence and uniqueness results for the free endpoint slender body PDE. The main difficulty in completing the proof of Theorem \ref{free_end_theorem} is in determining the $\epsilon$-dependence of each of the constants arising in the well-posedness theory. As we ultimately hope to use the energy framework to derive an error estimate for the slender body approximation in terms of $\epsilon$, we must characterize how each such constant depends on $\epsilon$. In this section we track the $\epsilon$-dependence of the constants in the Sobolev, trace, Korn, and pressure inequalities, each of which will be used in proving Theorem \ref{free_end_theorem}. The proofs of each inequality can be found in Appendix \ref{stability}. \\

%
%


We begin by noting that the Sobolev inequality \eqref{sobolev_ineq} holds in $\Omega_\epsilon$ with a constant independent of $\epsilon$.
\begin{lemma}\emph{(Free endpoint Sobolev inequality)}\label{sob_free}
Let $\Omega_\epsilon$ be as in Section \ref{geometry_section}. Then the constant $c_S$ arising in the Sobolev inequality
\[ \norm{\bu}_{L^6(\Omega_\epsilon)} \le c_S\norm{\nabla\bu}_{L^2(\Omega_\epsilon)} \]
depends only on $c_\Gamma$, $\kappa_{\max}$, $c_a$, $\bar c_a$, $a_0$, $\delta_a$, $c_\eta$, and $c_{\eta,0}$, but is independent of $\epsilon$.
\end{lemma}
The proof of Lemma \ref{sob_free} appears in Appendix \ref{korn_stab}. \\

Next, we will need to show a trace inequality \eqref{free_trace} for $\A_\epsilon$ functions \eqref{Aspace_def}. As in the closed loop setting \cite{closed_loop}, because of $\theta$-independence along $\Gamma_\epsilon$, we are essentially seeking a codimension 2 trace inequality for $D^{1,2}$ functions in $\R^3$. In contrast to the closed loop setting, however, we encounter issues in the $L^2$ trace diverging logarithmically toward the fiber endpoints, where the fiber becomes a truly codimension 2 object. This will necessitate the inclusion of a small weight in the trace inequality at the fiber endpoints. 

\begin{lemma}\emph{(Free endpoint trace inequality)}\label{free_trace_lemma}
Let $\Omega_\epsilon=\R^3\backslash \overline{\Sigma_\epsilon}$ for $\Sigma_\epsilon$ as in Section \ref{geometry_section} with radius $a=a(\varphi)$ given by Definition \ref{admissible_a}. For $\bu\in \A_\epsilon$ \eqref{Aspace_def}, the following one-dimensional trace inequality holds:
\begin{equation}\label{free_trace_ineq}
\norm{{\rm Tr}(\bu) \abs{\log(a)}^{-1/2}}_{L^2(-\eta_\epsilon,\eta_\epsilon)} \le \abs{\log\epsilon}^{1/2}C \norm{\nabla \bu}_{L^2(\Omega_\epsilon)},
\end{equation}
where the constant $C$ depends only on $c_\Gamma$, and $\kappa_{\max}$. 
\end{lemma}
The proof of Lemma \ref{free_trace_lemma} appears in Appendix \ref{trace_section}.\\

In addition, the proof of Theorem \ref{free_end_theorem} will require a Korn inequality bounding the gradient $\nabla\bu$ in $\Omega_\epsilon$ by the symmetric gradient $\E(\bu)$.  
\begin{lemma}\emph{(Free endpoint Korn inequality)}\label{korn_free}
Let $\Omega_\epsilon$ be as in Section \ref{geometry_section}. Then the constant $c_K$ arising in the Korn inequality 
\begin{equation}\label{free_korn}
 \norm{\nabla \bu}_{L^2(\Omega_\epsilon)} \le c_K\norm{\E(\bu)}_{L^2(\Omega_\epsilon)} 
 \end{equation}
depends on the constants $c_{\Gamma}$, $\kappa_{\max}$, $c_a$, $\bar c_a$, $a_0$, $\delta_a$, $c_\eta$, and $c_{\eta,0}$, but not on $\epsilon$.  
\end{lemma}
The proof of Lemma \ref{korn_free} will rely on the existence of an operator extending $\bu\in D^{1,2}(\Omega_\epsilon)$ into the interior of the slender body $\Sigma_\epsilon$ with symmetric gradient bounded independent of $\epsilon$. The existence of such an extension operator for slender bodies with spheroidal endpoints is proved in Appendix \ref{extension_op}. Lemma \ref{korn_free} is then shown in Appendix \ref{korn_stab}. \\

Finally, we will need the following lemma bounding the pressure within the fluid by the symmetric gradient of the fluid velocity. We show that the constant arising in this bound is independent of $\epsilon$. 

\begin{lemma}\emph{(Free endpoint pressure inequality)}\label{press_est_free}
Let $\Omega_\epsilon$ be as in Section \ref{geometry_section}. Then the constant $c_P$ arising in both the estimate for the solution $\bv\in D^{1,2}_0(\Omega_\epsilon)$ to 
\begin{equation}\label{p_ineq1}
 \dive\ts\bv = p \ts \text{ in }\Omega_\epsilon, \quad \norm{\bv}_{D^{1,2}(\Omega_\epsilon)}\le c_P\norm{p}_{L^2(\Omega_\epsilon)}, 
 \end{equation}
and the pressure inequality
\begin{equation}\label{p_ineq2}
 \norm{p}_{L^2(\Omega_\epsilon)} \le c_P\norm{\E(\bu)}_{L^2(\Omega_\epsilon)} 
 \end{equation}
depends on $c_\Gamma$ and $\kappa_{\max}$ but is independent of $\epsilon$.
\end{lemma}
The proof of Lemma \ref{press_est_free} is nearly identical to the closed loop setting \cite{closed_loop}, but the free endpoint argument is detailed in Appendix \ref{press_stab} to highlight some differences arising from the fiber endpoint geometry. 


\subsection{Existence, uniqueness, and energy estimate}\label{exist_uniq}
Using the inequalities of Section \ref{stability0}, the proof of Theorem \ref{free_end_theorem} proceeds exactly as in the closed loop setting \cite{closed_loop}. 

\begin{proof}[Proof of Theorem \ref{free_end_theorem}:] 
Given $\bm{f}\in L^2_a(-1,1)$, where the decay space $L^2_a(-1,1)$ is as defined in \eqref{first_decay_space}, we show existence of a unique weak solution $\bu\in \A_\epsilon^{\dive}$ to \eqref{free_weak_sto} by considering the bilinear form
\[ \mc{B}[\bu,\bv] := \int_{\Omega_\epsilon} 2 \ts \E(\bu):\E(\bv) \ts d\bx. \]
By Korn's inequality \eqref{korn_free}, we have that for any $\bu\in \A_\epsilon^{\dive}$,
\begin{align*}
\mc{B}[\bu,\bu] &= \int_{\Omega_\epsilon} 2 \ts |\E(\bu)|^2 \ts d\bx \ge \int_{\Omega_\epsilon} \frac{2}{c_K^2} |\nabla\bu|^2 \ts d\bx,
\end{align*}
so $\mc{B}[\cdot,\cdot]$ is coercive on $\A_\epsilon^{\dive}$. Additionally, $\mc{B}[\cdot,\cdot]$ is bounded: for $\bu$, $\bv$ in $\A_\epsilon^{\dive}$, we have
\begin{align*}
\abs{\mc{B}(\bu,\bv)} &\le \int_{\Omega_\epsilon} 2 |\E(\bu)| |\E(\bv)| \ts d\bx  \le 2 \|\E(\bu)\|_{L^2(\Omega_\epsilon)}\|\E(\bv)\|_{L^2(\Omega_\epsilon)} \le 2 \|\nabla\bu\|_{L^2(\Omega_\epsilon)}\|\nabla\bv\|_{L^2(\Omega_\epsilon)}.
\end{align*}

Finally, we show that for $\bm{f}\in L^2_a(-1,1)$ and $\bv\in \A_\epsilon^{\dive}$, the linear functional
\[ L(\bm{f}) := \int_{-1}^1 \bm{f}(s)\cdot \bv(\varphi(s)) \ts ds \]
is bounded. By Cauchy-Schwarz and the trace inequality \eqref{free_trace} on $\A_\epsilon^{\dive}$, we have
\begin{align*}
\int_{-1}^1 \bm{f}(s)\cdot\bv(\varphi(s)) \ts ds &\le \|\bm{f}\|_{L^2_a(-1,1)}\|\bv \abs{\log(a)}^{-1/2}\|_{L^2(-\eta_\epsilon,\eta_\epsilon)} \le c_T \|\nabla \bv\|_{L^2(\Omega_\epsilon)}\|\bm{f}\|_{L^2_a(-1,1)}.
\end{align*}

Thus the form $\mc{B}[\cdot,\cdot]$ is bounded and coercive on $\A_\epsilon^{\dive}$ and the functional $L(\cdot)$ is bounded on $\A_\epsilon^{\dive}$, and by the Lax-Milgram theorem there exists a unique solution $\bu\in \A_\epsilon^{\dive}$ to \eqref{free_weak_sto}. Furthermore, taking $\bv=\bu$ in \eqref{free_weak_sto} and using the Korn inequality \eqref{free_korn} and the trace inequality \eqref{free_trace}, we have 
\begin{align*}
\|\nabla \bu\|_{L^2(\Omega_\epsilon)}^2 &\le c_K^2 \|\E(\bu)\|_{L^2(\Omega_\epsilon)}^2 \le \frac{c_K^2}{2}\norm{\bm{f}}_{L^2_a(-1,1)} \norm{\bu \abs{\log(a)}^{-1/2}}_{L^2(-\eta_\epsilon,\eta_\epsilon)} \\
&\le \frac{c_K^2}{2}\bigg(\frac{1}{4 \delta}\norm{\bm{f}}_{L^2_a(-1,1)}^2+ \delta \norm{\bu \abs{\log(a)}^{-1/2}}_{L^2(-\eta_\epsilon,\eta_\epsilon)}^2\bigg) \\
& \le \frac{c_K^2}{2}\bigg(\frac{1}{4 \delta} \|\bm{f}\|_{L^2_a(-1,1)}^2+  \delta c_T^2\|\nabla \bu\|_{L^2(\Omega_\epsilon)}^2 \bigg)
\end{align*}
for $ \delta\in \R_+$. Taking $ \delta=\frac{1}{c_T^2c_K^2}$, we obtain
\begin{equation}\label{free_stokes_est}
\|\nabla \bu\|_{L^2(\Omega_\epsilon)} \le \frac{1}{2}c_K^2c_T\|\bm{f}\|_{L^2_a(-1,1)}.
\end{equation}

As in the closed-loop setting, we may use the existence of a unique velocity $\bu\in D^{1,2}(\Omega_\epsilon)$ satisfying \eqref{free_weak_sto} to show Lemma \ref{free_pressure_ex}, the existence of a unique pressure $p\in L^2(\Omega_\epsilon)$. The proof is identical to that in \cite{closed_loop}, Section 2.2, and we do not repeat it here. \\

Using Lemma \ref{press_est_free}, we may consider $\bv\in D^{1,2}_0(\Omega_\epsilon)$ satisfying
\begin{equation}
\begin{aligned}
\dive \ts\bv &=p \quad \text{ in }\Omega_\epsilon \\
\norm{\bm{v}}_{D^{1,2}(\Omega_\epsilon)} &\le c_P\norm{p}_{L^2(\Omega_\epsilon)}
\end{aligned}
\end{equation} 
for $c_P$ independent of $\epsilon$. Substituting this $\bv$ into \eqref{weak_free_p}, we then have
\begin{align*}
\int_{\Omega_{\epsilon}} p^2 \ts d\bx &= \int_{\Omega_{\epsilon}} 2 \ts\E(\bu):\E(\bv) \ts d\bx \le 2\|\E(\bu)\|_{L^2(\Omega_{\epsilon})}\|\E(\bv)\|_{L^2(\Omega_{\epsilon})} \le 2\|\E(\bu)\|_{L^2(\Omega_{\epsilon})}\|\nabla\bv\|_{L^2(\Omega_{\epsilon})} \\
& \le \frac{1}{ \delta}\|\E(\bu)\|_{L^2(\Omega_{\epsilon})}^2+ \delta\|\nabla\bv\|_{L^2(\Omega_{\epsilon})}^2 \le \frac{1}{ \delta}\|\E(\bu)\|_{L^2(\Omega_{\epsilon})}^2+ \delta c_P^2\|p\|_{L^2(\Omega_{\epsilon})}^2, \quad  \delta>0. 
\end{align*}
Taking $ \delta = \frac{1}{2c_P^2}$, we obtain
\begin{equation}\label{freeP_boundE}
\norm{p}_{L^2(\Omega_\epsilon)} \le 2c_P\norm{\E(\bu)}_{L^2(\Omega_\epsilon)}. 
\end{equation} 

Combining \eqref{free_stokes_est} and \eqref{freeP_boundE} we obtain
\begin{equation}\label{free_est_withC}
\norm{\bu}_{D^{1,2}(\Omega_\epsilon)} + \norm{p}_{L^2(\Omega_\epsilon)} \le \frac{1}{2}c_K^2c_T(1+2c_P)\|\bm{f}\|_{L^2_a(-1,1)} \le C\abs{\log\epsilon}^{1/2}\|\bm{f}\|_{L^2_a(-1,1)},
\end{equation}
where we have used the $\epsilon$-dependence of $c_K$, $c_T$, and $c_P$ given by Lemmas \ref{free_trace_lemma} - \ref{press_est_free}.
\end{proof}

\section{Free endpoint slender body residuals}\label{residuals_free}
In this section, we use the slender body approximation \eqref{SBT_free0} to estimate the $\theta$-dependence in the slender body velocity $\bu^{\SB}\big|_{\Gamma_\epsilon}(\varphi,\theta)$ and the residual between the prescribed force $\bm{f}$ and the slender body approximation $\bm{f}^{\SB}$. We follow a similar procedure to the closed loop setting (see \cite{closed_loop}, Section 3). However, in the free endpoint setting, the integrals arising from \eqref{SBT_free0} are no longer periodic in $s$ and therefore we no longer get to take advantage of the same cancellations due to integration symmetry that we used in the closed loop case. We instead prove a series of new technical lemmas to use for estimating the free endpoint integrals. These new technical lemmas will require us to impose a stronger decay condition than \eqref{first_decay_space} on the prescribed force $\bm{f}$ at the fiber endpoints. In particular, we will require $\bm{f}$ to belong to the decay space $\mc{C}^1(-1,1)\cap \mc{C}_a(-1,1)$, defined in \eqref{decay_condition}. \\

In Section \ref{lemmas}, we prove a series of general integral estimates (see Lemmas \ref{est_free1} - \ref{center_lem_free}) that will facilitate estimation of the residuals between the slender body approximation \eqref{SBT_free0} and the true solution to \eqref{free_PDE}. In Section \ref{vel_resid} we rely on the general lemmas in Section \ref{lemmas} to derive a $\mc{C}^0$ bound on the $\theta$-dependence in the slender body velocity approximation. Additionally, we use Lemma \ref{center_lem_free} to bound the residual between the slender body approximation on the fiber surface $\bu^{\SB}\big|_{\Gamma_\epsilon}$ and the centerline approximation $\bu^{\SB}_{\rm C}(s)$, given by \eqref{SBT_asymp_free}. Then in Section \ref{force_resid} we prove bounds on the slender body force residual, relying on the lemmas in Section \ref{lemmas} as well as an additional estimate (Lemma \ref{est_free4}) taking advantage of $\theta$-integration.

\subsection{Free endpoint integral estimates}\label{lemmas}
Recall that within the neighborhood $\mc{O}$ of the slender body surface $\Gamma_\epsilon$, we may parameterize points as
\[ \bx(\rho,\theta,s) = \X(\varphi(s))+ \rho\be_\rho(\varphi(s),\theta) \]
and write $\bm{R}$ as
\begin{equation}\label{R_free}
\bm{R}(\varphi(s),\theta,\rho;t) = \X(\varphi(s)) - \X(t)+ \rho \be_\rho(\varphi(s),\theta).
\end{equation}

As in the closed loop setting, the following derivative identities hold, due to the Bishop frame \eqref{bishop_ODE}: 
\begin{equation}\label{deriv_IDs}
\frac{\p \bm{R}}{\p \rho}=\be_\rho(\varphi,\theta), \quad \frac{1}{\rho}\frac{\p \bm{R}}{\p \theta} =\be_\theta(\varphi,\theta), \quad
\frac{1}{1-\rho\widehat{\kappa}}\frac{\p \bm{R}}{\p \varphi} =\be_{\rm t}(\varphi),
\end{equation}
where $\wh\kappa$ is given by \eqref{jac_free}. Note that in the closed loop framework, the analogous identity to $\varphi$-derivative of \eqref{deriv_IDs} includes an additional term $\kappa_3 \p \bm{R}/\p\theta$, which arises due to the fact that the moving frame used in the closed loop setting must be periodic, and therefore there is an additional non-zero coefficient $\kappa_3$ in the moving frame ODE \eqref{bishop_ODE} connecting $\be_{n_1}(\varphi)$ and $\be_{n_2}(\varphi)$. \\

Now, for $\X\in \mc{C}^2(-\eta_\epsilon,\eta_\epsilon)$, we can write
\begin{equation}\label{X_expand}
\X(\varphi(s)) - \X(t) = (\varphi(s)-t)\be_{\rm t}(\varphi(s)) + (\varphi(s)-t)^2\bm{Q}(s,t); \quad \abs{\bm{Q}(s,t)}\le \frac{\kappa_{\max}}{2},
\end{equation}
for $\varphi(s)\in (-\eta_\epsilon,\eta_\epsilon)$ and $t\in (-1,1)$. Then, along the slender body surface $\Gamma_\epsilon$, we have
\begin{equation}\label{R_expand}
\bm{R} = \bars\be_{\rm t}(\varphi(s))+ \epsilon a(\varphi(s))\be_\rho(\varphi(s),\theta) + \bars^2\bm{Q}(\bars,s); \quad \bars := \varphi(s) - t.
\end{equation}
As in the closed loop setting, we hereafter consider $\bm{R}$ as a function of $\bars$ and $s$ rather than $t$ and $s$. Using \eqref{R_expand}, we note the following bounds on $\bm{R}$, which, aside from replacing the uniform radius $\epsilon$ with the varying $\epsilon a(\varphi(s))$, are identical to Lemma 3.1 in the closed loop setting \cite{closed_loop}.

\begin{lemma}\label{R_ests_free}
For $a=a(\varphi(s))$ as in Definition \ref{admissible_a} and $\bm{R}$ as in \eqref{R_expand}, we have
\begin{align}
\label{RQ_free}
\abs{\abs{\bm{R}}-\sqrt{\bars^2+\epsilon^2a^2}}&\le \frac{\kappa_{\max}}{2} \bars^2,\\
\label{Rlb_free}
\abs{\bm{R}}&\ge C\sqrt{\bars^2+\epsilon^2a^2},
\end{align}
where $\abs{\bars}\le 1+\eta_\epsilon$ and $C$ depends only on $\kappa_{\max}$ and $c_\Gamma$. 
\end{lemma}


Similarly, working with respect to $s$ and $\bars$, we can expand the $a^2(\varphi(s)-\bars)$ factor appearing in the doublet coefficient of \eqref{SBT_free0}. Since $a\in \mc{C}^2$, by Definition \ref{admissible_a}, for $\bars \in (\varphi(s)-1,\varphi(s)+1)$ we have
\begin{equation}\label{a_expand}
a^2(\varphi(s)-\bars) = a^2(\varphi(s)) + \bars A(\varphi(s),\bars); \quad \abs{A(\varphi(s),\bars)} \le 2\bar c_a,
\end{equation}
where $\bar c_a$ is as in \eqref{bar_ca}. \\

Then for $\bx$ within the region $\mc{O}$ \eqref{mc_O} (and in particular at the slender body surface $\Gamma_\epsilon$), the slender body approximation \eqref{SBT_free0} can be written in terms of $\bars$ and $s$ as 
\begin{equation}\label{SBT_free}
\bu^{\SB}(\bx) =\frac{1}{8\pi}\int_{\varphi(s)-1}^{\varphi(s)+1} \bigg( \mc{S}(\bm{R})+\frac{\epsilon^2a^2(\varphi(s))}{2}\mc{D}(\bm{R})+ \frac{\epsilon^2\bars A(\varphi(s),\bars)}{2}\mc{D}(\bm{R}) \bigg)\bm{f}(\varphi(s)-\bars) \ts d\bars 
\end{equation}
for $\bm{R}$ as in \eqref{R_expand} and for $\mc{S}$ and $\mc{D}$ as in \eqref{stokeslet} and \eqref{doublet}, respectively. The corresponding slender body pressure can be written
\begin{equation}\label{SBp_free}
p^{\SB}(\bx) = \frac{1}{4\pi}\int_{\varphi(s)-1}^{\varphi(s)+1} \frac{\bm{R}\cdot {\bm f}(\varphi(s)-\bars)}{|\bm{R}|^3} \ts d\bars. 
\end{equation}

From now on, we will use $a$ without an explicit argument to refer to $a=a(\varphi(s))$, the slender body radius function at cross section $\varphi(s)$. \\

We recall a simple integral bound (Lemma 3.2 in the closed loop setting \cite{closed_loop}) that continues to hold for free ends.    
\begin{lemma}\label{integral_est}
Let $m\ge 0$ and $n>0$ be integers, and let $a=a(\varphi(s))$ be given by Definition \eqref{admissible_a}. Then, for $s\in (-1,1)$, we have 
\begin{equation}\label{int_est_eqn}
\int_{\varphi(s)-1}^{\varphi(s)+1} \frac{\abs{\bars}^m}{(\bars^2+\epsilon^2 a^2)^{n/2}}d\bars \le \begin{cases}
4\abs{\log (\epsilon a) }, & n=m+1 \\
\pi (\epsilon a)^{m+1-n} , & n\ge m+2.
\end{cases}
\end{equation}
\end{lemma}

%

From Lemmas \ref{R_ests_free} and \ref{integral_est}, we easily obtain the free endpoint version of Lemma 3.3 in \cite{closed_loop}. 
\begin{lemma}\label{est_free1}
Let $a=a(\varphi(s))$ be as in Definition \ref{admissible_a} and let $\bm{R}$ be as in \eqref{R_expand}. Consider $s\in (-1,1)$. For integers $m\ge 0$ and $n>0$, and for $\epsilon$ sufficiently small, we have
\begin{equation}
\int_{\varphi(s)-1}^{\varphi(s)+1} \frac{\abs{\bars}^m}{\abs{\bm{R}}^n} d\bars \le \begin{cases}
C \abs{\log (\epsilon a)}, & n=m+1 \\
C (\epsilon a)^{m+1-n}, & n\ge m+2,
\end{cases}
\end{equation}
where the constants $C$ depend only on $n$, $c_\Gamma$, and $\kappa_{\max}$.
\end{lemma}

It will also be useful to prove an alternate version of Lemma \ref{est_free1} where we rely on the extent of the effective centerline \eqref{effective_centerline} and the smallness of $a(\varphi)$ at the fiber endpoints to trade a factor of $\epsilon$ for a factor of $a$. We show: 
\begin{lemma}\label{est_free1_new}
Let $a=a(\varphi(s))$ be as in Definition \ref{admissible_a} and let $\bm{R}$ be as in \eqref{R_expand}. Consider $s\in (-1,1)$. For integers $m\ge 0$ and $n\ge m+1$, and for $\epsilon$ sufficiently small, we have
\begin{equation}\label{eps_for_a}
\int_{\varphi(s)-1}^{\varphi(s)+1} \frac{\abs{\bars}^m}{\abs{\bm{R}}^n}\ts d\bars \le \begin{cases}
C\epsilon^{m-n} a^{m+2-n}, & n\ge m+2 \\
C\abs{\log\epsilon}, & n=m+1
\end{cases}
\end{equation}
where the constant $C$ depends on $n$, $c_\Gamma$, $\kappa_{\max}$, $c_a$, $c_\eta$, and $c_{\eta,0}$, but not on $\epsilon$.
\end{lemma}

\begin{proof}
We show \eqref{eps_for_a} first for cross sections $\varphi(s)$ away from the very endpoints of the fiber. For $\abs{\varphi(s)} \le \eta_\epsilon - c_{\eta,0}\frac{\epsilon^2}{2}$, using \eqref{eta_epsilon} and Definition \ref{admissible_a}, we have that
\[ a(\varphi(s)) \ge C\epsilon, \]
where $C$ depends on $c_a$, $c_{\eta,0}$ and $c_\eta$. Then, by Lemma \ref{est_free1}, we have
\begin{equation}\label{middle_est} 
\int_{\varphi(s)-1}^{\varphi(s)+1} \frac{\abs{\bars}^m}{\abs{\bm{R}}^n}\ts d\bars \le \begin{cases}
C \epsilon^{m-n} a^{m+2-n}, & n\ge m+2 \\
C\abs{\log\epsilon}, & n=m+1,
\end{cases}
\end{equation}
where $C$ depends on $n$, $c_\Gamma$, $\kappa_{\max}$, $c_a$, $c_{\eta,0}$, and $c_\eta$. \\

We now show \eqref{eps_for_a} holds for cross sections $\eta_\epsilon - c_{\eta,0}\frac{\epsilon^2}{2}< \varphi(s)< \eta_\epsilon$, up to the fiber endpoint. The bound for $-\eta_\epsilon <\varphi(s)<-\eta_\epsilon + c_{\eta,0}\frac{\epsilon^2}{2}$ then follows by symmetry of $\varphi(s)$. \\

For $\eta_\epsilon - c_{\eta,0}\frac{\epsilon^2}{2}< \varphi(s)< \eta_\epsilon$, we have 
\[ \varphi(s)-1 \ge \eta_\epsilon - 1 - c_{\eta,0}\frac{\epsilon^2}{2} \ge \frac{c_{\eta,0}}{2}\epsilon^2, \]
using \eqref{eta_epsilon}. Then 
\begin{align*}
\int_{\varphi(s)-1}^{\varphi(s)+1} \frac{\abs{\bars}^m}{\abs{\bm{R}}^n}\ts d\bars &\le C\int_{\varphi(s)-1}^{\varphi(s)+1} \frac{\abs{\bars}^m}{(\bars^2+\epsilon^2a^2)^{n/2}}\ts d\bars \le C\int_{\varphi(s)-1}^{\varphi(s)+1} \abs{\bars}^{m-n} \ts d\bars \\
&\le \begin{cases}
C\epsilon^{2(m+1-n)}, & n\ge m+2 \\
C\abs{\log\epsilon}, & n=m+1
\end{cases}
 \le \begin{cases}
 C\epsilon^{m-n} a^{m+2-n} , & n\ge m+2 \\
C\abs{\log\epsilon}, & n=m+1,
\end{cases}
\end{align*}
where in the last inequality we used that, by Definition \ref{admissible_a}, we have
\[ a(\varphi(s)) \le C\epsilon. \]

Combining the endpoint estimates with the estimate \eqref{middle_est}, we obtain \eqref{eps_for_a}.
\end{proof}


Before we state the free endpoint analogues of Lemmas 3.4 and 3.5 in the case of the closed loop \cite{closed_loop}, we will need an auxiliary lemma making use of the decay condition \eqref{decay_condition} on the prescribed force $\bm{f}$. As in the closed loop setting, we want to be able to rely on direct integration whenever possible to obtain more refined integral bounds than Lemma \ref{est_free1}. However, for the closed loop, we could rely on symmetry in integrating in $\bars$ along the fiber centerline to achieve additional cancellation. In the case of free endpoints, since integration in $\bars$ is no longer symmetric, we must instead rely on the decay condition \eqref{decay_condition} for $\bm{f}$ at the fiber endpoints in order to obtain similarly improved integral estimates. Toward this end, we show the following auxiliary estimate for functions satisfying the decay condition \eqref{decay_condition}.

\begin{lemma}\label{aux_est}
Let $0 \le \varphi(s) < \eta_\epsilon$ and let $a=a(\varphi(s))$ be as in Definition \ref{admissible_a}. Suppose $\bm{g}\in \mc{C}_a(-1,1)$. Then for $j\ge 1$, we have
\begin{equation}\label{decay_est}
\abs{\frac{\bm{g}(s)}{((\varphi(s)-1)^2+\epsilon^2a^2)^{j/2}}} \le C \epsilon^{-j+1}a^{-j}\norm{\bm{g}}_{\mc{C}_a(-1,1)},
\end{equation}
where $C$ depends only on $\delta_a$, $c_a$, and $c_\eta$. 
\end{lemma}

\begin{proof}
We first consider cross sections $\varphi(s)$ toward the middle of the fiber: $\varphi(s)\le 1-\delta_a$, where $\delta_a$ is as in Definition \ref{admissible_a}. Here,
\[ \abs{\frac{\bm{g}(s)}{((\varphi(s)-1)^2+\epsilon^2a^2)^{j/2}}} \le \frac{\abs{\bm{g}(s)}}{\abs{\varphi(s)-1}^j} \le \delta_a^{-j} \abs{\bm{g}(s)} \le C \epsilon^{-j+1}a^{-j}\abs{\bm{g}(s)},\]
since $j\ge 1$. \\

For $\varphi(s)>1-\delta_a$, the radius function $a(\varphi(s))$ is spheroidal (see Definition \ref{admissible_a}) and we can take advantage of the decay condition \eqref{decay_condition} on $\bm{g}$. First note that for $\varphi(s)> 1-\epsilon^2$, we have 
\[ a(\varphi(s)) \le \epsilon(1+c_a\epsilon) \sqrt{3c_\eta+2} \le C\epsilon,\]
where we have used Definition \ref{admissible_a} along with \eqref{eta_epsilon}.  \\

Then 
\[ \abs{\frac{\bm{g}(s)}{((\varphi(s)-1)^2+\epsilon^2a^2)^{j/2}}} \le \frac{a\norm{\bm{g}}_{\mc{C}_a(-1,1)}}{(\epsilon a)^j} \le C\norm{\bm{g}}_{\mc{C}_a(-1,1)}\epsilon^{-j+1} a^{-j}. \]

If $\delta_a \le \epsilon^2$, then we are done. If not, consider $1-\delta_a < \varphi(s)\le 1-\epsilon^2$. Write $\abs{1-\varphi(s)}=\epsilon^k$, $\frac{\log\delta_a}{\log\epsilon}<k \le 2$. Then there exist $C_1$, $C_2$ independent of $\epsilon$ such that 
\[ C_1\le \frac{a(\varphi(s))}{\epsilon^{k/2}} \le C_2, \]
since, using \eqref{eta_epsilon},
\begin{equation}\label{estest}
\frac{\sqrt{\eta_\epsilon^2-\varphi^2(s)}}{\epsilon^{k/2}} = \frac{\sqrt{\epsilon^k\abs{1+\varphi(s)}+\abs{\eta_\epsilon-1}\abs{\eta_\epsilon+1}}}{\epsilon^{k/2}} \le \sqrt{2+3c_\eta \epsilon^{2-k}} \le \sqrt{2+3c_\eta}. 
\end{equation}
Furthermore,
\[ \frac{\sqrt{\epsilon^k\abs{1+\varphi(s)}+\abs{\eta_\epsilon-1}\abs{\eta_\epsilon+1}}}{\epsilon^{k/2}} \ge \sqrt{1+\varphi(s)} \ge 1,\]
and therefore, using the spheroidal endpoint condition of Definition \ref{admissible_a} and \eqref{estest}, we have 
\[ 1- c_a\epsilon^2\sqrt{2+3c_\eta} \le \frac{a(\varphi(s))}{\epsilon^{k/2}} \le \sqrt{2+3c_\eta } + c_a\epsilon^2\sqrt{2+3c_\eta}.\]

Thus, using that $a\sim \epsilon^{k/2}$ for $\log\delta_a/\log\epsilon <k\le 2$, we have
\begin{align*}
 \abs{\frac{\bm{g}(s)}{((\varphi(s)-1)^2+\epsilon^2a^2)^{j/2}}} &\le \frac{a\norm{\bm{g}}_{\mc{C}_a(-1,1)}}{\abs{\varphi(s)-1}^j} = \frac{a\norm{\bm{g}}_{\mc{C}_a(-1,1)}}{\epsilon^{jk}} \le C\norm{\bm{g}}_{\mc{C}_a(-1,1)}\epsilon^{k(-j+1/2)} \\
 & \le C \norm{\bm{g}}_{\mc{C}_a(-1,1)} \epsilon^{k(-j+1)/2}a^{-j} \le C \norm{\bm{g}}_{\mc{C}_a(-1,1)} \epsilon^{-j+1}a^{-j},
 \end{align*}
 since $j\ge 1$. 
 
\end{proof}

We use Lemma \ref{aux_est} to show the free end analogue of Lemma 3.4 from the closed loop case \cite{closed_loop}. 
\begin{lemma}\label{est_free2}
Let $a=a(\varphi(s))$ be as in Definition \ref{admissible_a} and let $\bm{R}$ be as in \eqref{R_expand}. Let $\bm{g}\in \mc{C}^1(-1,1)\cap \mc{C}_a(-1,1)$ satisfy \eqref{decay_condition}, and consider $s\in(-1,1)$. For odd integer $m>0$, integer $n\ge m+2$, and $\epsilon$ sufficiently small, we have
\begin{equation}
\begin{aligned}
\bigg| \int_{\varphi(s)-1}^{\varphi(s)+1}& \frac{\bars^m}{\abs{\bm{R}}^n} \bm{g}(\varphi(s)-\bars) d\bars \bigg| \\
&\le \begin{cases}
C\big(\norm{\bm{g}}_{\mc{C}^1(-1,1)} \abs{\log(\epsilon a)}+\norm{\bm{g}}_{\mc{C}_a(-1,1)}a^{-1} \big), & n=m+2 \\
C\big(\norm{\bm{g}}_{\mc{C}^1(-1,1)} (\epsilon a)^{m+2-n}+ \norm{\bm{g}}_{\mc{C}_a(-1,1)} \epsilon^{m+2-n}a^{m+1-n} \big), & n\ge m+3,
\end{cases}
\end{aligned}
\end{equation}
where the constants $C$ depend on the constants $c_\Gamma$, $\kappa_{\max}$, $\delta_a$, $c_a$, $c_\varphi$, and $c_\eta$ from Section \ref{geometry_section}, but not on $\epsilon$.
\end{lemma}

Note that the $\mc{C}^1$ term of Lemma \ref{est_free2} is exactly the same as Lemma 3.4 from \cite{closed_loop}, with the varying radius $\epsilon a$ replacing the uniform radius $\epsilon$. However, in the free end setting, we gain a new term due to the fiber endpoints with an additional factor of $1/a^2$. To avoid blowup at the fiber endpoints in the final error estimate \eqref{err_est_free}, this second term will require additional structure on $\bm{g}$. In particular, we will need $\bm{g}\in \mc{C}_a(-1,1)$. 

\begin{proof}[Proof of Lemma \ref{est_free2}:]
We show Lemma \ref{est_free2} for $0\le \varphi(s) < \eta_\epsilon$; the result for $-\eta_\epsilon < \varphi(s) \le 0$ follows by symmetry of $\varphi(s)$. \\

We first note that since $m$ is odd, we have 
\[ \int_{-\infty}^{\infty} \frac{\bars^m}{(\bars^2+\epsilon^2a^2)^{n/2}} d\bars =0, \]
and thus we can write
\begin{equation}
\begin{aligned}
\int_{\varphi(s)-1}^{\varphi(s)+1} \frac{\bars^m}{\abs{\bm{R}}^n} \bm{g}(\varphi(s)-\bars) d\bars &= \int_{\varphi(s)-1}^{\varphi(s)+1} \frac{\bars^m}{\abs{\bm{R}}^n} \bm{g}(\varphi(s)-\bars) d\bars - \bm{g}(s)\int_{-\infty}^{\infty} \frac{\bars^m}{(\bars^2+\epsilon^2a^2)^{n/2}} d\bars \\
&= I_1+I_2+I_3+I_4; \\
I_1 &= \int_{\varphi(s)-1}^{\varphi(s)+1} \frac{\bars^m}{\abs{\bm{R}}^n} \big( \bm{g}(\varphi(s)-\bars) - \bm{g}(s)\big) d\bars \\
I_2 &= \int_{\varphi(s)-1}^{\varphi(s)+1} \bars^m\bigg(\frac{1}{\abs{\bm{R}}^n} - \frac{1}{(\bars^2+\epsilon^2a^2)^{n/2}}\bigg) \bm{g}(s) d\bars \\
I_3 &= -\int_{\varphi(s)+1}^{\infty} \frac{\bars^m}{(\bars^2+\epsilon^2a^2)^{n/2}} \bm{g}(s) d\bars \\
I_4 &= - \int_{-\infty}^{\varphi(s)-1} \frac{\bars^m}{(\bars^2+\epsilon^2a^2)^{n/2}} \bm{g}(s) d\bars.
\end{aligned}
\end{equation}
In the closed loop setting, we relied on symmetry of the integrals about $\bars=0$ for complete cancellation of $I_3$ and $I_4$. Here, however, such symmetry only holds near the fiber center $(\varphi(s)=0)$. Instead we must rely on decay of $\bm{g}$ near the fiber endpoints \eqref{decay_condition} to make $I_4$ small. Note that in the case $\varphi(s)\le 0$, the decay condition is needed to make the integral $I_3$ small. \\

To estimate $I_1$, we note that, using Definition \ref{varphi_def}, 
\begin{align*}
\abs{\bm{g}(\varphi(s)-\bars) - \bm{g}(s)} &\le \big(\abs{\varphi(s)-s}+\abs{\bars}\big) \norm{\bm{g}'}_{\mc{C}(-1,1)} \le (c_{\varphi}\epsilon^2 + \abs{\bars})\norm{\bm{g}'}_{\mc{C}(-1,1)},
\end{align*}
and therefore
\begin{align*}
\abs{I_1} &\le \int_{\varphi(s)-1}^{\varphi(s)+1} \frac{c_{\varphi}\epsilon^2\abs{\bars}^m+\abs{\bars}^{m+1}}{\abs{\bm{R}}^n}  \norm{\bm{g}'}_{\mc{C}(-1,1)} \le \begin{cases}
C \norm{\bm{g}'}_{\mc{C}(-1,1)}(c_\varphi +\abs{\log(\epsilon a)}), & n=m+2 \\
C \norm{\bm{g}'}_{\mc{C}(-1,1)}(\epsilon a)^{m+2-n}, & n\ge m+3,
\end{cases}
\end{align*}
where we have used Lemmas \ref{est_free1} and \ref{est_free1_new}. \\

As in the closed loop setting, to estimate $I_2$, we use that
\[ \frac{1}{\abs{\bm{R}}^n} - \frac{1}{(\bars^2+\epsilon^2a^2)^{n/2}} = \bigg(\frac{1}{\abs{\bm{R}}} - \frac{1}{\sqrt{\bars^2+\epsilon^2a^2}} \bigg) \sum_{\ell=0}^{n-1} \frac{1}{\abs{\bm{R}}^\ell(\bars^2+\epsilon^2a^2)^{(n-1-\ell)/2}}.\]

By Lemma \ref{R_ests_free}, we have
\[ \abs{\frac{1}{\abs{\bm{R}}} - \frac{1}{\sqrt{\bars^2+\epsilon^2a^2}}} = \frac{\abs{\abs{\bm{R}}-\sqrt{\bars^2+\epsilon^2a^2}}}{\abs{\bm{R}}\sqrt{\bars^2+\epsilon^2a^2}} \le \frac{C\bars^2}{\bars^2+\epsilon^2a^2}. \]
and therefore, using the bound \eqref{Rlb_free} in Lemma \ref{R_ests_free}, we have 
\begin{equation}\label{Rn_est}
\abs{\frac{1}{\abs{\bm{R}}^n}- \frac{1}{(\bars^2+\epsilon^2a^2)^{n/2}}} \le \frac{C\bars^2}{(\bars^2+\epsilon^2a^2)^{(n+1)/2}}.
\end{equation}
Then, using Lemma \ref{integral_est}, we have
\begin{align*}
\abs{I_2} &\le C\norm{\bm{g}}_{\mc{C}(-1,1)}\int_{\varphi(s)-1}^{\varphi(s)+1} \frac{\bars^{m+2}}{(\bars^2+\epsilon^2a^2)^{(n+1)/2}} d\bars \\ 
&\le \begin{cases}
C\norm{\bm{g}}_{\mc{C}(-1,1)}\abs{\log(\epsilon a)},  & n=m+2\\
C\norm{\bm{g}}_{\mc{C}(-1,1)} (\epsilon a)^{m+2-n}, & n\ge m+3.
\end{cases} 
\end{align*}

Next we estimate $I_3$. Using that $\varphi(s)\ge0$ and $n\ge m+2$, we have
\begin{align*}
\abs{I_3} &\le \norm{\bm{g}}_{\mc{C}(-1,1)} \int_{\varphi(s)+1}^{\infty} \frac{\bars^m}{(\bars^2+\epsilon^2a^2)^{n/2}} d\bars \le  \norm{\bm{g}}_{\mc{C}(-1,1)} \int_{\varphi(s)+1}^{\infty} \bars^{m-n} \ts d\bars \\
&= \frac{(\varphi(s)+1)^{m-n+1}}{n-m-1} \norm{\bm{g}}_{\mc{C}(-1,1)} \le \norm{\bm{g}}_{\mc{C}(-1,1)}.
\end{align*}

Finally we estimate $I_4$. Again, this term is not present in the closed loop setting because of cancellation due to symmetry in $\bars$. In the free end setting, this term is the reason for the decay condition \eqref{decay_condition} at the fiber endpoints. First, note that direct integration in $\bars$ yields
\begin{align*}
I_4 = \bm{g}(s)\sum_{\ell=0,\text{ even}}^{m-1} C_\ell \frac{(\varphi(s)-1)^{m-1-\ell}(\epsilon a)^\ell}{((\varphi(s)-1)^2+\epsilon^2a^2)^{(n-2)/2}}.  
\end{align*}

Then, using Lemma \ref{aux_est}, we have
\begin{align*}
\abs{I_4} &\le C\abs{\bm{g}(s)}\sum_{\ell=0,\text{ even}}^{m-1}  \frac{(\epsilon a)^\ell}{((\varphi(s)-1)^2+\epsilon^2a^2)^{(n-m-1+\ell)/2}} \\
&\le C\norm{\bm{g}}_{\mc{C}_a(-1,1)} \epsilon^{m+2-n} a^{m+1-n},
\end{align*}
since $n\ge m+2$. Combining the estimates for $I_1$ through $I_4$, we obtain Lemma \ref{est_free2}.
\end{proof}

Next we show the following free endpoint analogue of Lemma 3.5 in the closed loop setting \cite{closed_loop}. Again, the $\mc{C}^1$ term of this estimate is essentially identical to the analogous closed loop estimate. However, as in Lemma \ref{est_free2}, an additional $\mc{C}_a$ term is required to account for the fiber endpoints. 

\begin{lemma}\label{est_free3}
Let $\bm{R}$ be as in \eqref{R_expand}, $a=a(\varphi(s))$ as in Definition \ref{admissible_a}, and $\bm{g}\in \mc{C}^1(-1,1)\cap \mc{C}_a(-1,1)$ satisfying \eqref{decay_condition}. Consider even integer $m\ge 0$ and odd integer $n\ge m+3$. Then, for $\varphi(s)\in(-\eta_\epsilon,\eta_\epsilon)$ and sufficiently small $\epsilon$, we have
\begin{equation}
\begin{split}
\bigg| \int_{\varphi(s)-1}^{\varphi(s)+1} &\frac{\bars^m}{\abs{\bm{R}}^n}\bm{g}(\varphi(s)-\bars)d\bars - (\epsilon a)^{m+1-n}d_{mn} \bm{g}(s) \bigg| \\
&\le C\big(\norm{\bm{g}}_{\mc{C}^1(-1,1)}(\epsilon a)^{m+2-n} +\norm{\bm{g}}_{\mc{C}_a(-1,1)}\epsilon^{m+2-n} a^{m+1-n} \big),\\
d_{mn}&=\int_{-\infty}^\infty \frac{\tau^m}{(\tau^2+1)^{n/2}}d\tau.
\end{split}
\end{equation}
The constant $C$ depends on $c_\Gamma$, $\kappa_{\max}$, $\delta_a$, $c_a$, $c_\varphi$, and $c_\eta$ from Section \ref{geometry_section}, but not on $\epsilon$. 
\end{lemma}

\begin{proof}
Again, it suffices to prove Lemma \ref{est_free3} for $0\le \varphi(s)<\eta_\epsilon$; the result for $-\eta_\epsilon <\varphi(s)\le 0$ follows by symmetry of $\varphi(s)$ (see Definition \ref{varphi_def}). \\

As in the proofs of Lemma \ref{est_free2} and the closed loop Lemma 3.5 \cite{closed_loop}, we write
\begin{equation}
\begin{split}
\int^{\varphi(s)-1}_{\varphi(s)+1}& \frac{\bars^m \bm{g}(\varphi(s)-\bars)}{\abs{\bm{R}}^n}d\bars - \bm{g}(s)\int_{-\infty}^{\infty} \frac{\bars^m}{(\bars^2+\epsilon^2a^2)^{n/2}}d\bars =I_1+I_2+I_3 +I_4, \\
I_1 &=\int_{\varphi(s)-1}^{\varphi(s)+1} \frac{\bars^m (\bm{g}(\varphi(s)-\bars)- \bm{g}(s))}{\abs{\bm{R}}^n}d\bars,\\
I_2 &=\int_{\varphi(s)-1}^{\varphi(s)+1} \bars^m\bigg(\frac{1}{\abs{\bm{R}}^n}-\frac{1}{(\bars^2+\epsilon^2a^2)^{n/2}}\bigg)\bm{g}(s) \ts d\bars,\\
I_3 &=\bm{g}(s)\int_{\varphi(s)+1}^\infty \frac{\bars^m}{(\bars^2+\epsilon^2a^2)^{n/2}}d\bars, \\
I_4 &=\bm{g}(s)\int_{-\infty}^{\varphi(s)-1} \frac{\bars^m}{(\bars^2+\epsilon^2a^2)^{n/2}}d\bars.
\end{split}
\end{equation}

We estimate $I_1$, $I_2$, and $I_3$ exactly as in Lemma \ref{est_free2}, yielding
\begin{align*}
\abs{I_1} &\le C\norm{\bm{g}'}_{\mc{C}(-1,1)}(\epsilon a)^{m+2-n}, \\
\abs{I_2} &\le C\norm{\bm{g}}_{\mc{C}(-1,1)}(\epsilon a)^{m+2-n},\\
\abs{I_3} &\le \norm{\bm{g}}_{\mc{C}(-1,1)}.
\end{align*}

For $I_4$, we first consider $\varphi(s) \le 1-\delta_a$, where $\delta_a$ is the constant defined in Definition \ref{admissible_a}. Here,  
\begin{align*}
\abs{I_4} &\le \norm{\bm{g}}_{\mc{C}(-1,1)}\int_{1-\varphi(s)}^{\infty} \bars^{m-n}d\bars = \frac{(1-\varphi(s))^{m-n+1}}{n-m-1}\norm{\bm{g}}_{\mc{C}(-1,1)} \le \frac{\delta_a^{m-n+1}}{n-m-1}\norm{\bm{g}}_{\mc{C}(-1,1)}.
\end{align*}
Since $\delta_a$ is a constant independent of $\epsilon$ and $n\ge m+3$, we obtain the desired estimate for $\varphi(s)\le 1-\delta_a$. \\

For $\varphi(s)>1-\delta_a$, using that $n$ is odd, we can directly integrate the expression for $I_4$ to obtain
\begin{align*}
I_4 &= \bm{g}(s) \sum_{\ell=0, \text{ even}}^{n-1} C_\ell  (\epsilon a)^{m-n+1-\ell}\frac{(\varphi(s)-1)^{n-2-\ell}}{((\varphi(s)-1)^2+\epsilon^2a^2)^{(n-2)/2}}.
\end{align*}

Then, using Lemma \ref{aux_est}, we have
\begin{align*}
\abs{I_4} &\le C\abs{\bm{g}(s)} \sum_{\ell=0, \text{ even}}^{n-1} \frac{(\epsilon a)^{m-n+1-\ell}}{((\varphi(s)-1)^2+\epsilon^2a^2)^{\ell/2}} \le C\norm{\bm{g}}_{\mc{C}_a(-1,1)} \epsilon^{m+2-n}a^{m+1-n} .
\end{align*}

Furthermore, for $\varphi(s)\in (-\eta_\epsilon,\eta_\epsilon)$, we have 
\[ \int_{-\infty}^{\infty} \frac{\bars^m}{(\bars^2+\epsilon^2a^2)^{n/2}}d\bars = (\epsilon a)^{m+1-n}\int_{-\infty}^{\infty} \frac{\tau^m}{(\tau^2+1)^{n/2}} d\tau \equiv (\epsilon a)^{m+1-n}d_{nm}. \]
Combining this expression with the estimates for $I_1$ through $I_4$ yields Lemma \ref{est_free3}.
\end{proof}

Finally, we show a lemma related to the centerline estimate \eqref{center_err_free} of Theorem \ref{free_err_theorem}. This will be the free end analogue to Lemma 3.6 in the closed loop setting \cite{closed_loop}. In fact, the result in the free endpoint setting is nearly identical to the closed loop case since the estimate \eqref{center_err_free} extends only along the effective centerline of the slender body ($s\in(-1,1)$ rather than $\varphi(s)\in (-\eta_\epsilon,\eta_\epsilon)$). Accordingly, to show the following lemma, it will be useful to define a few quantities that have not yet appeared in the free end setting. \\

First, for the centerline computation, it will often be convenient to work in terms of 
\[ \ws := s-t \quad  \text{ for } s,t \text{ both in } (-1,1)\]
rather than in terms of $\bars=\varphi(s)-t$. Recall the difference $\bm{R}_{\rm C}$ \eqref{RC} between two points on the effective fiber centerline. Using the $\mc{C}^2$ regularity of $\X$, we write
\begin{equation}\label{RC_def}
\bm{R}_{\rm C}(s,\ws) = \X(s) - \X(s-\ws) = \ws\be_{\rm t}(s)+\ws^2\bm{Q}(s,\ws); \quad \ws:= s-t
\end{equation}

It will also be useful to distinguish between $\bm{R}$ evaluated at $\varphi(s)$ along the true fiber centerline, and $\bm{R}$ evaluated at $s$ along the effective centerline. Thus we define 
\begin{equation}\label{tR_def}
\wR(s,\ws,\theta) = \bm{R}(s,\ws,\theta) = \ws \be_{\rm t}(s) + \epsilon a(s) \be_\rho(s,\theta)+\ws^2\bm{Q}(s,\ws).
\end{equation}
Hereafter, $\bm{R}$ will always be evaluated at $\varphi(s)\in (-\eta_\epsilon,\eta_\epsilon)$, while $\wR$ will always be evaluated at $s\in (-1,1)$. \\

Using Lemma \ref{R_ests_free} and \eqref{no_intersect}, we note the following useful upper and lower bounds for $\bm{R}_{\rm C}$ and $\wR$ in terms of $\ws$:
\begin{equation}\label{upper_bds}
\big| |\wR|-\sqrt{\ws^2+\epsilon^2a^2(s)} \big| \le \frac{\kappa_{\max}}{2}\ws^2, \quad \big| \abs{\bm{R}_{\rm C}}-\abs{\ws} \big|\le \frac{\kappa_{\max}}{2}\ws^2,
\end{equation}
\begin{equation}\label{lower_bds}
|\wR|\ge C\sqrt{\ws^2+\epsilon^2a^2(s)}, \quad \abs{\bm{R}_{\rm C}} \ge c_\Gamma\abs{\ws}.
\end{equation}

Using \eqref{RC_def} and \eqref{tR_def}, we show the following lemma.
\begin{lemma}\label{center_lem_free}
Let $\bm{R}$ be as in \eqref{R_expand} and $\bm{R}_{\rm C}$ as in \eqref{RC_def}, and let $n=1,3$. Then for $\bm{g}\in \mc{C}^1(-1,1)$ and $\epsilon$ sufficiently small, we have
\begin{equation}\label{cent_lem_eq}
\begin{aligned}
\bigg|\int_{s-1}^{s+1} \frac{|\ws|^{n-1}}{|\wR|^n}\bm{g}(s-\ws)d\ws - &\int_{s-1}^{s+1} \bigg(\frac{\abs{\ws}^{n-1}}{\abs{\bm{R}_{\rm C}}^n} \bm{g}(s-\ws)-\frac{\bm{g}(s)}{\abs{\ws} } \bigg)d\ws + L(s)\bm{g}(s) + (n-1)\bm{g}(s) \bigg| \\
& \le \epsilon \abs{\log\epsilon} C\norm{\bm{g}}_{\mc{C}^1(-1,1)},
\end{aligned}
\end{equation} 
where
\[ L(s) = -\log\bigg( \frac{2(1-s^2)+2\sqrt{(1-s^2)^2+\epsilon^2a^2(s)}}{\epsilon^2a^2(s)}\bigg) \]
and the constant $C$ depends on $c_\Gamma$ and $\kappa_{\max}$ but not on $\epsilon$.
\end{lemma}

\begin{proof}
Note that this result follows by almost exactly the same arguments used to show Lemma 3.6 in \cite{closed_loop}.
We begin by considering a bound for the expression
\begin{equation}\label{bmJ}
\begin{aligned}
\bm{J} = \int_{s-1}^{s+1} &\bigg[\bigg(\frac{\abs{\ws}^{n-1}}{|\wR|^n} - \frac{\abs{\ws}^{n-1}}{\abs{\bm{R}_{\rm C}}^n} \bigg) \bm{g}(s-\ws) \\
&\qquad + \frac{(\epsilon a)^2 \bm{g}(s)}{\abs{\ws}\sqrt{\ws^2+(\epsilon a)^2} (\abs{\ws}+\sqrt{\ws^2+(\epsilon a)^2})} + (n-1)\bm{g}(s)\bigg]d\ws. 
\end{aligned}
\end{equation}
Here and for the remainder of this proof only, we have $a=a(s)$ instead of the usual $a=a(\varphi(s))$, unless otherwise specified. \\

To obtain a bound for $\bm{J}$, we consider $\bm{J}$ as the sum
\begin{align*}
\bm{J} &= \bm{J}_1 + \bm{J}_2+\bm{J}_3; \\
\bm{J}_1 &:= \int_{s-1}^{s+1} \bigg(\frac{\ws^{n-1}}{|\wR|^n} - \frac{\ws^{n-1}}{\abs{\bm{R}_{\rm C}}^n} \bigg) (\bm{g}(s-\ws)-\bm{g}(s))d\ws \\
\bm{J}_2 &:= \int_{s-1}^{s+1} \bigg(\frac{1}{|\wR|} - \frac{1}{\abs{\bm{R}_{\rm C}}} +\frac{(\epsilon a)^2}{\abs{\ws}\sqrt{\ws^2+(\epsilon a)^2} (\abs{\ws}+\sqrt{\ws^2+(\epsilon a)^2})}  \bigg) \bm{g}(s) d\ws \\
\bm{J}_3 &:= \int_{s-1}^{s+1} \bigg(\frac{\ws^{n-1}}{|\wR|^n} -\frac{1}{|\wR|} - \frac{\bars^{n-1}}{\abs{\bm{R}_{\rm C}}^n} +\frac{1}{\abs{\bm{R}_{\rm C}}} \bigg) \bm{g}(s) d\ws + (n-1) \bm{g}(s).
\end{align*}

For ease of notation in estimating $\bm{J}_1$ through $\bm{J}_3$, we use \eqref{RC_def} and \eqref{tR_def} to define the quantity  
\begin{equation}\label{IR_free}
I_R := \frac{1}{|\wR|} - \frac{1}{\abs{\bm{R}_{\rm C}}} = \frac{-\epsilon^2a^2 - 2\epsilon a \ws \bm{Q}\cdot\be_\rho}{|\wR|\abs{\bm{R}_{\rm C}}(|\wR|+\abs{\bm{R}_{\rm C}})}.
\end{equation}

Now, note that by \eqref{lower_bds}, we have 
\[ \abs{\frac{\ws^{n-1}}{|\wR|^n} - \frac{\ws^{n-1}}{\abs{\bm{R}_{\rm C}}^n}} \le \abs{I_R}\sum_{\ell=0}^{n-1}\frac{\ws^{n-1}}{|\wR|^\ell \abs{\bm{R}_{\rm C}}^{n-1-\ell}}\le C\abs{I_R}  \le C\bigg( \frac{(\epsilon a)^2}{\abs{\ws}(\ws^2+ (\epsilon a)^2)} + \frac{\epsilon a \abs{\ws}}{\ws^2+(\epsilon a)^2}\bigg), \]
and, since $\bm{g}\in \mc{C}^1(-1,1)$, we have 
\begin{align*}
\abs{\bm{g}(s-\ws) - \bm{g}(s)} &\le \ws \norm{\bm{g}'}_{\mc{C}(-1,1)}.
\end{align*}

Then, by Lemma \ref{integral_est} with $s$ in place of $\varphi(s)$, we have 
\begin{align*}
\abs{\bm{J}_1} &\le C\norm{\bm{g}'}_{\mc{C}(-1,1)} \int_{s-1}^{s+1} \frac{(\epsilon a)^2+\epsilon a\ws^2}{\ws^2+(\epsilon a)^2} d\ws \le \epsilon a(s) C\norm{\bm{g}'}_{\mc{C}(-1,1)}.
\end{align*}

Next, using \eqref{IR_free}, the integrand of $\bm{J}_2$ satisfies 
\begin{align*}
\bigg| I_R &+\frac{(\epsilon a)^2}{\abs{\ws}\sqrt{\ws^2+(\epsilon a)^2} (\abs{\ws}+\sqrt{\ws^2+(\epsilon a)^2})} \bigg| \\
&\le \abs{\frac{1}{\sqrt{\ws^2+(\epsilon a)^2}} - \frac{1}{|\wR|}} \frac{(\epsilon a)^2}{\abs{\bm{R}_{\rm C}} (\abs{\bm{R}_{\rm C}}+|\wR|)} + \abs{\frac{1}{\abs{\ws}} - \frac{1}{\abs{\bm{R}_{\rm C}}} } \frac{(\epsilon a)^2}{\sqrt{\ws^2+(\epsilon a)^2} (\abs{\bm{R}_{\rm C}}+|\wR|)}  \\ 
&\quad +\abs{\frac{1}{(\abs{\ws}+\sqrt{\ws^2+(\epsilon a)^2})}- \frac{1}{(\abs{\bm{R}_{\rm C}}+|\wR|)}}\frac{(\epsilon a)^2}{\abs{\ws}\sqrt{\ws^2+ (\epsilon a)^2}} + \frac{C\epsilon a\ws^2}{|\wR|\abs{\bm{R}_{\rm C}}(\abs{\bm{R}_{\rm C}}+|\wR|)} \\
&\le C \frac{(\epsilon a)^2+\epsilon a \abs{\ws}}{\ws^2+(\epsilon a)^2}.
\end{align*}
Here we have used \eqref{upper_bds} to bound each of the difference expressions, and \eqref{lower_bds} to bound each of the denominators. Then by Lemma \ref{int_est_eqn} we obtain
\begin{align*}
\abs{\bm{J}_2} \le C \int_{s-1}^{s+1} \frac{((\epsilon a)^2+\epsilon a \abs{\ws})|\bm{g}(s)|}{\ws^2+(\epsilon a)^2} d\ws \le C \epsilon\abs{\log\epsilon} \norm{\bm{g}}_{\mc{C}(-1,1)}.
\end{align*}

Now, for $n=3$, we must also estimate $\bm{J}_3$. We have
\begin{align*}
\abs{\bm{J}_3} &\le \norm{\bm{g}}_{\mc{C}(-1,1)}\int_{s-1}^{s+1} \bigg( \bigg|\frac{\ws^2+ (\epsilon a)^2-|\wR|^2}{|\wR|^3} - \frac{\ws^2-\abs{\bm{R}_{\rm C}}^2}{\abs{\bm{R}_{\rm C}}^3}\bigg| +\bigg|\frac{(\epsilon a)^2}{|\wR|^3} - \frac{(\epsilon a)^2}{\sqrt{\ws^2+(\epsilon a)^2}^3} \bigg|  \bigg) d\ws \\
&\qquad + \norm{\bm{g}}_{\mc{C}(-1,1)}\bigg|\int_{s-1}^{s+1} \frac{(\epsilon a)^2}{\sqrt{\ws^2+(\epsilon a)^2}^3}d\ws - 2 \bigg| \\
&\le \norm{\bm{g}}_{\mc{C}(-1,1)}\int_{s-1}^{s+1} \bigg(\frac{2\ws^3\bm{Q}\cdot\be_{\rm t}-\ws^4\abs{\bm{Q}}^2-2\epsilon a \ws^2\bm{Q}\cdot\be_\rho}{|\wR|^3} - \frac{2\ws^3\bm{Q}\cdot\be_{\rm t}-\ws^4\abs{\bm{Q}}^2}{\abs{\bm{R}_{\rm C}}^3} \bigg) d\ws \\
&\qquad +C \norm{\bm{g}}_{\mc{C}(-1,1)}\int_{s-1}^{s+1} \frac{\ws^2}{(\ws^2+(\epsilon a)^2)^2} d\ws  + \norm{\bm{g}}_{\mc{C}(-1,1)}\bigg|\frac{2}{\sqrt{1+4(\epsilon a)^2}} - 2 \bigg| \\
&\le C\norm{\bm{g}}_{\mc{C}(-1,1)}\int_{s-1}^{s+1} (|\ws|^3+\ws^4)\bigg(\abs{I_R} \sum_{\ell=0}^2\frac{1}{\abs{\bm{R}_{\rm C}}^\ell|\wR|^{2-\ell}}  \bigg) d\ws+C \epsilon a \abs{\log(\epsilon a)} \norm{\bm{g}}_{\mc{C}(-1,1)} \\
&\le C\norm{\bm{g}}_{\mc{C}(-1,1)}\int_{s-1}^{s+1} \frac{(\epsilon a)^2+\epsilon a \ws^2 + (\epsilon a)^2|\ws|+\epsilon a |\ws|^3}{|\wR|^2} d\ws+ C \epsilon \abs{\log\epsilon} \norm{\bm{g}}_{\mc{C}(-1,1)} \\ 
&\le C\epsilon \abs{\log\epsilon} \norm{\bm{g}}_{\mc{C}(-1,1)}.
\end{align*}
Here we have used \eqref{RC_def}, \eqref{tR_def}, and \eqref{Rn_est} in the second inequality, \eqref{IR_free} and Lemmas \ref{integral_est} and \ref{est_free1} in the third inequality, and \eqref{lower_bds} in the fourth inequality. \\

It remains to show that $\bm{J}$ is in fact close to the expression \eqref{cent_lem_eq} of Lemma \ref{center_lem_free}. First note that 
\begin{align*}
\int_{s-1}^{s+1}&\bigg(\frac{(\epsilon a)^2}{\abs{\ws}\sqrt{\ws^2+(\epsilon a)^2} (\abs{\ws}+\sqrt{\ws^2+(\epsilon a)^2})} -\frac{1}{\abs{\ws}} \bigg) d\ws \\
&= -\int_{s-1}^{s+1}\frac{\abs{\ws}\sqrt{\ws^2+(\epsilon a)^2}+\ws^2}{\abs{\ws}\sqrt{\ws^2+ (\epsilon a)^2} (\abs{\ws}+\sqrt{\ws^2+ (\epsilon a)^2})} d\ws = -\int_{s-1}^{s+1}\frac{1}{\sqrt{\ws^2+(\epsilon a)^2} } d\ws \\
& = - \log\bigg(\frac{\big(\sqrt{(s+1)^2+(\epsilon a)^2}+(s+1)\big) \big(\sqrt{(s-1)^2+(\epsilon a)^2}-(s-1) \big)}{(\epsilon a)^2} \bigg).
\end{align*}

Now, let  
\[ L_1=\big(\sqrt{(s+1)^2+(\epsilon a)^2}+(s+1)\big) \big(\sqrt{(s-1)^2+(\epsilon a)^2}-(s-1) \big), \quad L_2=2(1-s^2)+2\sqrt{(1-s^2)^2+4\epsilon^2a^2}. \]
We work with the positive half of the fiber ($s\ge0$); the result for the negative half follows by exactly the same arguments. We can show that 
\begin{align*}
\abs{L_1-L_2} &= \bigg|\frac{(1-s)\epsilon^2a^2}{\sqrt{(1+s)^2+\epsilon^2a^2}+(1+s)} + \frac{[(1+s)^2-4]}{\sqrt{(1-s^2)^2+\epsilon^2a^2(1+s)^2}+\sqrt{(1+s^2)^2+4\epsilon^2a^2}} \\
&\qquad +\frac{[(1+s)^2-4]\epsilon^2a^2+\epsilon^2a^2(1-s)^2+\epsilon^4a^4}{\sqrt{(1-s^2)^2+\epsilon^2a^2(1+s)^2+\epsilon^2a^2(1-s)^2+\epsilon^4a^4}+\sqrt{(1-s^2)^2+4\epsilon^2a^2}} \bigg| \\
&\le C\frac{(1-s)\epsilon^2 a^2+\epsilon^4a^4}{(1-s)+\epsilon a}.
\end{align*}
Furthermore, we have that both
\[ L_1, L_2\ge C\big((1-s)+\epsilon a\big), \]
and therefore 
\[ \abs{\log L_1 - \log L_2} \le C\frac{(1-s)\epsilon^2a^2+\epsilon^4a^4}{\big((1-s)+\epsilon a \big)^2}. \]

Now, if $s\le 1-\varphi^{-1}(\delta_a)$ for $\delta_a$ as in Definition \ref{admissible_a}, then $\abs{\log L_1 - \log L_2} \le C\epsilon^2a^2$. \\

 Furthermore, if $s\ge1-\epsilon$, then, using that $a(s)\ge C\epsilon$ for $s\in(-1,1)$ by Definition \ref{admissible_a}, we have $\abs{\log L_1 - \log L_2} \le C\epsilon$. \\
 
 If $\varphi^{-1}(\delta_a)\le \epsilon$, then we conclude $\abs{\log L_1 - \log L_2} \le C\epsilon$. If not, we consider $1-\varphi^{-1}(\delta_a)\le s \le 1-\epsilon$. Writing $1-s=\epsilon^k$ for $\log(\varphi^{-1}(\delta_a))/\log\epsilon \le k \le 1$, as in the proof of Lemma \ref{aux_est}, we have 
\[ C_1\le \frac{a(s)}{\epsilon^{k/2}} \le C_2, \]
for some $C_1$, $C_2$ independent of $\epsilon$, by Definition \ref{admissible_a}. Then
\[ \abs{\log L_1 - \log L_2} \le C\frac{\epsilon^{2+2k}}{\epsilon^{2k}(1+\epsilon^{1-k/2})^2} \le C\epsilon^2,\]
and therefore
\[ \bigg| \bm{J} - \int_{s-1}^{s+1} \bigg(\frac{\abs{\ws}^{n-1}}{\abs{\bm{R}_{\rm C}}^n} \bm{g}(s-\ws)-\frac{\bm{g}(s)}{\abs{\ws} } \bigg)d\ws + L(s)\bm{g}(s) + (n-1)\bm{g}(s) \bigg| \le C\epsilon \norm{\bm{g}}_{\mc{C}(-1,1)}. \]
Thus we obtain Lemma \ref{center_lem_free}.
\end{proof}

\subsection{Slender body velocity residual}\label{vel_resid}
We first wish to determine the difference between the approximate slender body velocity $\bu^{\SB}$ \eqref{SBT_free} along the surface $\Gamma_\epsilon$ and our (partial) boundary velocity data: $\bu\big|_{\Gamma_\epsilon}$ is independent of $\theta$. Thus we seek to measure the degree to which the slender body approximation $\bu^{\SB}\big|_{\Gamma_\epsilon}$ depends on $\theta$. Using Lemmas \ref{est_free1}, \ref{est_free2}, and \ref{est_free3} in place of the corresponding closed loop lemmas (see Lemmas 3.3, 3.4, and 3.5, respectively, in \cite{closed_loop}), we can show a series of propositions analogous to Propositions 3.7, 3.8, and 3.9 in the closed loop setting \cite{closed_loop} which will allow us to quantify this $\theta$-dependence. Additionally, we show a bound for the residual between the slender body approximation \eqref{SBT_free} along the fiber surface and the centerline approximation \eqref{SBT_asymp_free}, which will eventually allow us to prove estimate \eqref{center_err_free} of Theorem \ref{free_err_theorem}. We begin with the following proposition. 

\begin{proposition}\label{theta_free}
Let $a=a(\varphi(s))$ be as in Definition \ref{admissible_a}, and let $\bm{f}\in \mc{C}^1(-1,1)\cap \mc{C}_a(-1,1)$ satisfy \eqref{decay_condition}. Consider $\bu^{\SB}(\bx)$ \eqref{SBT_free} for $\bx\in \Gamma_\epsilon$. For sufficiently small $\epsilon$, we have 
\begin{equation}
\abs{\frac{1}{\epsilon a} \frac{\p\bu^{\SB}}{\p\theta}} \le C \big( \norm{\bm{f}}_{\mc{C}^1(-1,1)}\abs{\log(\epsilon a)}+\norm{\bm{f}}_{\mc{C}_a(-1,1)}a^{-1} \big)
\end{equation}
where the constant $C$ depends on the constants $c_\Gamma$, $\kappa_{\max}$, $\delta_a$, $c_a$, $\bar c_a$, $c_\varphi$, and $c_\eta$ from Section \ref{geometry_section}, but not on $\epsilon$.
\end{proposition}

\begin{proof}
We begin by writing out the $\theta$-derivative as  
\begin{equation}\label{theta_deriv_def}
\begin{aligned}
\frac{1}{\epsilon a}\frac{\p \bu^{\SB}}{\p \theta} &= (\p_\theta \bu^{\SB})_1+ (\p_\theta \bu^{\SB})_2; \\
(\p_\theta \bu^{\SB})_1&:=  \frac{1}{8\pi}\int_{\varphi(s)-1}^{\varphi(s)+1} \frac{1}{\epsilon a}\bigg(\frac{\p}{\p\theta}\mc{S}(\bm{R}) + \frac{\epsilon^2a^2}{2} \frac{\p }{\p\theta}\mc{D}(\bm{R}) \bigg) \bm{f}(\varphi(s)-\bars)\ts d\bars  \\
(\p_\theta \bu^{\SB})_2 &:= \frac{1}{8\pi} \int_{\varphi(s)-1}^{\varphi(s)+1} \frac{3\epsilon^2 \bars A(\varphi(s),\bars)}{2}\bigg(\frac{\bars^2\bm{Q}\cdot\be_{\theta}}{|\bm{R}|^5} +\frac{\be_{\theta}\bm{R}^{\rm T}+\bm{R}\be_{\theta}^{\rm T}}{|\bm{R}|^5} \\
&\hspace{2cm} -\frac{5\bm{R}\bm{R}^{\rm T}(\bars^2\bm{Q}\cdot\be_{\theta})}{|\bm{R}|^7} \bigg)\bm{f}(\varphi(s)-\bars)\ts d\bars \\
\end{aligned}
\end{equation}
where $A$ is as defined in \eqref{a_expand}. Recall that, unless explicitly stated otherwise, we have $a=a(\varphi(s))$. \\ 

We first note that each of the terms appearing in $(\p_\theta \bu^{\SB})_1$ can be estimated exactly as in the proof of Proposition 3.7 in the closed loop setting \cite{closed_loop}, but with Lemmas 3.3, 3.4, and 3.5 replaced by Lemmas \ref{est_free1}, \ref{est_free2}, and \ref{est_free3}, respectively. Doing so, we obtain 
\[ \abs{(\p_\theta \bu^{\SB})_1} \le C \big( \norm{\bm{f}}_{\mc{C}^1(-1,1)}\abs{\log(\epsilon a)}+\norm{\bm{f}}_{\mc{C}_a(-1,1)}a^{-1} \big), \]
where the constant $C$ depends on $c_\Gamma$, $\kappa_{\max}$, $\delta_a$, $c_a$, $c_\varphi$, and $c_\eta$ but not on $\epsilon$. \\

Thus it remains to bound the additional term $(\p_\theta \bu^{\SB})_2$ resulting from expansion \eqref{a_expand} of the doublet coefficient. Using \eqref{a_expand} and Lemma \ref{est_free1}, we have
\begin{align*}
\abs{(\p_\theta \bu^{\SB})_2} &\le 4\pi\norm{\bm{f}}_{\mc{C}_a(-1,1)} 3\epsilon^2a \int_{\varphi(s)-1}^{\varphi(s)+1}\bigg(\frac{C\abs{\bars}^3\abs{A}}{|\bm{R}|^5} +\frac{\abs{\bars}\abs{A}}{\abs{\bm{R}}^4}\bigg) d\bars \\
&\le C \norm{\bm{f}}_{\mc{C}_a(-1,1)} a^{-1}.
\end{align*}

Combining the estimates for $(\p_\theta \bu^{\SB})_1$ and $(\p_\theta \bu^{\SB})_2$, we obtain Proposition \ref{theta_free}.
\end{proof}

We also show the following proposition, which is the free end analogue of Proposition 3.8 in \cite{closed_loop}.
\begin{proposition}\label{theta_s_free}
Let $a=a(\varphi(s))$ for $s\in (-1,1)$ be as in Definition \ref{admissible_a}, and let $\bm{f}\in \mc{C}^1(-1,1)\cap \mc{C}_a(-1,1)$ satisfy \eqref{decay_condition}. Consider $\bu^{\SB}(\bx)$ \eqref{SBT_free} for $\bx\in \Gamma_\epsilon$. For sufficiently small $\epsilon$, we have
\begin{equation}
\abs{\frac{\p}{\p\theta}\frac{\p\bu^{\SB}}{\p \varphi} } \le C \big( \norm{\bm{f}}_{\mc{C}^1(-1,1)}+\norm{\bm{f}}_{\mc{C}_a(-1,1)}a^{-1} \big)
\end{equation}
where the constant $C$ depends on $c_\Gamma$, $\kappa_{\max}$, $\delta_a$, $c_a$, $c_\varphi$, and $c_\eta$ from Section \ref{geometry_section}, but not on $\epsilon$.
\end{proposition}

\begin{proof}
Following the same outline as the proof of Proposition 3.8 in the closed loop setting \cite{closed_loop}, we rewrite 
\begin{equation}\label{deriv_thetas_def}
\begin{aligned}
\frac{\p}{\p\theta}\frac{\p\bu^{\SB}}{\p \varphi} &= (1-\epsilon a\wh\kappa)\frac{\p \bm{I}^{\SB}}{\p\theta} - \epsilon a\frac{\p\wh\kappa}{\p\theta}\bm{I}^{\SB} ; \\
\bm{I}^{\SB} &:= \frac{1}{1-\epsilon a\wh\kappa}\frac{\p \bu^{\SB}}{\p\varphi} = (\p_\varphi \bu^{\SB})_1+(\p_\varphi \bu^{\SB})_2; \\
(\p_\varphi \bu^{\SB})_1 :=& \frac{1}{8\pi}\int_{\varphi(s)-1}^{\varphi(s)+1} \frac{1}{1-\epsilon a\wh\kappa} \bigg( \frac{\p}{\p\varphi}\mc{S}(\bm{R})+ \frac{\epsilon^2a^2}{2}\frac{\p}{\p\varphi}\mc{D}(\bm{R})\bigg)\bm{f}(\varphi(s)-\bars) \ts d\bars, \\
(\p_\varphi \bu^{\SB})_2 :=& \frac{1}{8\pi}\int_{\varphi(s)-1}^{\varphi(s)+1} \frac{3\epsilon^2 \bars A(\varphi(s),\bars)}{2} \bigg(\frac{\bars+\bars^2\bm{Q}\cdot\be_{\rm t}}{|\bm{R}|^5} + \frac{\be_{\rm t}\bm{R}^{\rm T}+\bm{R}\be_{\rm t}^{\rm T}}{|\bm{R}|^5}\\
&\hspace{3cm} -\frac{5\bm{R}\bm{R}^{\rm T}(\bars+\bars^2\bm{Q}\cdot\be_{\rm t})}{|\bm{R}|^7} \bigg)\bm{f}(\varphi(s)-\bars) \ts d\bars.
\end{aligned}
\end{equation}
where $A$ is as defined in \eqref{a_expand}. Again, unless explicitly stated otherwise, we have $a=a(\varphi(s))$. \\

We begin by bounding $(\p_\varphi \bu^{\SB})_1$ and $\p /\p\theta(\p_\varphi \bu^{\SB})_1$. Here, each term in both expressions can be estimated following the same steps as in the proof of Proposition 3.8 in the closed loop case \cite{closed_loop}, but with Lemmas 3.3, 3.4, and 3.5 replaced by Lemmas \ref{est_free1}, \ref{est_free2}, and \ref{est_free3}, respectively. Note that in the closed loop setting, the analogous result includes a factor of $\kappa_3\neq0$ arising from the closed-loop moving frame. Unlike the Bishop frame \eqref{bishop_ODE}, the closed-loop moving frame is required to be periodic in $s$, and therefore the moving frame ODE must include a non-zero coefficient relating $\be_{n_1}(\varphi)$ and $\be_{n_2}(\varphi)$. In the free end case, we repeat the same computations as in the closed loop setting, but take $\kappa_3=0$. In doing so, we obtain
\begin{align*}
\abs{(\p_\varphi \bu^{\SB})_1} &\le C \norm{\bm{f}}_{\mc{C}(-1,1)}(\epsilon a)^{-1} \\
\abs{\frac{\p}{\p\theta}(\p_\varphi \bu^{\SB})_1} &\le C \big(\norm{\bm{f}}_{\mc{C}^1(-1,1)}+\norm{\bm{f}}_{\mc{C}_a(-1,1)}a^{-1} \big),
\end{align*}
where the first $C$ depends on $c_\Gamma$ and $\kappa_{\max}$, and the second $C$ depends on $c_\Gamma$, $\kappa_{\max}$, $\delta_a$, $c_a$, $c_\varphi$, and $c_\eta$, but not on $\epsilon$. \\

Next we bound $(\p_\varphi \bu^{\SB})_2$. Using Lemma \ref{est_free1}, we have
\begin{align*}
\abs{(\p_\varphi \bu^{\SB})_2} &\le \frac{3}{8\pi} \epsilon^2a \bar c_a \norm{\bm{f}}_{\mc{C}_a(-1,1)}\int_{\varphi(s)-1}^{\varphi(s)+1} \bigg(3\frac{\bars^2+C\abs{\bars}^3}{\abs{\bm{R}}^5}+\frac{\abs{\bars}}{\abs{\bm{R}}^4} \bigg) d\bars \le C\norm{\bm{f}}_{\mc{C}_a(-1,1)} a^{-1},
\end{align*}
where $C$ depends on $c_{\Gamma}$, $\kappa_{\max}$, and $\bar c_a$. \\

Finally, we bound $\p /\p\theta(\p_\varphi \bu^{\SB})_2$. Using \eqref{deriv_IDs} and Lemma \ref{est_free1}, we have
\begin{align*}
\abs{\frac{\p}{\p\theta}(\p_\varphi \bu^{\SB})_2} &\le C \epsilon^3a^2\norm{\bm{f}}_{\mc{C}_a(-1,1)}\int_{\varphi(s)-1}^{\varphi(s)+1} \bigg(\frac{\bars^2+\abs{\bars}^3}{\abs{\bm{R}}^6} + \frac{\abs{\bars}}{\abs{\bm{R}}^5} \bigg) \ts d\bars \le C \norm{\bm{f}}_{\mc{C}_a(-1,1)}a^{-1},
\end{align*}
where $C$ depends on $c_{\Gamma}$, $\kappa_{\max}$, and $\bar c_a$. \\

Plugging these estimates into \eqref{deriv_thetas_def} and using \eqref{jac_free} and \eqref{curvature_eq} to bound $\p \wh\kappa/\p\theta$ by $\kappa_{\max}$, we obtain
\begin{align*}
\abs{\frac{\p}{\p\theta}\frac{\p\bu^{\SB}}{\p \varphi}} &\le (1+\epsilon a\kappa_{\max})\abs{\frac{\p \bm{I}^{\SB}}{\p\theta}} + \epsilon a\kappa_{\max}\abs{\bm{I}^{\SB}}  \\
&\le C\big(\norm{\bm{f}}_{\mc{C}^1(-1,1)}+\norm{\bm{f}}_{\mc{C}_a(-1,1)}a^{-1} \big).
\end{align*}
\end{proof}

We now define the following velocity residual, which measures the $\theta$-dependence of the slender body surface velocity $\bu^{\SB}\big|_{\Gamma_\epsilon}$: 
\begin{equation}\label{ur_free}
\bu^{\rm{r}}(\theta,\varphi(s)) := \bu^{\SB}(\epsilon a(\varphi(s)), \theta, \varphi(s)) - \frac{1}{2\pi}\int_0^{2\pi}\bu^{\SB}(\epsilon a(\varphi(s)),\omega,\varphi(s)) \ts d\omega.
\end{equation}

Using the above definition of $\bu^{\rm r}$, we can show the following free end analogue to Proposition 3.9 in the closed loop setting \cite{closed_loop}: 
\begin{proposition}\label{ur_ests_free}
Let $a=a(\varphi(s))$ for $s\in(-1,1)$ be as in Definition \ref{admissible_a} and let the prescribed force $\bm{f}\in \mc{C}^1(-1,1)\cap \mc{C}_a(-1,1)$ satisfy \eqref{decay_condition}. Consider the residual $\bu^{\rm r}$ defined in \eqref{ur_free}. For sufficiently small $\epsilon$, we have 
\begin{align}
\label{urest_free}
\abs{\bu^{\rm r}}&\le C\big(\norm{\bm{f}}_{\mc{C}^1(-1,1)}\epsilon \abs{\log\epsilon}+\norm{\bm{f}}_{\mc{C}_a(-1,1)}\epsilon \big),\\
\label{urtheta_free}
\abs{\frac{1}{\epsilon a}\frac{\p \bu^{\rm r}}{\p \theta}}&\le C\big( \norm{\bm{f}}_{\mc{C}^1(-1,1)}\abs{\log(\epsilon a)} +\norm{\bm{f}}_{\mc{C}_a(-1,1)}a^{-1}\big),\\
\label{urs_free}
\abs{\frac{\p \bu^{\rm r}}{\p \varphi}}&\le C\big(\norm{\bm{f}}_{\mc{C}^1(-1,1)} +\norm{\bm{f}}_{\mc{C}_a(-1,1)}a^{-1} \big),
\end{align}
where the constants $C$ depend on $c_\Gamma$, $\kappa_{\max}$, $\delta_a$, $c_a$, $c_\varphi$, and $c_\eta$, but not on $\epsilon$. 
\end{proposition}

\begin{proof}
The proof exactly follows the proof of Proposition 3.9 in the closed loop setting \cite{closed_loop}, but uses Propositions \ref{theta_free} and \ref{theta_s_free} instead of the closed loop analogues (Propositions 3.7 and 3.8, respectively,  in \cite{closed_loop}). 
\end{proof}

Finally, we show the following bound for the residual between the slender body approximation $\bu^{\SB}\big|_{\Gamma_\epsilon}$ given by \eqref{SBT_free} and the centerline approximation $\bu^{\SB}_{\rm C}$ given by \eqref{SBT_asymp_free}. Note that the following estimate holds along the effective centerline ($s\in(-1,1)$).

\begin{proposition}\label{center_prop}
Let $\bu^{\SB}(s,\theta)$ denote \eqref{SBT_free} evaluated along the slender body surface $\Gamma_\epsilon$ for $s\in(-1,1)$, and let $\bu^{\SB}_{\rm C}(s)$ be the slender body centerline equation \eqref{SBT_asymp_free}. Then the residual $|\bu^{\SB}(s,\theta)-\bu^{\SB}_{\rm C}(s)|$ satisfies 
\begin{equation}\label{center_residual}
\abs{\bu^{\SB}(s,\theta)-\bu^{\SB}_{\rm C}(s)} \le C\big( \epsilon\abs{\log\epsilon}\norm{\bm{f}}_{\mc{C}^1(-1,1)} +  \epsilon\norm{\bm{f}}_{\mc{C}_a(-1,1)} \big),
\end{equation}
where $C$ depends on $c_\Gamma$, $\kappa_{\max}$, $c_\eta$, $\delta_a$, $c_a$, $\bar c_a$, and $c_\varphi$, but not on $\epsilon$.
\end{proposition}

\begin{proof}
Recalling the definition of $\ws$ and $\wR$ (see \eqref{RC_def}, \eqref{tR_def}), for $s\in(-1,1)$, we rewrite the Stokeslet term of $\bu^{\SB}(s,\theta)$ \eqref{SBT_free} as
\begin{equation}
\begin{aligned}
\int_{s-1}^{s+1}&\mc{S}(\wR)\bm{f}(s-\ws)d\ws = \mc{S}_1 + \mc{S}_2; \\
\mc{S}_1 &:= \int_{s-1}^{s+1} \frac{\bm{f}(s-\ws)}{|\wR|} d\ws, \quad \mc{S}_2 := \int_{s-1}^{s+1} \frac{\wR\wR^{\rm T}}{|\wR|^3} \bm{f}(s-\ws) d\ws.
\end{aligned}
\end{equation}

Defining
\begin{equation}\label{cent_int1}
\bm{J}_{\mc{S},1} = \int_{s-1}^{s+1} \bigg(\frac{ \bm{f}(s-\ws)}{\abs{\bm{R}_{\rm C}}}- \frac{\bm{f}(s)}{\abs{\ws}} \bigg) d\ws - \bm{f}(s)\log(\epsilon^2),
\end{equation}
we have that Lemma \ref{center_lem_free} implies
\begin{align*}
\abs{\mc{S}_1 - \bm{J}_{\mc{S},1}} &\le \epsilon \abs{\log\epsilon}C \norm{\bm{f}}_{\mc{C}^1(-1,1)}.
\end{align*}

Similarly, to estimate $\mc{S}_2$, we define
\begin{equation}\label{cent_int2}
\bm{J}_{\mc{S},2} = \int_{s-1}^{s+1} \bigg(\frac{ \bm{R}_{\rm C}\bm{R}_{\rm C}^{\rm T}}{\abs{\bm{R}_{\rm C}}^3}\bm{f}(s-\ws) - \frac{\be_{\rm t}(s)\be_{\rm t}(s)^{\rm T}}{\abs{\ws}}\bm{f}(s) \bigg) d\ws + \big[L(s) - 2\big]\be_{\rm t}(s)(\be_{\rm t}(s)\cdot\bm{f}(s)),
\end{equation}
where $L(s) = -\log\big( \frac{2(1-s^2)+2\sqrt{(1-s^2)^2+\epsilon^2a^2(s)}}{\epsilon^2a^2(s)}\big)$. We then have that Lemma \ref{center_lem_free} implies 
 \begin{align*}
\bigg|\mc{S}_2- \bm{J}_{\mc{S},2}& - \int_{s-1}^{s+1} \frac{\epsilon^2a^2(s)\be_\rho\be_{\rho}(s,\theta)^{\rm T}}{|\wR|^3} \bm{f}(s-\ws)d\ws \bigg| \le C\epsilon\abs{\log\epsilon}\norm{\bm{f}}_{\mc{C}^1(-1,1)} \\
& + C\epsilon a(s)\int_{s-1}^{s+1} \frac{ \ws^2+ \abs{\ws}}{|\wR|^3}\abs{\bm{f}}d\ws + C\norm{\bm{f}}_{\mc{C}(-1,1)}\int_{s-1}^{s+1} \big(|\ws|^3 +\ws^4 \big)\bigg|\frac{1}{|\wR|^3} - \frac{1}{\abs{\bm{R}_{\rm C}}^3} \bigg| d\ws  \\
&\le C\epsilon\abs{\log\epsilon}\norm{\bm{f}}_{\mc{C}^1(-1,1)} + C\epsilon a(s)\norm{\bm{f}}_{\mc{C}(-1,1)}\int_{s-1}^{s+1}\frac{\epsilon a(s)+ \ws^2+ \epsilon a(s)\abs{\ws}+|\ws|^3}{\ws^2+\epsilon^2a^2(s)} d\ws \\
&\le C\epsilon\abs{\log\epsilon}\norm{\bm{f}}_{\mc{C}^1(-1,1)}.
\end{align*}
Here we have used \eqref{IR_free} and \eqref{lower_bds} to estimate $\big| |\wR|^{-3}-\abs{\bm{R}_{\rm C}}^{-3} \big|$. Then, using Lemma \ref{est_free3}, we have
\begin{align*}
\big|\mc{S}_2- \bm{J}_{\mc{S},2} - 2\be_\rho(s,\theta)(\be_{\rho}(s,\theta)\cdot\bm{f}(s)) \big|  &\le C\big( \epsilon\abs{\log\epsilon}\norm{\bm{f}}_{\mc{C}^1(-1,1)} +  \epsilon\norm{\bm{f}}_{\mc{C}_a(-1,1)} \big).
\end{align*}

Adding the estimates for $\mc{S}_1$ and $\mc{S}_2$, the Stokeslet term of \eqref{SBT_free} satisfies
\begin{equation}\label{center_stokeslets}
\begin{aligned}
\bigg|\int_{s-1}^{s+1}&\mc{S}(\wR)\bm{f}(s-\ws)d\ws - \bm{J}_{\mc{S},1}-\bm{J}_{\mc{S},2}  - 2\be_\rho(\be_{\rho}\cdot\bm{f}(s)) \bigg| \\
 &\le C\big( \epsilon\abs{\log\epsilon}\norm{\bm{f}}_{\mc{C}^1(-1,1)} +  \epsilon\norm{\bm{f}}_{\mc{C}_a(-1,1)} \big).
\end{aligned}
\end{equation}

Next we estimate the doublet terms of \eqref{SBT_free}, which, for $s\in(-1,1)$, we rewrite as 
\begin{equation}
\begin{aligned}
\int_{s-1}^{s+1} \bigg(&\frac{\epsilon^2a^2(s)}{2}\mc{D}(\wR)+ \frac{\epsilon^2\ws A}{2}\mc{D}(\wR) \bigg)\bm{f}(s-\ws) \ts d\ws = \mc{D}_1 + \mc{D}_2 + \mc{D}_3; \\
\mc{D}_1&:= \frac{\epsilon^2a^2(s)}{2}\int_{s-1}^{s+1} \frac{\bm{f}(s-\ws)}{|\wR|^3} \ts d\ws, \quad \mc{D}_2 := -\frac{3\epsilon^2a^2(s)}{2}\int_{s-1}^{s+1} \frac{\wR\wR^{\rm T}}{|\wR|^5} \bm{f}(s-\ws) \ts d\ws, \\
\mc{D}_3 &:= \int_{s-1}^{s+1} \frac{\epsilon^2\ws A}{2}\bigg(\frac{{\bf I}}{|\wR|^3} -3\frac{\wR\wR^{\rm T}}{|\wR|^5} \bigg)\bm{f}(s-\ws) \ts d\ws.
\end{aligned}
\end{equation}

First note that Lemma \ref{est_free3} implies
\begin{align*}
\abs{\mc{D}_1- \bm{f}(s)} \le C \big(\epsilon a \norm{\bm{f}}_{\mc{C}^1(-1,1)}+\epsilon \norm{\bm{f}}_{\mc{C}_a(-1,1)} \big).
\end{align*}

Next, we have that $\D_2$ satisfies 
\begin{align*}
\big| \mc{D}_2&+ \big(\be_{\rm t}(s)\be_{\rm t}(s)^{\rm T}+2\be_\rho(s,\theta)\be_\rho(s,\theta)^{\rm T} \big) \bm{f}(s) \big| \\
&\le  \abs{\frac{3\epsilon^2a^2(s)}{2}\int_{s-1}^{s+1} \frac{\ws^2\be_{\rm t}\be_{\rm t}(s)^{\rm T}}{|\wR|^5} \bm{f}(s-\ws)d\ws + \be_{\rm t}(s)(\be_{\rm t}\cdot\bm{f}(s))} \\
&\quad +\abs{\frac{3\epsilon^4a^4(s)}{2}\int_{s-1}^{s+1} \frac{\be_\rho\be_\rho^{\rm T}}{|\wR|^5} \bm{f}(s-\ws)d\ws + 2\be_\rho(s,\theta)(\be_\rho\cdot\bm{f}(s))} +C \epsilon a\norm{\bm{f}}_{\mc{C}(-1,1)} \\
&\le  C\big(\epsilon a \norm{\bm{f}}_{\mc{C}^1(-1,1)}+\epsilon \norm{\bm{f}}_{\mc{C}_a(-1,1)} \big).
\end{align*}
Here we have used \eqref{R_expand} along with Lemma \ref{est_free1} in the first inequality, and Lemma \ref{est_free3} in the second inequality. \\

Finally, by Lemma \ref{est_free1} and \eqref{decay_condition}, we have
\begin{align*}
\abs{\mc{D}_3} \le C\epsilon^2\int_{s-1}^{s+1} \frac{\abs{\ws}\abs{\bm{f}}}{|\wR|^3} d\ws \le C\epsilon\norm{\bm{f}}_{\mc{C}_a(-1,1)}.
\end{align*}

Then, defining
\begin{equation}\label{JD1_def}
\bm{J}_{\mc{D},1} = ({\bf I}-\be_{\rm t}(s)\be_{\rm t}(s)^{\rm T})\bm{f}(s),
\end{equation}
the doublet terms of \eqref{SBT_free} satisfy
\begin{equation}\label{center_doublets}
\begin{aligned}
\bigg| \int_{s-1}^{s+1}& \bigg(\frac{\epsilon^2a^2(s)}{2}\mc{D}(\wR)+ \frac{\epsilon^2\ws A}{2}\mc{D}(\wR) \bigg)\bm{f}(s-\ws) \ts d\ws -\bm{J}_{\mc{D},1}  + 2\be_\rho(\be_\rho\cdot\bm{f}(s)) \bigg| \\
& \le  C \big(\epsilon a \norm{\bm{f}}_{\mc{C}^1(-1,1)}+\epsilon \norm{\bm{f}}_{\mc{C}_a(-1,1)} \big). 
\end{aligned}
\end{equation}

Combining \eqref{center_stokeslets} and \eqref{center_doublets}, we have that, for $s\in(-1,1)$, $\bu^{\SB}\big|_{\Gamma_\epsilon}$ satisfies
\begin{equation}\label{uSB_center_est}
\abs{\bu^{\SB}(s,\theta) - \bm{J}_{\mc{S},1}- \bm{J}_{\mc{S},2}- \bm{J}_{\mc{D},1}} \le C\big( \epsilon\abs{\log\epsilon}\norm{\bm{f}}_{\mc{C}^1(-1,1)} +  \epsilon\norm{\bm{f}}_{\mc{C}_a(-1,1)} \big).
\end{equation}

Noting that $\bm{J}_{\mc{S},1}+ \bm{J}_{\mc{S},2}+ \bm{J}_{\mc{D},1}=\bu^{\SB}_{\rm C}(s)$ \eqref{SBT_asymp_free}, we obtain Proposition \ref{center_prop}.
\end{proof}

\subsection{Slender body force residual}\label{force_resid}
It remains to estimate the slender body approximation to the total surface force, given by
\begin{equation}\label{force_free}
{\bm f}^{\SB}(s) = \int_0^{2\pi} \big(2\E(\bu^{\SB})-p^{\SB}{\bf I}\big){\bm n}\big|_{(\varphi(s),\theta)} \mc{J}_\epsilon(\varphi(s),\theta) \varphi'(s) \ts d\theta
\end{equation}
for $\bu^{\SB}$ and $p^{\SB}$ as defined in \eqref{SBT_free} and \eqref{SBp_free} and the free-endpoint Jacobian factor $\mc{J}_\epsilon$ given by \eqref{jac_free}. As in the velocity residual estimate, we will need to prescribe a force $\bm{f}\in \mc{C}^1(-1,1)\cap \mc{C}_a(-1,1)$ \eqref{decay_condition} in order to obtain a $\mc{C}^0$ bound for the force residual $\bm{f}-\bm{f}^{\SB}$. \\

For the force estimate, however, we encounter more differences from the closed loop setting than for the velocity estimate. One such difference is that, unlike in the closed loop setting \cite{closed_loop}, where the surface normal ${\bm n}$ was simply the vector $\be_{\rho}$, the variation in radius $\epsilon a(\varphi)$ along the length of the slender body means that the free end surface normal vector is more complicated. Accounting for variation in radius $\epsilon a(\varphi(s))$ along the length of the fiber, the unit normal to $\Gamma_\epsilon$, directed into the slender body, is given by 
\begin{equation}\label{normal_to_gamma} 
{\bm n}(\varphi,\theta) = -\frac{1}{\sqrt{1+\epsilon^2(a')^2}} \be_{\rho}(\varphi,\theta) + \frac{\epsilon a'(\varphi)}{\sqrt{1+\epsilon^2(a')^2}} \be_{\rm t}(\varphi).
\end{equation}

Then, with respect to the Bishop frame \eqref{bishop_ODE} about the slender body centerline, the strain rate in the direction normal to the slender body surface can be expressed as 
\begin{equation}\label{strain_rate_expr}
\begin{aligned}
2\E(\bu){\bm n}|_{\Gamma_\epsilon} &= \frac{1}{\sqrt{1+\epsilon^2(a')^2}} \big(\E_\rho(\bu) + \epsilon a' \E_t(\bu) \big); \\
\E_\rho(\bu) &:= -\frac{\p \bu}{\p \rho}- \bigg(\frac{\p\bu}{\p\rho}\cdot\be_{\rho}\bigg)\be_{\rho}- \frac{1}{\epsilon a} \bigg(\frac{\p\bu}{\p\theta}\cdot\be_{\rho}\bigg)\be_{\theta} -\frac{1}{1-\epsilon a\wh\kappa} \bigg(\frac{\p\bu}{\p \varphi}\cdot\be_{\rho}\bigg)\be_{\rm t},\\
\E_t(\bu)&:= \frac{1}{1-\epsilon a\wh\kappa}\frac{\p \bu}{\p \varphi}+ \frac{1}{1-\epsilon a\wh\kappa}\bigg(\frac{\p \bu}{\p \varphi}\cdot\be_{\rm t} \bigg)\be_{\rm t} +\bigg(\frac{\p\bu}{\p \rho}\cdot\be_{\rm t}\bigg)\be_{\rho} +\frac{1}{\epsilon a}\bigg( \frac{\p\bu}{\p\theta}\cdot\be_{\rm t}\bigg)\be_{\theta}. 
\end{aligned}
\end{equation}

For future reference, we write down the full expressions for the components of $\E(\bu^{\SB})\bm{n}\big|_{\Gamma_\epsilon}$. Using \eqref{SBT_free} and the identities \eqref{deriv_IDs}, and considering $\bm{R}$ as a function of $\bars$ and $s$ rather than $t$ and $s$, we have 
\begin{equation}\label{uSB_derivs1}
\begin{split}
\frac{\p \bu^{\SB}}{\p\rho} &= (\p_\rho \bu^{\SB})_1 + (\p_\rho \bu^{\SB})_2; \\
(\p_\rho \bu^{\SB})_1 &:= \frac{1}{8\pi}\int_{\varphi(s)-1}^{\varphi(s)+1} \bigg[\frac{\bm{R}_0\cdot\be_{\rho}+\epsilon a}{|{\bm R}|^3} - \frac{\be_{\rho}{\bm R}^{\rm T}+{\bm R}\be_{\rho}^{\rm T}}{|{\bm R}|^3} +\frac{3\bm{R}\bm{R}^{\rm T}(\bm{R}_0\cdot\be_{\rho}+\epsilon a)}{|{\bm R}|^5} \\
&\hspace{2cm} +\frac{3(\epsilon a)^2}{2}\bigg(\frac{\bm{R}_0\cdot\be_{\rho}+\epsilon a}{|{\bm R}|^5}+ \frac{\be_{\rho}\bm{R}^{\rm T}+\bm{R}\be_{\rho}^{\rm T}}{|\bm{R}|^5} \\
&\hspace{4cm} -\frac{5\bm{R}\bm{R}^{\rm T}(\bm{R}_0\cdot\be_{\rho}+\epsilon a)}{|\bm{R}|^7} \bigg)\bigg] {\bm f}(\varphi(s)-\bars) \ts d\bars, \\
(\p_\rho \bu^{\SB})_2 &:= \frac{1}{8\pi}\int_{\varphi(s)-1}^{\varphi(s)+1} \frac{3\epsilon^2 \bars A(\varphi(s),\bars)}{2}\bigg(\frac{\bm{R}_0\cdot\be_{\rho}+\epsilon a}{|{\bm R}|^5}+ \frac{\be_{\rho}\bm{R}^{\rm T}+\bm{R}\be_{\rho}^{\rm T}}{|\bm{R}|^5} \\
&\hspace{3cm} -\frac{5\bm{R}\bm{R}^{\rm T}(\bm{R}_0\cdot\be_{\rho}+\epsilon a)}{|\bm{R}|^7} \bigg){\bm f}(\varphi(s)-\bars) \ts d\bars;
\end{split}
\end{equation}

\begin{equation}\label{uSB_derivs2}
\begin{split}
\frac{1}{\epsilon a}\frac{\p \bu^{\SB}}{\p \theta} &= (\p_\theta \bu^{\SB})_1 + (\p_\theta \bu^{\SB})_2, \\
(\p_\theta \bu^{\SB})_1 &:= \frac{1}{8\pi}\int_{\varphi(s)-1}^{\varphi(s)+1} \bigg[\frac{\bm{R}_0\cdot\be_{\theta}}{|\bm{R}|^3} - \frac{\be_{\theta}\bm{R}^{\rm T}+\bm{R}\be_{\theta}^{\rm T}}{|\bm{R}|^3} +\frac{3\bm{R}\bm{R}^{\rm T}(\bm{R}_0\cdot\be_{\theta})}{|\bm{R}|^5} \\
&\qquad +\frac{3(\epsilon a)^2}{2}\bigg(\frac{\bm{R}_0\cdot\be_{\theta}}{|\bm{R}|^5} + \frac{\be_{\theta}\bm{R}^{\rm T}+\bm{R}\be_{\theta}^{\rm T}}{|\bm{R}|^5}  -\frac{5\bm{R}\bm{R}^{\rm T}(\bm{R}_0\cdot\be_{\theta})}{|\bm{R}|^7} \bigg)\bigg] \bm{f}(\varphi(s)-\bars)\ts d\bars, \\
(\p_\theta \bu^{\SB})_2 &:= \frac{1}{8\pi}\int_{\varphi(s)-1}^{\varphi(s)+1} \frac{3\epsilon^2 \bars A(\varphi(s),\bars)}{2}\bigg(\frac{\bm{R}_0\cdot\be_{\theta}}{|\bm{R}|^5} + \frac{\be_{\theta}\bm{R}^{\rm T}+\bm{R}\be_{\theta}^{\rm T}}{|\bm{R}|^5} \\
&\hspace{3cm} -\frac{5\bm{R}\bm{R}^{\rm T}(\bm{R}_0\cdot\be_{\theta})}{|\bm{R}|^7} \bigg) \bm{f}(\varphi(s)-\bars)\ts d\bars;
\end{split}
\end{equation}

\begin{equation}\label{uSB_derivs3}
\begin{split}
\frac{1}{1-\epsilon a\wh\kappa}\frac{\p\bu^{\SB}}{\p \varphi} &= (\p_\varphi \bu^{\SB})_1 +(\p_\varphi \bu^{\SB})_2, \\
(\p_\varphi \bu^{\SB})_1 :=& \frac{1}{8\pi}\int_{\varphi(s)-1}^{\varphi(s)+1} \bigg[\frac{\bm{R}_0\cdot\be_{\rm t}}{|\bm{R}|^3} - \frac{\be_{\rm t}\bm{R}^{\rm T}+\bm{R}\be_{\rm t}^{\rm T}}{|\bm{R}|^3} +\frac{3\bm{R}\bm{R}^{\rm T}(\bm{R}_0\cdot\be_{\rm t})}{|\bm{R}|^5} \\
&\quad +\frac{3(\epsilon a)^2}{2}\bigg(\frac{\bm{R}_0\cdot\be_{\rm t}}{|\bm{R}|^5} + \frac{\be_{\rm t}\bm{R}^{\rm T}+\bm{R}\be_{\rm t}^{\rm T}}{|\bm{R}|^5} -\frac{5\bm{R}\bm{R}^{\rm T}(\bm{R}_0\cdot\be_{\rm t})}{|\bm{R}|^7} \bigg)\bigg]\bm{f}(\varphi(s)-\bars) \ts d\bars, \\
(\p_\varphi \bu^{\SB})_2 :=& \frac{1}{8\pi}\int_{\varphi(s)-1}^{\varphi(s)+1} \frac{3\epsilon^2 \bars A(\varphi(s),\bars)}{2} \bigg(\frac{\bm{R}_0\cdot\be_{\rm t}}{|\bm{R}|^5} + \frac{\be_{\rm t}\bm{R}^{\rm T}+\bm{R}\be_{\rm t}^{\rm T}}{|\bm{R}|^5}\\
&\hspace{3cm} -\frac{5\bm{R}\bm{R}^{\rm T}(\bm{R}_0\cdot\be_{\rm t})}{|\bm{R}|^7} \bigg)\bm{f}(\varphi(s)-\bars) \ts d\bars.
\end{split}
\end{equation}
Here we use
\begin{equation}\label{R0_0}
\bm{R}_0(\varphi(s),\bars) = \X(\varphi(s)) - \X(\varphi(s)-\bars) = \bars\be_{\rm t}(\varphi(s))+\bars^2\bm{Q}(\varphi(s),\bars); \quad \abs{\bm{R}_0}\le \abs{\bars}+ C\bars^2,
\end{equation}
to distinguish from the $\bm{R}_{\rm C}= \X(s) - \X(t)$, which extends only along the effective centerline, $s\in (-1,1)$. \\


Given the form \eqref{strain_rate_expr} of the normal strain rate, it will be more convenient to consider the Jacobian factor $\mc{J}_\epsilon$ \eqref{jac_free} in the following way. For $\varphi(s)\in (-\eta_\epsilon,\eta_\epsilon)$ and $a=a(\varphi(s))$ as in Definition \ref{admissible_a}, we have that $\mc{J}_\epsilon$ satisfies 
\begin{equation}\label{J_bound}
\abs{\mc{J}_\epsilon(\varphi(s),\theta) - \epsilon a \sqrt{1+\epsilon^2(a')^2}} \le c_J (\epsilon a)^2,
\end{equation}
where $c_J$ depends only on $\kappa_{\max}$. The bound \eqref{J_bound} follows from noting that
\begin{align*}
\big|\mc{J}_\epsilon(\varphi(s),\theta) - \epsilon a\sqrt{1+\epsilon^2(a')^2}\big| &= \bigg| \frac{\epsilon^2 a^2\big[(1-\epsilon a\wh\kappa)^2+\epsilon^2 (a')^2\big] -\epsilon^2 a^2(1+\epsilon^2(a')^2) }{\epsilon a \sqrt{(1-\epsilon a\wh\kappa)^2+\epsilon^2 (a')^2} + \epsilon a \sqrt{1+\epsilon^2(a')^2} } \bigg| \\
&\le \epsilon a \big|(1-\epsilon a\wh\kappa)^2+\epsilon^2 (a')^2 - 1 -\epsilon^2 (a')^2\big| \le 3 \kappa_{\max}(\epsilon a)^2.
\end{align*}

We then define
\begin{equation}\label{Fdef}
\wt{\bm F}(s) := \int_0^{2\pi} \bigg(\E_\rho(\bu^{\SB}) + \epsilon a' \E_t(\bu^{\SB}) -(-p^{\SB}\be_\rho + \epsilon a' p^{\SB}\be_{\rm t}\big) \bigg)\bigg|_{(\varphi(s),\theta)} \epsilon a \ts d\theta,
\end{equation}
where $\E_\rho(\cdot)$ and $\E_t(\cdot)$ are as defined in \eqref{strain_rate_expr}. \\

We can show that the expression \eqref{Fdef} for $\wt{\bm F}$ is in fact close to $\bm{f}^{\SB}$, which will greatly simplify our calculation of the slender body force residual. 
\begin{proposition}\label{fandF}
Let the prescribed force $\bm{f}\in \mc{C}_a(-1,1)$ satisfy \eqref{decay_condition} and let $\bm{f}^{\SB}$ be as in \eqref{force_free} and $\wt{\bm F}$ as in \eqref{Fdef}. Then the difference $\abs{\bm{f}^{\SB}(s)-\wt{\bm F}(s)}$ satisfies
\begin{equation}\label{fandF_est}
\abs{\bm{f}^{\SB}(s)-\wt{\bm F}(s)} \le \epsilon C \norm{\bm{f}}_{\mc{C}_a(-1,1)},
\end{equation}
where the constant $C$ depends on $\kappa_{\max}$, $c_{\Gamma}$, $c_{\varphi}$, $c_a$, $c_\eta$, $c_{\eta,0}$, and $\bar c_a$, but not on $\epsilon$. 
\end{proposition}

\begin{proof}
We have
\begin{equation}\label{f_diff}
\begin{aligned}
\abs{\bm{f}^{\SB}(s)- \wt{\bm F}(s)} &\le \int_0^{2\pi}\abs{2\E(\bu^{\SB}){\bm n}-p^{\SB}{\bm n}}\bigg(\abs{\mc{J}_\epsilon(\varphi(s),\theta) - \epsilon a\sqrt{1+\epsilon^2\dot a^2}}\\
&\hspace{3cm} + \abs{\mc{J}_\epsilon(\varphi(s),\theta)}\abs{\varphi'(s)-1}\bigg)\ts d\theta \\
&\le \epsilon^2 C \big(\abs{2\E(\bu^{\SB})}+ \abs{p^{\SB}}\big)\big( a^2 + \epsilon a+ (\epsilon a)^2+ \epsilon^2\big),
\end{aligned}
\end{equation}
where we used \eqref{J_bound} and Definition \ref{varphi_def}, as well as the bound \eqref{bar_ca}.  \\

Now, using the $\bu^{\SB}$ derivative expressions \eqref{uSB_derivs1} - \eqref{uSB_derivs3} in the strain rate expression \eqref{strain_rate_expr}, and recalling the form of $\bm{R}_0$ \eqref{R0_0}, we have
\begin{align*}
\abs{2\E(\bu^{\SB})} &\le C\bigg(\abs{\frac{\p \bu^{\SB}}{\p \rho}} + \frac{1}{\epsilon a} \abs{\frac{\p\bu^{\SB}}{\p\theta}} + \frac{1}{\abs{1-\epsilon a\wh\kappa}} \abs{\frac{\p \bu^{\SB}}{\p \varphi}}\bigg) \\
&\le aC\norm{\bm{f}}_{\mc{C}_a(-1,1)}\int_{\varphi(s)-1}^{\varphi(s)+1} \bigg[\frac{\abs{\bm{R}_0}}{\abs{\bm{R}}^3} + \frac{1}{\abs{\bm{R}}^2}+\frac{\epsilon a}{\abs{\bm{R}}^3} \\
&\quad + (\epsilon a)^2\bigg(\frac{\abs{\bm{R}_0}}{\abs{\bm{R}}^5} +\frac{1}{\abs{\bm{R}}^4}+ \frac{\epsilon a}{\abs{\bm{R}}^5}\bigg) + C\epsilon^2 \bigg(\frac{\abs{\bars \bm{R}_0}}{\abs{\bm{R}}^5} +\frac{\abs{\bars}}{\abs{\bm{R}}^4}+ \frac{\epsilon a\abs{\bars}}{\abs{\bm{R}}^5}\bigg) \bigg] \ts d\bars \\
&\le \epsilon^{-1}C\norm{\bm{f}}_{\mc{C}_a(-1,1)} .
\end{align*}
Here we have used Lemma \ref{est_free1} to estimate the first six terms in the integral expression, and Lemma \ref{est_free1_new} to estimate the last three terms. \\ 

Furthermore, 
\begin{equation}\label{pterm_est}
 \abs{p^{\SB}} \le aC\norm{\bm{f}}_{\mc{C}_a(-1,1)}\int_{\varphi(s)-1}^{\varphi(s)+1} \frac{1}{|\bm{R}|^2} \ts d\bars \le C\norm{\bm{f}}_{\mc{C}_a(-1,1)}\epsilon^{-1},
 \end{equation}
by Lemma \ref{est_free1}. \\

Plugging these estimates for $\abs{2\E(\bu^{\SB})}$ and $ \abs{p^{\SB}}$ into \eqref{f_diff}, we obtain \eqref{fandF_est}.
\end{proof}

Thus, to obtain a residual estimate between the prescribed force $\bm{f}$ and the slender body approximation $\bm{f}^{\SB}$, it suffices to compare $\bm{f}$ and $\wt{\bm F}$. \\
As in the closed loop setting, we can rely on $\theta$-integration to obtain improved bounds for the force residual, removing a factor of $\log\epsilon$ that would arise from only using Lemmas \ref{est_free1} - \ref{est_free3}. We show the following free end analogue of Lemma 3.12 from \cite{closed_loop}. 
\begin{lemma}\label{est_free4}
Let $\bm{R}$ be as in \eqref{R_expand} for $\varphi(s)\in (-\eta_\epsilon,\eta_\epsilon)$ and let $g\in \mc{C}(-1,1)$. Suppose $m$ is a non-negative integer and $n=m+1$ or $n=m+2$. Then for $k\in \Z$, $k\neq 0$, $\theta_0\in \R$, and $\epsilon$ sufficiently small, we have 
\begin{equation}
\abs{\int_0^{2\pi}\int_{\varphi(s)-1}^{\varphi(s)+1} \frac{\bars^m g(\varphi(s)-\bars)}{\abs{\bm{R}}^n} \cos(k(\theta+\theta_0)) \ts d\bars d\theta} \le C\begin{cases}
  \epsilon\abs{\log\epsilon} \norm{g}_{\mc{C}(-1,1)}, & n=m+1 \\
 \norm{g}_{\mc{C}(-1,1)}, & n=m+2,
 \end{cases},
\end{equation}
where the constant $C$ depends on $n$, $c_{\Gamma}$, and $\kappa_{\max}$, but not on $\epsilon$. 
\end{lemma}

Note that the presence of $k\in \Z$ and $\theta_0\in \R$ in the statement of Lemma \ref{est_free4} means that the result applies to $\frac{\bars^m g(\varphi(s)-\bars)}{\abs{\bm{R}}^n}$ integrated against $\sin\theta$ or against odd triples of the form $\sin^j\theta \cos^\ell \theta$, $j+\ell$ odd (see discussion following Lemma 3.12 in \cite{closed_loop}). \\

\begin{proof}
The proof is essentially identical to the closed loop case (see \cite{closed_loop}, Lemma 3.12). We write
\begin{equation}
\begin{aligned}
I &=\int_0^{2\pi}\int_{\varphi(s)-1}^{\varphi(s)+1}\frac{\bars^m g(\varphi(s)-\bars)}{\abs{\bm{R}}^n}\cos(k(\theta+\theta_0))d\bars \ts d\theta \\
&= \int_0^{2\pi}\int_{\varphi(s)-1}^{\varphi(s)+1}\bigg(\frac{1}{\abs{\bm{R}}^n}-\frac{1}{(\abs{\bm{R}_0}^2+(\epsilon a)^2)^{n/2}} \bigg)\bars^m g(\varphi(s)-\bars)\cos(k(\theta+\theta_0))d\bars \ts d\theta,
\end{aligned}
\end{equation}
where we use that both $\bm{R}_0(s,\bars)$ and $\epsilon a(\varphi(s))$ are independent of $\theta$, and thus the second term integrates to zero. We have that the integrand of $I$ satisfies 
\begin{align*}
\bigg|\bigg(\frac{1}{\abs{\bm{R}}^n}& -\frac{1}{(\abs{\bm{R}_0}^2+(\epsilon a)^2)^{n/2}} \bigg)\bars^m g(\varphi(s)-\bars)\cos(k(\theta+\theta_0))\bigg| \\
&\le \frac{\norm{g}_{\mc{C}(-1,1)}\abs{\bars}^m\abs{\abs{\bm{R}}^2-(\abs{\bm{R}_0}^2+(\epsilon a)^2)}}{\abs{\bm{R}}(\abs{\bm{R}_0}^2+(\epsilon a)^2)^{1/2}(\abs{\bm{R}}+(\abs{\bm{R}_0}^2+(\epsilon a)^2)^{1/2})} \sum_{j=0}^{n-1}\frac{1}{\abs{\bm{R}}^j(\abs{\bm{R}_0}^2+(\epsilon a)^2)^{(n-1-j)/2}}.
\end{align*}
Using \eqref{X_expand} and \eqref{R_expand}, we can show
\[ \abs{\abs{\bm{R}}^2-\abs{\bm{R}_0}^2-(\epsilon a)^2}\le C\epsilon a \bars^2.\]
Furthermore, by Lemma \ref{R_ests_free} and \eqref{no_intersect}, we have
\[ \abs{\bm{R}}\ge C\sqrt{\bars^2+(\epsilon a)^2}, \quad \abs{\bm{R}_0}\ge c_\Gamma\abs{\bars}.\]
Altogether we have
\begin{align*}
\abs{I}&\le  \epsilon a C\norm{g}_{\mc{C}(-1,1)}\int_0^{2\pi}\int_{\varphi(s)-1}^{\varphi(s)+1}\frac{\abs{\bars}^{m+2}}{(\bars^2+(\epsilon a)^2)^{(n+2)/2}} \ts d\bars d\theta \le C\begin{cases}
  \epsilon\abs{\log\epsilon} \norm{g}_{\mc{C}(-1,1)}, & n=m+1 \\
 \norm{g}_{\mc{C}(-1,1)}, & n=m+2,
 \end{cases}
\end{align*}
where we have used Lemma \ref{integral_est}.

\end{proof}


Now we are equipped to carry out residual estimates between $\wt{\bm F}$ and the prescribed force $\bm{f}$. We write
\begin{equation}\label{Fdef_split}
\begin{split}
\wt{\bm F}(s) &= \wt{ \bm{F}}_\rho(s) + \wt{\bm{F}}_t(s); \\
\wt{ \bm{F}}_\rho(s)&:=  \epsilon a \int_0^{2\pi} \bigg(\E_\rho(\bu^{\SB}) +p^{\SB}\be_\rho \bigg)\bigg|_{(\varphi(s),\theta)} \ts d\theta \\
\wt{ \bm{F}}_{\rm t}(s)&:= \epsilon^2 (aa') \int_0^{2\pi} \bigg(\E_t(\bu^{\SB}) - p^{\SB}\be_{\rm t} \bigg)\bigg|_{(\varphi(s),\theta)} \ts d\theta,
\end{split}
\end{equation}
where $\E_\rho$ and $\E_t$ are as defined in \eqref{strain_rate_expr}. \\

We note that the component $\wt{ \bm{F}}_\rho(s)$ in fact recovers most of the true slender body force $\bm{f}(s)$. In particular, following the force residual estimate in the closed loop setting \cite{closed_loop}, we easily obtain the following proposition.
\begin{proposition}\label{Frho_est}
Let $\Sigma_\epsilon$ be as in Section \ref{geometry_section} and let $a=a(\varphi(s))$ be as in Definition \ref{admissible_a}. Consider $\bm{f}\in \mc{C}^1(-1,1)\cap \mc{C}_a(-1,1)$ satisfying \eqref{decay_condition}, and $\wt{ \bm{F}}_\rho(s)$ as defined in \eqref{Fdef_split}. Then $\wt{ \bm{F}}_\rho(s)$ satisfies 
\begin{equation}
\abs{\wt{ \bm{F}}_\rho(s)- \bm{f}(s)} \le C \big(\epsilon a \norm{\bm{f}}_{\mc{C}^1(-1,1)} +\epsilon \norm{\bm{f}}_{\mc{C}_a(-1,1)} \big),
\end{equation}
where the constant $C$ depends on $c_\Gamma$, $\kappa_{\max}$, $\delta_a$, $c_a$, $\bar c_a$, $c_\varphi$, and $c_\eta$, but not on $\epsilon$.
\end{proposition}

\begin{proof}
Using \eqref{strain_rate_expr} and the notation of \eqref{uSB_derivs1} - \eqref{uSB_derivs3}, we write
\begin{equation}\label{Frho_split}
\begin{aligned}
\wt{\bm F}_\rho(s) &= F_1 + F_2; \\
F_1 &:= -\epsilon a \int_0^{2\pi}\bigg( (\p_\rho \bu^{\SB})_1+ \big((\p_\rho \bu^{\SB})_1\cdot\be_{\rho}\big)\be_{\rho}+ \big((\p_\theta \bu^{\SB})_1\cdot\be_{\rho}\big)\be_{\theta} \\
&\hspace{3cm} + \big((\p_\varphi \bu^{\SB})_1\cdot\be_{\rho}\big)\be_{\rm t} - p^{\SB}\be_\rho \bigg) d\theta, \\
F_2 &:= -\epsilon a \int_0^{2\pi}\bigg( (\p_\rho \bu^{\SB})_2+ \big((\p_\rho \bu^{\SB})_2\cdot\be_{\rho}\big)\be_{\rho}+ \big((\p_\theta \bu^{\SB})_2\cdot\be_{\rho}\big)\be_{\theta} \\
&\hspace{4cm}+ \big((\p_\varphi \bu^{\SB})_2\cdot\be_{\rho}\big)\be_{\rm t}  \bigg) d\theta.
\end{aligned}
\end{equation}
Here $F_2$ accounts for the additional doublet term resulting from the expansion \eqref{a_expand} of the doublet coefficient. \\

Note that each of the terms arising in the full expression for $F_1$ can be estimated following the proofs of Propositions 3.14 through 3.19 in the closed loop setting \cite{closed_loop}, using Lemmas \ref{est_free1}, \ref{est_free2}, \ref{est_free3}, and \ref{est_free4} in place of the closed loop Lemmas 3.3, 3.4, 3.5, and 3.12, respectively. Doing this, we obtain 
\[ \abs{F_1-\bm{f}(s)} \le C \big(\epsilon a \norm{\bm{f}}_{\mc{C}^1(-1,1)} +\epsilon \norm{\bm{f}}_{\mc{C}_a(-1,1)} \big), \]
where the constant $C$ depends on $c_\Gamma$, $\kappa_{\max}$, $\delta_a$, $c_a$, $\bar c_a$, $c_\varphi$, and $c_\eta$, but not on $\epsilon$. Note that the bulk of the prescribed force is recovered from the term $F_1$. \\

Next we estimate $F_2$. Using Lemma \eqref{est_free1}, we have
\begin{align*}
\abs{F_2} &\le C \epsilon a^2 \norm{\bm{f}}_{\mc{C}_a(-1,1)}\int_{\varphi(s)-1}^{\varphi(s)+1}  \epsilon^2 \bigg(\frac{\abs{\bars \bm{R}_0}}{\abs{\bm{R}}^5} +\frac{\abs{\bars}}{\abs{\bm{R}}^4}+ \frac{\epsilon a\abs{\bars}}{\abs{\bm{R}}^5}\bigg) d\bars \le \epsilon C\norm{\bm{f}}_{\mc{C}_a(-1,1)}.
\end{align*}

Combining the estimates for $F_1$ and $F_2$, we obtain Proposition \ref{Frho_est}.
\end{proof}

Since $\wt{\bm F}_{\rho}(s)$ is close to the true prescribed force $\bm{f}(s)$, it remains only to show that $\wt{ \bm{F}}_{\rm t}$ is small. In particular, we show:
\begin{proposition}\label{Ft_est}
Let $\Sigma_\epsilon$ be a slender body as defined in Section \ref{geometry_section}. Consider $\bm{f}\in \mc{C}_a(-1,1)$ satisfying \eqref{decay_condition}, and $\wt{ \bm{F}}_{\rm t}(s)$ as defined in \eqref{Fdef_split}. Then $\wt{ \bm{F}}_{\rm t}(s)$ satisfies
\begin{equation}
\abs{\wt{ \bm{F}}_{\rm t}(s)} \le \epsilon C \norm{\bm{f}}_{\mc{C}_a(-1,1)},
\end{equation}
where the constant $C$ depends on $c_\Gamma$, $\kappa_{\max}$, $c_a$, $\bar c_a$, $c_{\eta}$, and $c_{\eta,0}$, but not on $\epsilon$.
\end{proposition}

\begin{proof}
Using the form of $\E_t$ \eqref{strain_rate_expr} with the integral expressions \eqref{uSB_derivs1} - \eqref{uSB_derivs3}, along with the form of $\bm{R}_0$ \eqref{R0_0}, we have
\begin{align*}
\abs{\wt{ \bm{F}}_{\rm t}(s)} &\le 2\pi \epsilon^2 \abs{aa'} \big(\abs{\E_t(\bu^{\SB})}+ \abs{p^{\SB}} \big) \\
& \le \epsilon^2aC \norm{\bm{f}}_{\mc{C}_a(-1,1)}\int_{\varphi(s)-1}^{\varphi(s)+1} \bigg[\frac{\abs{\bm{R}_0}}{\abs{\bm{R}}^3}+\frac{1}{\abs{\bm{R}}^2} + \frac{\epsilon a}{\abs{\bm{R}}^3}\\
&\qquad+ (\epsilon a)^2\bigg(\frac{\abs{\bm{R}_0}}{\abs{\bm{R}}^5}+ \frac{1}{\abs{\bm{R}}^4} + \frac{\epsilon a}{\abs{\bm{R}}^5} \bigg) + \epsilon^2\bigg(\frac{\abs{\bars\bm{R}_0}}{\abs{\bm{R}}^5}+ \frac{\abs{\bars}}{\abs{\bm{R}}^4} + \frac{\epsilon a\abs{\bars}}{\abs{\bm{R}}^5} \bigg)\bigg] d\bars \\
&\le \epsilon C \norm{\bm{f}}_{\mc{C}_a(-1,1)}.
\end{align*}
Here we have used Lemma \ref{est_free1} to estimate the first six terms of the integral expression, and Lemma \ref{est_free1_new} to estimate the last three terms. 
\end{proof}

Putting everything together, we obtain the following proposition.
\begin{proposition}\label{free_f_est}
Let $\Sigma_\epsilon$ be a slender body as defined in Section \ref{geometry_section} and let $a=a(\varphi(s))$ be as in Definition \ref{admissible_a}. Consider $\bm{f}\in \mc{C}^1(-1,1)\cap \mc{C}_a(-1,1)$ satisfying \eqref{decay_condition} and consider the corresponding slender body approximation $\bm{f}^{\SB}$. Then
\begin{equation}
\abs{\bm{f}^{\SB}(s)- \bm{f}(s)} \le C \big(\epsilon a \norm{\bm{f}}_{\mc{C}^1(-1,1)} +\epsilon \norm{\bm{f}}_{\mc{C}_a(-1,1)} \big),
\end{equation}
where the constant $C$ depends on $c_\Gamma$, $\kappa_{\max}$, $\delta_a$, $c_a$, $c_\varphi$, $c_\eta$, and $\bar c_a$ but not on $\epsilon$.
\end{proposition}

\begin{proof}
Using Propositions \ref{fandF}, \ref{Frho_est}, and \ref{Ft_est}, we have
\begin{align*}
\abs{\bm{f}(s)-\bm{f}^{\SB}(s)} &\le \abs{\bm{f}(s)-\wt{\bm F}(s)}+ \abs{\wt{\bm{F}}(s)-\bm{f}^{\SB}(s)}  \\
&\le \abs{\bm{f}(s)-\wt{\bm F}(s)}+ \abs{\wt{\bm{F}}_\rho(s)-\bm{f}^{\SB}(s)} + \abs{\wt{\bm F}_t(s)}\\
& \le C \big(\epsilon a \norm{\bm{f}}_{\mc{C}^1(-1,1)} +\epsilon \norm{\bm{f}}_{\mc{C}_a(-1,1)} \big).
\end{align*}
\end{proof}

\section{Free endpoint error estimate}\label{err_est_sec}
Using the above residual estimates for the slender body force $\bm{f}^{\SB}$ and $\theta$-dependence in $\bu^{\SB}\big|_{\Gamma_\epsilon}$, we can use the variational framework established in Section \ref{well_posed} to show the $D^{1,2}$ error estimate of Theorem \ref{free_err_theorem}. 

\subsection{Setup of framework}
Given $\bm{f}\in \mc{C}^1(-1,1)\cap \mc{C}_a(-1,1)$, let $(\bu,p)$ be the true solution to \eqref{free_PDE} along with corresponding stress tensor $\bm{\sigma}= -p{\bf I}+2\E(\bu)$, and let $\bu^{\SB}$, $p^{\SB}$, $\bm{\sigma}^{\SB}$ be the corresponding slender body approximation \eqref{SBT_free}. Define $\bu^{\rm e}= \bu^{\SB}-\bu$, $p^{\rm e}= p^{\SB}-p$, and $\bm{\sigma}^{\rm e} = -p^{\rm e}{\bf I} +2\E(\bu^{\rm e})$. Then $(\bu^{\rm e},p^{\rm e})$ satisfies
\begin{equation}\label{free_err_PDE}
\begin{aligned}
-\Delta \bu^{\rm e} + \nabla p^{\rm e} &= 0 \\
\dive \ts\bu^{\rm e} &= 0 \qquad \text{ in }\Omega_\epsilon \\
\int_0^{2\pi} ( \bm{\sigma}^{\rm e} {\bm n})\big|_{(\varphi(s),\theta)} \ts \mc{J}_\epsilon(\varphi(s),\theta) \varphi'(s) d\theta &= \bm{f}^{\rm e}(s) \quad \text{on }\Gamma_\epsilon \\
\bu^{\rm e} \big|_{\Gamma_\epsilon} &= \bar\bu^{\rm e}\big(\varphi(s)\big) + \bu^{\rm r}(\varphi(s),\theta) \\
\bu^{\rm e} \to 0 \ts &\text{ as } |\bx| \to \infty.
\end{aligned}
\end{equation}
Here, the function $\bar \bu^{\rm e}(\varphi) := (\bu^{\SB} - \bu^{\rm r})\big|_{\Gamma_\epsilon}(\varphi)-\bu\big|_{\Gamma_\epsilon}(\varphi)$ is unknown, since the true boundary value $\bu\big|_{\Gamma_\epsilon}(\varphi)$ is unknown, but is $\theta$-independent, by definition of $\bu^{\rm r}$ \eqref{ur_free}. The residuals $\bm{f}^{\rm e}(s):= \bm{f}^{\SB}(s) - \bm{f}(s)$ and $\bu^{\rm r}(\varphi(s),\theta)$ are both known quantities (see Propositions \ref{free_f_est} and \ref{ur_ests_free}, respectively). \\

As in the closed loop case \cite{closed_loop}, we write \eqref{free_err_PDE} in the following variational form. For any $\bw\in D^{1,2}(\Omega_\epsilon)$, the error $(\bu^{\rm e},p^{\rm e})$ satisfies
\begin{equation}\label{variational_err_free}
\int_{\Omega_\epsilon} \bigg(2\E(\bu^{\rm e}):\E(\bw) - p^{\rm e}\dive\ts \bw \bigg) \ts d\bx=\int_{\Gamma_\epsilon} (\bm{\sigma}^{\rm e}\bm{n})\cdot\bw\ts dS.
\end{equation}

Note that unless the test function $\bw$ additionally satisfies the $\theta$-independence condition $\bw\big|_{\Gamma_\epsilon}=\bw(s)$, i.e. $\bw$ in fact belongs to the function space $\A_\epsilon$ \eqref{Aspace_def}, we cannot simplify the expression \eqref{variational_err_free} any further. In particular, without the $\theta$-independence condition we cannot make use of the force data $\bm{f}^{\rm e}(s)$. This presents a small issue for deriving an error estimate, since, in fact, the function $\bu^{\rm e}$ is not in $\A_\epsilon$. \\

However, we can use \eqref{variational_err_free} and exactly follow the proof of Lemma \ref{free_pressure_ex} to show the following pressure estimate:
\begin{equation}\label{press_err_free}
\norm{p^{\rm e}}_{L^2(\Omega_\epsilon)} \le c_P\norm{\E(\bu^{\rm e})}_{L^2(\Omega_\epsilon)},
\end{equation}
where the constant $c_P$ is as in Lemma \ref{press_est_free}. \\

Following the closed loop arguments, we make a specific choice of $\bw$. In particular, we take 
\begin{equation}\label{w_choice}
\bw = \wt\bu^{\rm e}:= \bu^{\rm e} - \wt\bv,
\end{equation}
where $\wt\bv\in D^{1,2}(\Omega_\epsilon)$ with $\wt\bv\big|_{\Gamma_\epsilon} = \bu^{\rm r}(\varphi(s),\theta)$ is a correction function that we construct explicitly to satisfy a bound in terms of $\norm{\bm{f}}_{\mc{C}^1(-1,1)}$ and $\norm{\bm{f}}_{\mc{C}_a(-1,1)}$. Note that $\wt\bu^{\rm e}\big|_{\Gamma_\epsilon} = \bar\bu^{\rm e}(s)$, and therefore $\wt\bu^{\rm e}$ is in the desired function space $\A_\epsilon$.  \\

Using this $\wt\bu^{\rm e}$ in place of $\bw$ in \eqref{variational_err_free}, we then obtain 
\begin{equation}\label{variational_err_free2}
\int_{\Omega_\epsilon} \bigg(2\E(\bu^{\rm e}):\E(\wt\bu^{\rm e}) - p^{\rm e}\dive\ts \wt\bu^{\rm e} \bigg) \ts d\bx= \int_{-1}^1 \bar\bu^{\rm e}(\varphi(s))\cdot \bm{f}^{\rm e}(s) \ts ds,
\end{equation}
from which we can derive the $D^{1,2}(\Omega_\epsilon)$ error estimate for $(\bu^{\rm e},p^{\rm e})$ stated in Theorem \ref{free_err_theorem}. \\

Before we make use of \eqref{variational_err_free2}, we must construct the correction $\wt\bv$ such that $\nabla\wt\bv$ can be bounded by $\norm{\bm{f}}_{\mc{C}^1(-1,1)}$ and $\norm{\bm{f}}_{\mc{C}_a(-1,1)}$. \\

The true force $\bm{f}(s)$ and the radius function $a(\varphi(s))$ are both at least $\mc{C}^1$, and therefore the $\theta$-dependent component of the slender body velocity $\bu^{\rm r}(\varphi(s),\theta)$ \eqref{ur_free} is also at least $\mc{C}^1$ on $\Gamma_\epsilon$. Thus, as in the closed loop setting \cite{closed_loop}, we can begin to construct the desired $\wt\bv$ by simply extending $\bu^{\rm r}(\varphi(s),\theta)$ radially from $\Gamma_\epsilon$ into the interior of $\Omega_\epsilon$. For $\varphi(s)\in(-\eta_\epsilon,\eta_\epsilon)$, we define
\[ \bu^{\SB}_{\text{ext}}(\rho,\varphi(s),\theta) = \begin{cases}
\bu^{\rm r}(\varphi(s),\theta) & \text{if } \rho<4\epsilon a(\varphi(s)), \\
0 & \text{otherwise}.
\end{cases} \] 

Define $\psi(\rho,\varphi(s))$ to be a smooth cutoff equal to 1 for $\rho<2\epsilon a(\varphi(s))$ and equal to $0$ for $\rho>4\epsilon a(\varphi(s))$ with smooth decay between. We require that this decay satisfies
\begin{equation}\label{free_psi_decay}
\abs{\frac{\p\psi}{\p\rho}} \le \frac{C}{\epsilon a(\varphi(s))}, \quad \abs{\frac{\p \psi}{\p\varphi}} \le \frac{C a'(\varphi(s))}{\epsilon}.
\end{equation}

Then define $\wt\bv$ as 
\begin{equation}\label{vee_def_free}
\wt\bv(\rho,\varphi(s),\theta) = \psi(\rho,\varphi(s))\bu^{\SB}_{\text{ext}}(\rho,\varphi(s),\theta).
\end{equation}

As we will eventually require an $L^2(\Omega_\epsilon)$ estimate of $\nabla\wt\bv$, we note that 
\[\nabla\wt\bv = \psi\nabla\bu^{\SB}_{\text{ext}} + (\nabla\psi)(\bu^{\SB}_\text{ext})^{\rm T}. \]

We have that
\[ \phi\nabla\bu^{\SB}_{\text{ext}} = \frac{\psi}{\epsilon a}\frac{\p\bu^{\rm r}}{\p \theta} \be_\theta^{\rm T} + \frac{\psi}{1-\rho\wh\kappa}\frac{\p \bu^{\rm r}}{\p \varphi}\be_{\rm t}^{\rm T} \]
and
\[(\nabla\phi)(\bu^{\SB}_\text{ext})^{\rm T} = \frac{\p \psi}{\p \rho} \be_\rho(\bu^{\rm r})^{\rm T} + \frac{1}{1-\rho\wh\kappa}\frac{\p \psi}{\p \varphi}\be_{\rm t}(\bu^{\rm r})^{\rm T}. \]

Thus, we can bound $\nabla \wt\bv$ as
\begin{align*}
\abs{\nabla\wt\bv} &\le \abs{\frac{1}{\epsilon a}\frac{\p \bu^{\rm r}}{\p \theta}}+\frac{1}{1-4\epsilon a\kappa_{\max}}\abs{\frac{\p\bu^{\rm r}}{\p\varphi}} + \abs{\frac{\p\psi}{\p\rho}}\abs{\bu^{\rm r}} + \frac{1}{1-4\epsilon a\kappa_{\max}}\abs{\frac{\p\psi}{\p\varphi}}\abs{\bu^{\rm r}} \\
&\le C \big( \abs{\log(\epsilon a)} \norm{\bm{f}}_{\mc{C}^1(-1,1)} + a^{-1} \norm{\bm{f}}_{\mc{C}_a(-1,1)} \big),
\end{align*}
where we recall from Definition \ref{admissible_a} that the product $\abs{aa'}$ is bounded for each $\varphi(s)\in (-\eta_\epsilon,\eta_\epsilon)$.\\

Then, relying on the extent of supp$(\wt\bv)$, we have
\begin{equation}\label{vee_est_free}
\begin{aligned}
\norm{\wt\bv}_{L^2(\Omega_\epsilon)} &= \bigg( \int_{-\eta_\epsilon}^{\eta_\epsilon}\int_0^{2\pi}\int_{\epsilon a}^{4\epsilon a} \abs{\nabla\wt\bv}^2 \rho(1-\rho\wh\kappa) d\rho d\theta d\varphi \bigg)^{1/2} \\
&\le C\bigg( \int_{-\eta_\epsilon}^{\eta_\epsilon}\int_0^{2\pi} 12(\epsilon a)^2 \big( \abs{\log(\epsilon a)} \norm{\bm{f}}_{\mc{C}^1(-1,1)} \\
&\hspace{4cm}+ a^{-1} \norm{\bm{f}}_{\mc{C}_a(-1,1)} \big)^2 (1+4\epsilon a \kappa_{\max} ) d\theta d\varphi \bigg)^{1/2} \\
&\le C \big( \epsilon\abs{\log\epsilon}\norm{\bm{f}}_{\mc{C}^1(-1,1)}+ \epsilon\norm{\bm{f}}_{\mc{C}_a(-1,1)} \big),
\end{aligned}
\end{equation}
where $C$ depends on the constants $c_\Gamma$, $\kappa_{\max}$, $\delta_a$, $c_a$, $c_\varphi$, and $c_\eta$, but not on $\epsilon$. 

\subsection{Error estimate}
With this definition of $\wt\bv$, we may now use \eqref{variational_err_free2} to estimate $\bu^{\rm e}$, closely following the arguments in \cite{closed_loop}. We first rewrite \eqref{variational_err_free2} as
\begin{equation}\label{variational_err_free3}
\begin{aligned}
\int_{\Omega_\epsilon} 2\abs{\E(\bu^{\rm e})}^2 \ts d\bx &= \int_{\Omega_\epsilon} \bigg(2\E(\bu^{\rm e}):\E(\wt\bv) - p^{\rm e} \dive\ts \wt\bv \bigg) d\bx + \int_{-1}^1 \bm{f}^{\rm e}(s)\cdot\bar\bu^{\rm e}(\varphi(s)) \ts ds \\
&\le \abs{\int_{\Omega_\epsilon} 2\E(\bu^{\rm e}):\E(\wt\bv) d\bx}+\abs{\int_{\Omega_\epsilon} p^{\rm e} \dive\ts \wt\bv d\bx } + \abs{\int_{-1}^1 \bm{f}^{\rm e}(s)\cdot\bar\bu^{\rm e}(\varphi(s)) \ts ds}.
\end{aligned}
\end{equation}

Now, the first term on the right hand side of \eqref{variational_err_free3} can be estimated via Cauchy-Schwarz:
\begin{align*}
\abs{\int_{\Omega_\epsilon} 2\E(\bu^{\rm e}):\E(\wt\bv) d\bx} &\le 2\norm{\E(\bu^{\rm e})}_{L^2(\Omega_\epsilon)}\norm{\E(\wt\bv)}_{L^2(\Omega_\epsilon)} \\
&\le \delta\norm{\E(\bu^{\rm e})}_{L^2(\Omega_\epsilon)}^2+ \frac{1}{\delta}\norm{\E(\wt\bv)}_{L^2(\Omega_\epsilon)}^2 \\
&\le \delta\norm{\E(\bu^{\rm e})}_{L^2(\Omega_\epsilon)}^2+ \frac{2}{\delta}\norm{\nabla\wt\bv}_{L^2(\Omega_\epsilon)}^2
\end{align*}
for any $\delta\in\R_+$. \\

Using \eqref{press_err_free} and Cauchy-Schwarz, the second term on the right hand side of \eqref{variational_err_free3} satisfies
\begin{align*}
\abs{\int_{\Omega_\epsilon} p^{\rm e} \dive\ts \wt\bv d\bx } &\le \norm{p^{\rm e}}_{L^2(\Omega_\epsilon)}\norm{\nabla\wt\bv}_{L^2(\Omega_\epsilon)} \\
&\le 4c_P\norm{\E(\bu^{\rm e})}_{L^2(\Omega_\epsilon)}\norm{\nabla\wt\bv}_{L^2(\Omega_\epsilon)} \\
&\le \delta\norm{\E(\bu^{\rm e})}_{L^2(\Omega_\epsilon)}^2 + \frac{4c_P^2}{\delta}\norm{\nabla\wt\bv}_{L^2(\Omega_\epsilon)}^2.
\end{align*}

Finally, the third term in \eqref{variational_err_free3} may be estimated using the trace inequality \eqref{free_trace_ineq}, the Korn inequality \eqref{korn_free}, and Cauchy-Schwarz. Recalling the definition \eqref{first_decay_space} of the space $L^2_a(-1,1)$, we have 
\begin{align*}
\abs{\int_{-1}^1 \bm{f}^{\rm e}(s)\cdot\bar\bu^{\rm e}(\varphi(s)) \ts ds} &\le \norm{\bm{f}^{\rm e}}_{L^2_a(-1,1)}\norm{\bar\bu^{\rm e} \abs{\log(a)}^{-1/2}}_{L^2(-\eta_\epsilon,\eta_\epsilon)} \\
&\le c_T\norm{\nabla\wt\bu^{\rm e}}_{L^2(\Omega_\epsilon)}\norm{\bm{f}^{\rm e}}_{L^2_a(-1,1)}\\
&\le c_Tc_K\norm{\E(\wt\bu^{\rm e})}_{L^2(\Omega_\epsilon)}\norm{\bm{f}^{\rm e}}_{L^2_a(-1,1)}\\
&\le \delta\norm{\E(\wt\bu^{\rm e})}_{L^2(\Omega_\epsilon)}^2+ \frac{c_T^2c_K^2}{4\delta}\norm{\bm{f}^{\rm e}}_{L^2_a(-1,1)}^2\\
&\le \delta\norm{\E(\bu^{\rm e})}_{L^2(\Omega_\epsilon)}^2+ 2\delta\norm{\nabla\wt\bv}_{L^2(\Omega_\epsilon)}^2+ \frac{c_T^2c_K^2}{4\delta}\norm{\bm{f}^{\rm e}}_{L^2_a(-1,1)}^2.
\end{align*}

Taking $\delta=\frac{1}{3}$, by \eqref{variational_err_free3}, we obtain
\begin{equation}\label{variational_err_free4}
\norm{\E(\bu^{\rm e})}_{L^2(\Omega_\epsilon)}^2 \le \frac{3c_T^2c_K^2}{4\delta}\norm{\bm{f}^{\rm e}}_{L^2_a(-1,1)}^2 + \bigg(\frac{20}{3}+3c_P^2 \bigg)\norm{\nabla\wt\bv}_{L^2(\Omega_\epsilon)}^2.
\end{equation}

Then, using Korn's inequality \eqref{korn_free},
\begin{equation}\label{variational_err_free5}
\norm{\nabla\bu^{\rm e}}_{L^2(\Omega_\epsilon)}^2 \le \frac{3c_T^2c_K^4}{4\delta}\norm{\bm{f}^{\rm e}}_{L^2_a(-1,1)}^2 + c_K^2\bigg(\frac{20}{3}+\frac{3c_P}{4} \bigg)\norm{\nabla\wt\bv}_{L^2(\Omega_\epsilon)}^2.
\end{equation}

Recall from Section \ref{stability0} that the Korn constant $c_K$ and the pressure constant $c_P$ are both independent of $\epsilon$, while the trace constant $c_T$ has a $\abs{\log\epsilon}^{1/2}$ dependence.\\

Now, from \eqref{vee_est_free} we have that $\wt \bv$ satisfies
\begin{align*}
\norm{\nabla\wt\bv}_{L^2(\Omega_\epsilon)} &\le  C \big( \epsilon\abs{\log\epsilon}\norm{\bm{f}}_{\mc{C}^1(-1,1)}+ \epsilon\norm{\bm{f}}_{\mc{C}_a(-1,1)} \big).
\end{align*}
Furthermore, using Proposition \ref{free_f_est}, we have 
\begin{align*}
\norm{\bm{f}^{\rm e}}_{L^2_a(-1,1)} &\le \bigg(\int_{-1}^1\abs{\bm{f}^{\rm e}(s)}^2 \abs{\log(a(\varphi(s)))} ds\bigg)^{1/2} \\
 &\le \epsilon C \big(\norm{\bm{f}}_{\mc{C}^1(-1,1)}+\norm{\bm{f}}_{\mc{C}_a(-1,1)} \big) \bigg(\int_{-1}^1\abs{\log(a(\varphi(s)))} ds \bigg)^{1/2}\\
 & \le \epsilon C \big(\norm{\bm{f}}_{\mc{C}^1(-1,1)}+\norm{\bm{f}}_{\mc{C}_a(-1,1)} \big) ,
\end{align*}
since, by Definition \ref{admissible_a}, we have that 
\[  \int_{-1}^1\abs{\log(a(\varphi(s)))} ds <\infty \]
is some finite number due to the spheroidal endpoint condition (2). \\

Plugging the estimates for $\wt\bv$ and $\bm{f}^{\rm e}$ into \eqref{variational_err_free5}, we obtain
\begin{equation}\label{err_est_free}
\norm{\bu^{\rm e}}_{D^{1,2}(\Omega_\epsilon)} \le C \big(\epsilon\abs{\log\epsilon}\norm{\bm{f}}_{\mc{C}^1(-1,1)}+\epsilon\norm{\bm{f}}_{\mc{C}_a(-1,1)} \big) ,
\end{equation}
where $C$ depends on the constants $c_\Gamma$, $\kappa_{\max}$, $\delta_a$, $c_a$, $\bar c_a$, $c_\varphi$, $c_\eta$, and $c_{\eta,0}$, but not on $\epsilon$. Using \eqref{press_err_free}, we also have 
\begin{equation}\label{err_est_withp}
\norm{\bu^{\rm e}}_{D^{1,2}(\Omega_\epsilon)} + \norm{p^{\rm e}}_{L^2(\Omega_\epsilon)} \le C \big(\epsilon\abs{\log\epsilon}\norm{\bm{f}}_{\mc{C}^1(-1,1)}+\epsilon\norm{\bm{f}}_{\mc{C}_a(-1,1)} \big) ,
\end{equation}
where, by Lemma \ref{press_est_free}, $C$ still depends only on $c_\Gamma$, $\kappa_{\max}$, $\delta_a$, $c_a$, $\bar c_a$, $c_\varphi$, $c_\eta$, and $c_{\eta,0}$. \\

Finally, we derive the centerline error estimate \eqref{center_err_free} to complete the proof of Theorem \ref{free_err_theorem}. Notice that this estimate applies only along the effective centerline of the fiber, $s\in(-1,1)$. \\

First, recalling the definition of $\mc{J}_\epsilon$ \eqref{jac_free}, we define
\[ \mc{J}_{\min}:= \int_{-1}^1 \int_0^{2\pi} \mc{J}_\epsilon(s,\theta)d\theta ds.\]
Notice that this is not quite the expression for the surface area of the fiber, as it does not extend to the true endpoints at $\varphi=\pm\eta_\epsilon$. We have
\begin{align*}
\mc{J}_{\min} &\ge \int_{-1}^1\int_0^{2\pi} \epsilon a(1-\epsilon a\wh\kappa) d\theta ds = 2\pi \epsilon \int_{-1}^1a(s) ds \ge 4\pi\epsilon (1-\delta_a)a_0 \ge C\epsilon,
\end{align*}
where we used requirement (3.) of Definition \ref{admissible_a}. \\

We will use the following triangle inequality to show the desired centerline estimate \eqref{center_err_free}. Considering the trace of the true solution ${\rm Tr}(\bu)(s)$, $s\in (-1,1)$ along the effective centerline only, we have  
\begin{equation}\label{triangle_ineq}
\begin{aligned}
\norm{{\rm Tr}(\bu) - \bu^{\SB}_{\rm C}}_{L^2(-1,1)}^2 &\le \frac{1}{\mc{J}_{\min}}\int_0^{2\pi} \int_{-1}^1 \abs{ {\rm Tr}(\bu)(s) - \bu^{\SB}_{\rm C}(s)}^2 \mc{J}_\epsilon(s,\theta) ds d\theta   \\
&\le \frac{C}{\epsilon}\bigg(\int_0^{2\pi} \int_{-1}^1 \abs{ {\rm Tr}(\bu)(s) - {\rm Tr}(\bu^{\SB})(s,\theta)}^2 \mc{J}_\epsilon(s,\theta) ds d\theta \\
&\qquad +\int_0^{2\pi} \int_{-1}^1 \abs{ {\rm Tr}(\bu^{\SB})(s,\theta) - \bu^{\SB}_{\rm C}(s)}^2 \mc{J}_\epsilon(s,\theta) ds d\theta \bigg).
\end{aligned}
\end{equation}
Here ${\rm Tr}(\bu^{\SB})(s,\theta)$ is \eqref{SBT_free} evaluated on $\Gamma_\epsilon$. We will estimate the two terms of \eqref{triangle_ineq} separately. \\

 We begin with the error term ${\rm Tr}(\bu^{\rm e})(s,\theta)= {\rm Tr}(\bu)(s) - {\rm Tr}(\bu^{\SB})(s,\theta)$, $s\in(-1,1)$, on the fiber surface. Recalling \eqref{free_err_PDE}, we have
\begin{align*}
\int_0^{2\pi} \int_{-1}^1 \abs{ {\rm Tr}(\bu^{\rm e})(s,\theta) }^2 \mc{J}_\epsilon(s,\theta) ds d\theta &\le \int_0^{2\pi} \int_{-1}^1 \abs{\bar \bu^{\rm e}(s) }^2 \mc{J}_\epsilon(s,\theta) ds d\theta \\
&\qquad + \int_0^{2\pi} \int_{-1}^1 \abs{\bu^{\rm r}(s,\theta) }^2 \mc{J}_\epsilon(s,\theta) ds d\theta.
\end{align*}
We define
\begin{align*}
\mc{J}_{\max} & := \sup_{s\in(-1,1)} \int_0^{2\pi} \mc{J}_\epsilon(s,\theta) \abs{\log(a(s))} d\theta.
\end{align*}
Note that since we are only considering $s\in(-1,1)$, away from the true endpoints, we have $C\epsilon\le a(s)\le 1$, by Definition \ref{admissible_a}. Then, using that $\abs{a'}\le \bar c_a /a(s)$, we can show 
\begin{align*}
\mc{J}_{\max} &\le \int_0^{2\pi}\big( \epsilon a\abs{\log a}(1-\epsilon a\wh\kappa) + \epsilon^2a\abs{\log a}\abs{a'} \big) d\theta \le 2\pi \epsilon (1+ C\epsilon^2\abs{\log\epsilon})\le C\epsilon.
\end{align*}

Using that $\bar \bu^{\rm e}$ is independent of $\theta$, we then have  
\begin{align*}
\bigg( \int_0^{2\pi}\int_{-1}^1 \abs{\bar \bu^{\rm e}(s) }^2 \mc{J}_\epsilon(s,\theta) ds d\theta \bigg)^{1/2} &= \bigg( \int_{-1}^1 \abs{\bar \bu^{\rm e}(s) }^2  \abs{\log(a(s))}^{-1} \int_0^{2\pi}\mc{J}_\epsilon(s,\theta) \abs{\log(a(s))} d\theta ds\bigg)^{1/2} \\
&\le C\sqrt{\mc{J}_{\max}}\bigg(\int_{-1}^1 \abs{\bar \bu^{\rm e}(s) }^2  \abs{\log(a(s))}^{-1}ds\bigg)^{1/2} \\
&\le C\sqrt{\mc{J}_{\max}}\norm{\bar\bu^{\rm e}  \abs{\log(a)}^{-1/2}}_{L^2(-\eta_\epsilon,\eta_\epsilon)} \le C\sqrt{\epsilon} c_T \norm{\nabla \wt\bu^{\rm e}}_{L^2(\Omega_\epsilon)} \\
&\le C(\epsilon\abs{\log\epsilon})^{1/2} \big( \norm{\bu^{\rm e}}_{D^{1,2}(\Omega_\epsilon)}+\norm{\nabla\wt\bv}_{L^2(\Omega_\epsilon)} \big) \\
&\le C\big( (\epsilon\abs{\log\epsilon})^{3/2} \norm{\bm{f}}_{\mc{C}^1(-1,1)} + \epsilon^{3/2}\abs{\log\epsilon}^{1/2} \norm{\bm{f}}_{\mc{C}_a(-1,1)} \big).
\end{align*}
Here we have used the trace inequality (Lemma \ref{free_trace_lemma}) in the third line and the estimates \eqref{vee_est_free} and \eqref{err_est_free} in the last inequality. \\

Furthermore, by Proposition \ref{ur_ests_free}, we have 
\begin{align*}
\bigg(\int_0^{2\pi} \int_{-1}^1 \abs{\bu^{\rm r}(s,\theta) }^2 \mc{J}_\epsilon(s,\theta) ds d\theta\bigg)^{1/2} &\le C \sqrt{\mc{J}_{\max}} \big( \epsilon\abs{\log\epsilon}\norm{\bm{f}}_{\mc{C}^1(-1,1)}+ \epsilon \norm{\bm{f}}_{\mc{C}_a(-1,1)}\big) \\
& \le C \big(\epsilon^{3/2}\abs{\log\epsilon}\norm{\bm{f}}_{\mc{C}^1(-1,1)} +\epsilon^{3/2}\norm{\bm{f}}_{\mc{C}_a(-1,1)} \big).
\end{align*}

Together, we obtain the $L^2$ surface estimate  
\begin{equation}\label{free_trace_final}
\begin{aligned}
\bigg(\int_0^{2\pi} \int_{-1}^1 &\abs{ {\rm Tr}(\bu^{\rm e})(s,\theta) }^2 \mc{J}_\epsilon(s,\theta) ds d\theta\bigg)^{1/2} \\
& \le C \big( (\epsilon \abs{\log\epsilon})^{3/2} \norm{\bm{f}}_{\mc{C}^1(-1,1)} + \epsilon^{3/2} \abs{\log\epsilon}^{1/2} \norm{\bm{f}}_{\mc{C}_a(-1,1)} \big),
\end{aligned}
\end{equation}
where $C$ depends on $c_\Gamma$, $\kappa_{\max}$, $\delta_a$, $c_a$, $\bar c_a$, $a_0$, $c_\varphi$, $c_\eta$, and $c_{\eta,0}$, but not on $\epsilon$. \\

Furthermore, using Proposition \ref{center_prop}, for $s\in (-1,1)$, we have that the difference between the surface slender body approximation ${\rm Tr}(\bu^{\SB})(s,\theta)$ and the centerline approximation $\bu^{\SB}_{\rm C}(s)$ satisfies 
\begin{equation}\label{free_center_final}
\begin{aligned}
\bigg(\int_0^{2\pi}\int_{-1}^1\abs{{\rm Tr}(\bu^{\SB})(s,\theta) -\bu^{\SB}_{\rm C}(s)}^2 \mc{J}_\epsilon(s,\theta) ds d\theta \bigg)^{1/2} &\le C\sqrt{\mc{J}_{\max}}\epsilon\abs{\log\epsilon}\norm{\bm{f}}_{\mc{C}^1(-1,1)}  \\
&\le C\epsilon^{3/2}\abs{\log\epsilon} \norm{\bm{f}}_{\mc{C}^1(-1,1)} .
\end{aligned}
\end{equation}

Combining the estimates \eqref{free_trace_final} and \eqref{free_center_final} in \eqref{triangle_ineq}, we obtain the centerline estimate
\begin{align*}
\norm{{\rm Tr}(\bu) - \bu^{\SB}_{\rm C}}_{L^2(-1,1)} &\le C \big( \epsilon \abs{\log\epsilon}^{3/2} \norm{\bm{f}}_{\mc{C}^1(-1,1)} + \epsilon \abs{\log\epsilon}^{1/2} \norm{\bm{f}}_{\mc{C}_a(-1,1)} \big).
\end{align*}

\appendix
\section{Appendix}
\subsection{Dependence of well-posedness constants on $\epsilon$}\label{stability}
In this appendix, we prove each of the key well-posedness inequalities stated in Section \ref{stability0}, paying close attention to how each resulting constant depends on $\epsilon$. \\

First, we prove a simple centerline straightening lemma (Lemma \ref{straightening_lemma}) to facilitate computations in the following sections. Using Lemma \ref{straightening_lemma}, we prove the form \eqref{free_trace} of the trace inequality for $\A_\epsilon$ functions. As in the closed loop setting \cite{closed_loop}, we show that the $L^2$ trace along the fiber blows up like $\abs{\log\epsilon}^{1/2}$ as $\epsilon\to 0$; however, unlike in the closed loop case, we show that an additional weight is needed at the fiber endpoints to ensure that the trace does not diverge. This results in the required decay condition \eqref{first_decay_space} for any prescribed force $\bm{f}$. \\

Next we show the existence of an extension operator for slender bodies with free, spheroidal endpoints whose symmetric gradient is bounded independent of $\epsilon$. As far as we know, this result is new and may be more widely useful beyond the scope of this paper. This extension operator is then used to show $\epsilon$-independence for the Korn and Sobolev inequalities in $\Omega_\epsilon$. \\

Finally we verify that the pressure estimate \eqref{freeP_boundE} does not depend on $\epsilon$. The proof here is essentially the same as in the closed loop setting, relying only on being able to write the region $\mc{O}$ \eqref{mc_O} as a finite union of star-shaped domains with respect to an $\epsilon$-independent ball.

\subsubsection{Centerline straightening}
To facilitate the calculations in the following sections, it will be useful to define an $\epsilon$-independent straightening operator within the neighborhood $\mc{O}$ \eqref{mc_O} about $\Sigma_\epsilon$, taking $\X$ as in Section \ref{geometry_section} to a straight line. We show the following lemma. 
\begin{lemma}\label{straightening_lemma}
Consider a fiber centerline $\X(\varphi)$ as in Section \ref{geometry_section} and the neighborhood $\mc{O}$ given by \eqref{mc_O}. Within $\mc{O}$, there exists a $\mc{C}^1$ operator $\Psi$ taking the moving frame basis vectors $\be_{\rm t}(\varphi)$, $\be_\rho(\varphi,\theta)$, $\be_\theta(\varphi,\theta)$ defined with respect to a Bishop frame about $\X_\text{ext}(\varphi)$ to straight cylindrical coordinates $\varphi \be_{\rm t} + \rho \be_\rho + \theta \be_\theta$ about a straight fiber centerline such that $\nabla\Psi$ and $\nabla(\Psi^{-1})$ are both bounded independent of $\epsilon$. 
\end{lemma}

\begin{proof}
Let $\mc{O}_{\text{str}}$ denote the region $\mc{O}$ with straight centerline $\X_{\text{ext}}$, parameterized with respect to straight cylindrical coordinates $(\varphi,\theta,\rho)$. Let $\Psi_0: \mc{O}_{\text{str}} \to \mc{O}$ denote the coordinate map taking the straight cylindrical coordinates $(\varphi,\theta,\rho) \mapsto \X_\text{ext}(\varphi)+ \rho \be_{\rho}(\varphi,\theta)$ about the curved slender body centerline. Note that the map $\Psi_0$ is $\mc{C}^1$ due to the Bishop frame \eqref{bishop_ODE}. Furthermore, $\nabla\Psi_0$ is full rank at each point $(\varphi,\theta,\rho)$, which can be seen as follows. Letting $\overline \be_{\rm t},\overline\be_\theta,\overline\be_\rho$ denote the straight cylindrical basis vectors, for any $\varphi_0\in(1-r_{\max},1+r_{\max})$, we can choose a rotation $\bm{R}\in SO(2)$ such that $\bm{R}\be_{\rm t}(\varphi_0)=\overline\be_{\rm t}$, $\bm{R}\be_\rho(\varphi_0,\theta)=\overline\be_\rho$, and $\bm{R}\be_\theta(\varphi_0,\theta)=\overline\be_\theta$. Then, using \eqref{bishop_ODE},
\[ \nabla(\bm{R}\Psi_0) = (1-\rho\wh\kappa)\overline\be_{\rm t}\overline\be_{\rm t}^{\rm T} + \overline\be_\theta\overline\be_\theta^{\rm T}
+ \overline\be_\rho\overline\be_\rho^{\rm T}\]
has positive determinant, since $\rho<r_{\max}$. \\

We then define the straightening operator $\Psi$ on the $\X_\text{ext}$ by $\Psi = \Psi_0^{-1}$. Note $\Psi_0^{-1}\in \mc{C}^1$ by the inverse function theorem. Also, the map $\Psi$ depends only on the shape of the fiber centerline -- in particular, on the constants $c_{\Gamma}$ and $\kappa_{\max}$. 
\end{proof}

\begin{remark}\label{straight_rem}
We can use Lemma \ref{straightening_lemma} to work locally with a straight fiber without introducing additional $\epsilon$-dependence into resulting estimates. \\

Let $\mc{O}_2$ denote the region
\begin{align*}
\mc{O}_2 &:= \bigg\{\bx \in \Omega_\epsilon \ts : \ts \bx= \X(\varphi)+\rho \be_\rho(\varphi,\theta); \ts \abs{\varphi} \le 1, \ts 0\le \theta<2\pi, \ts \epsilon a(\varphi) < \rho< \frac{r_{\max}}{2}  \bigg\} \\
&\hspace{2cm} \bigcup \bigg\{\bx \in \Omega_\epsilon \ts : \ts \bx= \X_\text{ext}(\varphi)+\rho \be_\rho(\varphi,\theta); \ts 1<\abs{\varphi} < 1+\frac{r_{\max}}{2},\\
&\hspace{5cm}  0\le \theta<2\pi, \ts \epsilon a(\varphi) < \rho< \sqrt{r_{\max}^2/4-(\abs{\varphi}-1)^2}  \bigg\},
\end{align*}
i.e. the region $\mc{O}$ with extent $r_{\max}/2$ rather than $r_{\max}$. \\

We use the following partition of unity to divide $\Omega_\epsilon$ in two regions: 
\begin{align*}
\phi_1 &= \begin{cases}
1 & \text{if } \bx \in \mc{O}_2 \\ 
0 & \text{if } \bx \in \R^3\backslash \mc{O}
\end{cases}, \quad \phi_2 = 1-\phi_1,
\end{align*}
with smooth transition between $\mc{O}_2$ and $\mc{O}$ for both $\phi_1$ and $\phi_2$. Since $r_{\max}$ depends only on the constants $c_{\Gamma}$ and $\kappa_{\max}$ and not on $\epsilon$, these cutoffs are both independent of $\epsilon$. Therefore we can consider the region $\mc{O}$ (where $\phi_1$ is supported) separately from $\R^3\backslash \mc{O}_2$. In particular, whenever only information near the slender body surface is needed, we can use Lemma \ref{straightening_lemma} to consider only slender bodies $\Sigma_\epsilon$ with straight centerline. This will be useful for both the trace and extension operators constructed in the following sections, both of which rely only on information very close to $\Gamma_\epsilon$.  
\end{remark}

\subsubsection{Trace inequality}\label{trace_section}
Here we prove the weighted trace inequality (Lemma \ref{free_trace_lemma}) for functions in the space $\A_\epsilon$. We show that the weight $\abs{\log(a)}^{-1/2}$ is needed to avoid a logarithmic divergence in the trace at the fiber endpoints. 


\begin{proof}[Proof of Lemma \ref{free_trace_lemma}]
It suffices to show that \eqref{free_trace_lemma} holds within the region $\mc{O}$. In particular, using the cutoff described in Remark \ref{straight_rem}, we consider only functions $\bu$ supported within a distance $r_{\max}\le 1$ of the fiber centerline. \\

Using Lemma \ref{straightening_lemma}, we proceed to show that the estimate \eqref{free_trace_ineq} holds in the exterior of a straight slender body, parameterized with respect to straight cylindrical coordinates as
 \begin{equation}\label{Sigma_str}
 \Sigma^\epsilon_\text{str}= \{ (\varphi,\theta,\rho) \ts : \ts -\eta_\epsilon <\varphi <\eta_\epsilon, \ts 0\le \theta<2\pi, \ts \rho \le \epsilon a(\varphi) \}
 \end{equation} 
for $a(\varphi)$ as in Definition \ref{admissible_a}. Let 
\begin{equation}\label{Gamma_str}
\Gamma^\epsilon_\text{str} = \p \Sigma^\epsilon_\text{str}, \quad \Omega^\epsilon_\text{str} = \R^3\backslash\overline{\Sigma^\epsilon_\text{str}}.
\end{equation}

We define the admissible set 
\[ \A^\epsilon_\text{str} = \big\{ \bv \in D^{1,2}(\Omega^\epsilon_\text{str}) \ts : \ts  \bv\big|_{\Gamma^\epsilon_\text{str}} = \bv(\varphi); \quad \text{supp}(\bv) \subset \{\bx(\varphi,\theta,\rho) \ts : \ts \rho < 1\} \big\}. \] 

Following the same outline as in the closed loop setting \cite{closed_loop}, we show that \eqref{free_trace_ineq} holds for $\bu \in \mc{C}^\infty_0(\overline{\Omega^\epsilon_\text{str}})\cap \A^\epsilon_\text{str}$, where again the $\overline{\Omega^\epsilon_\text{str}}$ notation denotes functions supported up to $\Gamma^\epsilon_\text{str}$. The result for $\bu\in \A^\epsilon_\text{str}$ then follows by density. \\

Now, for $\bx\in \Gamma^\epsilon_\text{str}$, we use the fundamental theorem of calculus to write $\bu(\bx)=\bu(\epsilon a(\varphi), \theta,\varphi)$ as
\[ \bu(\varphi,\theta,\epsilon a(\varphi)) = \int_{\epsilon a(\varphi)}^1 \frac{\p \bu}{\p\rho} \ts d\rho.\]

Then we have
\begin{align*}
\abs{\bu(\varphi,\theta,\epsilon a(\varphi))} &\le \int_{\epsilon a(\varphi)}^1 \abs{\frac{\p \bu}{\p\rho}} d\rho \le \int_{\epsilon a(\varphi)}^1 \frac{1}{\sqrt{\rho}}\sqrt{\rho} \abs{\frac{\p \bu}{\p\rho}} d\rho \\
&\le \bigg(\int_{\epsilon a(\varphi)}^1 \frac{1}{\rho} d\rho \bigg)^{1/2} \bigg( \int_{\epsilon a(\varphi)}^1 \abs{\frac{\p \bu}{\p\rho}}^2 \rho d\rho\bigg)^{1/2} \\
&= \abs{\log(\epsilon a(\varphi))}^{1/2}  \bigg( \int_{\epsilon a(\varphi)}^1 \abs{\frac{\p \bu}{\p\rho}}^2 \rho d\rho\bigg)^{1/2},
\end{align*}
and therefore
\begin{equation}\label{free_trace_bd} 
\abs{{\rm Tr}(\bu)}^2 \le \abs{\log(\epsilon a(\varphi))} \int_{\epsilon a(\varphi)}^1 \abs{\frac{\p \bu}{\p\rho}}^2 \rho \ts d\rho. 
\end{equation}

Already we can see that the bound \eqref{free_trace_bd} diverges at the fiber endpoints $\varphi=\pm \eta_\epsilon$. This is due to the fact that as $a(\varphi)\to 0$ at the fiber endpoints, we are trying to limit to a truly one-dimensional trace of an $H^1$ function in $\R^3$.  \\

To correct for this, we multiply $\abs{{\rm Tr}(\bu)}^2$ in \eqref{free_trace_bd} by $\abs{\log(a(\varphi))}^{-1}$. Then, since $\bu$ belongs to $\A_\epsilon^{\dive}$, using the $\theta$-independence of ${\rm Tr}(\bu)={\rm Tr}(\bu)(\varphi)$, we may write
\[ \norm{{\rm Tr}(\bu) \abs{\log(a)}^{-1/2}}_{L^2(-\eta_\epsilon,\eta_\epsilon)}^2 = \frac{1}{2\pi} \int_{-\eta_\epsilon}^{\eta_\epsilon} \int_0^{2\pi} \abs{{\rm Tr}(\bu)(\varphi)}^2\abs{\log(a(\varphi))}^{-1} d\theta d\varphi. \]

Using \eqref{free_trace_bd}, we then have that ${\rm Tr}(\bu)(\varphi)$ satisfies
\begin{align*}
&\norm{{\rm Tr}(\bu) \abs{\log(a)}^{-1/2}}_{L^2(-\eta_\epsilon,\eta_\epsilon)}^2 \\
&\qquad \le \frac{1}{2\pi}\int_{-\eta_\epsilon}^{\eta_\epsilon} \int_0^{2\pi} \abs{\log(\epsilon a(\varphi))} \abs{\log(a(\varphi))}^{-1} \int_{\epsilon a(\varphi)}^1 \abs{\frac{\p \bu}{\p\rho}}^2 \rho \ts d\rho d\theta d\varphi\\
&\qquad \le \frac{1}{\pi} \abs{\log\epsilon} \norm{\nabla \bu}_{L^2(\Omega^\epsilon_\text{str})}^2,
\end{align*}
where we have used that $a\le 1$. 
\end{proof} 

Note that, for $\bu\in \A_\epsilon$, this one-dimensional $L^2$ trace inequality along the fiber centerline relates to a (weighted) trace inequality over the entire fiber surface $\Gamma_\epsilon$ by 
\begin{align*}
\norm{{\rm Tr}(\bu)\abs{\log(a)}^{-1/2}}_{L^2(\Gamma_\epsilon)}^2 &= \int_{-\eta_\epsilon}^{\eta_\epsilon}\int_0^{2\pi} \abs{{\rm Tr}(\bu)(\varphi)}^2 \abs{\log(a(\varphi))}^{-1} \mc{J}_\epsilon(\varphi,\theta) \ts d\theta d\varphi \\
&\le \int_{-\eta_\epsilon}^{\eta_\epsilon}\int_0^{2\pi} \abs{{\rm Tr}(\bu)(\varphi)}^2\abs{\log(a(\varphi))}^{-1}(\epsilon a+\epsilon^2a^2\kappa_{\max}+\epsilon^2\bar c_a) \ts d\theta d\varphi \\
&\le 2\pi \epsilon(1+\epsilon\kappa_{\max}+\epsilon\bar c_a)\norm{{\rm Tr}(\bu) \abs{\log(a)}^{-1/2}}_{L^2(-\eta_\epsilon,\eta_\epsilon)}^2. 
\end{align*}
Here we have used \eqref{jac_free} and \eqref{bar_ca} to bound $\mc{J}_\epsilon$. 

\subsubsection{Extension operator for thin fiber with spheroidal endpoints}\label{extension_op}
In order to determine the $\epsilon$-dependence in the Korn inequality (Lemma \eqref{korn_free}), as in \cite{closed_loop}, it will be convenient to show the existence of a $D^{1,2}$ extension into the interior of $\Sigma_\epsilon$ whose symmetric gradient is bounded independent of $\epsilon$. Again, due to Lemma \ref{straightening_lemma}, it suffices to show the existence of such an operator for a slender filament with straight centerline. We again parameterize the straight centerline with respect to cylindrical coordinates $(\varphi,\theta,\rho)$ as $\{ (\varphi,0,0) \ts : \ts -\eta_\epsilon <\varphi <\eta_\epsilon \}$. \\

Let $\Sigma^\epsilon_\text{str}$ and $\Gamma^\epsilon_\text{str}$ be as in \eqref{Sigma_str} and \eqref{Gamma_str}. 
We define the domain $\mc{W}^\epsilon$, a ``fattened'' version of $\Sigma^\epsilon_\text{str}$, as 
 \begin{equation}\label{mcW_def}
 \begin{aligned}
 \mc{W}^\epsilon= \bigg\{(\varphi,\theta,\rho) \ts : \ts  -2\eta_\epsilon+1 &< \varphi < 2\eta_\epsilon-1, \ts 0\le \theta<2\pi, \\
&\qquad  0\le \rho <\begin{cases}
 2\epsilon a(\varphi), & \abs{\varphi}\le 1 \\
2\epsilon a(\frac{\varphi+1}{2}), & 1< \abs{\varphi} <2\eta_\epsilon -1 
 \end{cases} \bigg\},
 \end{aligned}
 \end{equation}
and take
 \begin{equation}\label{mcH_def}
 \mc{H}^\epsilon=\mc{W}^\epsilon\backslash \overline{\Sigma^\epsilon_\text{str}}.
 \end{equation}
 
Note that on the bounded domain $\mc{H}^\epsilon$, the function spaces $D^{1,2}(\mc{H}^\epsilon)$ and $H^1(\mc{H}^\epsilon)$ coincide; thus it suffices to construct an $H^1$ extension from $\mc{H}^\epsilon$ to $\mc{W}^\epsilon$ whose symmetric gradient is bounded independent of $\epsilon$. We show the following lemma.

\begin{lemma}\emph{(Free endpoint slender body extension)}\label{extension_free}
Let $\mc{W}^\epsilon$ and $\mc{H}^\epsilon$ be as in \eqref{mcW_def} and \eqref{mcH_def}, respectively. For $\bu\in H^1(\mc{H}^\epsilon)$, there exists an operator $T_\epsilon: H^1(\mc{H}^\epsilon) \to H^1(\mc{W}^\epsilon)$ extending $\bu$ into the interior of the slender body such that 
\begin{enumerate}
\item $T_\epsilon \bu \big|_{\mc{H}^\epsilon} = \bu$ 
\item $\norm{\E(T_\epsilon \bu)}_{L^2(\mc{W}^\epsilon)} \le C \norm{\E(\bu)}_{L^2(\mc{H}^\epsilon)}$, where the constant $C$ depends on the constants $c_a$, $\bar c_a$, $a_0$, $\delta_a$, $c_{\eta,0}$, and $c_\eta$, but is independent of $\epsilon$.
\end{enumerate}
\end{lemma}

As an immediate corollary, due to Lemma \ref{straightening_lemma}, we have:
\begin{corollary}\emph{(Free endpoint extension -- curved centerline)}\label{extension_free2}
Let $\Sigma_\epsilon$ and $\Omega_\epsilon$ be as in Section \ref{geometry_section}. For $\bu\in D^{1,2}(\Omega_\epsilon)$, there exists an operator $\wt{T}_\epsilon: D^{1,2}(\Omega_\epsilon)\to D^{1,2}(\R^3)$ extending $\bu$ into the interior of the slender body $\Sigma_\epsilon$ such that 
\begin{enumerate}
\item $\wt {T}_\epsilon \bu \big|_{\Omega_\epsilon} = \bu$ 
\item $\norm{\E(\wt{T}_\epsilon \bu)}_{L^2(\mc{W}^\epsilon)} \le C \norm{\E(\bu)}_{L^2(\mc{H}^\epsilon)}$, where $C$ now depends on $\kappa_{\max}$ and $c_\Gamma$ in addition to $c_a$, $\bar c_a$, $a_0$, $\delta_a$, $c_{\eta,0}$, and $c_\eta$, but remains independent of $\epsilon$.
\end{enumerate}
\end{corollary}

\begin{proof}[Proof of Lemma \ref{extension_free}]  We subdivide the somewhat long and technical proof into 6 main steps. \\
{\bf Step 1: Reflection $E$ into the interior of $\Sigma_\text{str}^\epsilon$:} \\
First, for any $\bx=\bx(\varphi,\theta,\rho)\in \mc{W}^\epsilon$, we let
\[ \delta(\varphi,\rho) = \inf_{s\in(-1,1)} \abs{\bx - (s,0,0)}\]
denote the distance from $\bx$ to the straightened effective centerline \eqref{effective_centerline}. Note that since $\mc{W}^\epsilon$ is radially symmetric for each $\varphi$, the distance $\delta(\varphi,\rho)$ depends only on $\varphi$ and $\rho$ and not on $\theta$. In particular, we have
\[ \delta(\varphi,\rho) = \begin{cases}
\rho, & \abs{\varphi} \le 1 \\
\sqrt{(\abs{\varphi}-1)^2+\rho^2}, & 1 < \abs{\varphi} <\eta_\epsilon.
\end{cases} \]

Now, for any $(\varphi,\theta,\rho)\in \mc{W}^\epsilon$, define $\varphi^*\in (-\eta_\epsilon,\eta_\epsilon)$ such that $(\varphi^*,\theta, \epsilon a(\varphi^*))$ is the projection of $(\varphi,\theta,\rho)$ onto $\Gamma^\epsilon_{\text{str}}$ along the straight line distance $\delta(\varphi,\rho)$ to the effective centerline. Note that for $\abs{\varphi}\le 1$, $\varphi^*=\varphi$. Toward the fiber endpoints, $1<\abs{\varphi}<\eta_\epsilon$, we have 
\begin{equation}\label{varphi_star}
 \begin{cases}
\frac{\epsilon a(\varphi^*)}{\varphi^*-1}=\frac{\rho}{\varphi-1}, & 1<\varphi<\eta_\epsilon \\
\frac{\epsilon a(\varphi^*)}{\varphi^*+1}=\frac{\rho}{\varphi+1}, & -\eta_\epsilon <\varphi < -1,
\end{cases} 
\end{equation}
This means that the pair $(\varphi^*,a(\varphi^*))$ lies on the line connecting $(\varphi,\rho)$ to the effective endpoint $(\pm1,0)$. Note that as $\rho\to 0$, $\varphi^*\to \pm \eta_\epsilon$, and as $\varphi\to 1$, $\varphi^*\to 1$. Also, for $1<\abs{\varphi}<\eta_\epsilon$, we have 
\begin{equation}\label{star_derivs}
\frac{\p\varphi^*}{\p\varphi} = \frac{\epsilon a(\varphi^*)}{\rho+\epsilon |a'(\varphi^*)|(|\varphi| - 1)}, \quad 
\frac{\p\varphi^*}{\p\rho} =\frac{-(|\varphi^*|-1)}{\rho+ \epsilon |a'(\varphi^*)|(|\varphi|- 1)}.
 \end{equation}

Consider $\bv=\bv(\varphi,\theta,\rho)\in H^1(\mc{H}^\epsilon)$ supported only at $(\varphi,\theta,\rho)\in \mc{H}^\epsilon$ satisfying $\delta(\varphi,\rho) \le \frac{5}{3}\delta(\varphi^*,\epsilon a(\varphi^*))$ (see Figure \ref{fig:E_pic}). This initial constraint on the support of $\bv$ is important for the definition \eqref{extE} and will be addressed for functions with more general support in Step 3. We define an extension $E\bv$ into the interior of $\Sigma^\epsilon_\text{str}$ as the reflection of $\bv$ across $\Gamma^\epsilon_\text{str}$ along the shortest straight line connecting $\bx$ to the effective centerline segment $[-1,1]$ (See Figure \ref{fig:E_pic}). More precisely, 
\begin{equation}\label{extE}
E\bv = \begin{cases}
 \bv(\bx), & \text{ if } \bx\in \mc{H}^\epsilon ; \\
\begin{cases}
\bv\big(\varphi,\theta, 2\epsilon a(\varphi)-\rho\big), & \abs{\varphi} \le1 \\
\bv\big(2\varphi^*-\varphi-1,\theta,2\epsilon a(\varphi^*)-\rho\big), & 1 < \varphi \le \eta_\epsilon  \\
\bv\big(2\varphi^*-\varphi+1,\theta,2\epsilon a(\varphi^*)-\rho\big), & -\eta_\epsilon \le \varphi < -1,
\end{cases} & \text{ if } \bx = (\varphi,\theta,\rho)\in \Sigma^\epsilon_\text{str}. 
\end{cases} 
\end{equation}
Here $\varphi^*$ is as in \eqref{varphi_star}. Note that this extension is well-defined for each $\bx\in \Gamma_\text{str}^\epsilon$ due to the limits on supp$(\bv)$. \\

\begin{figure}[!h]
\centering
\includegraphics[scale=0.7]{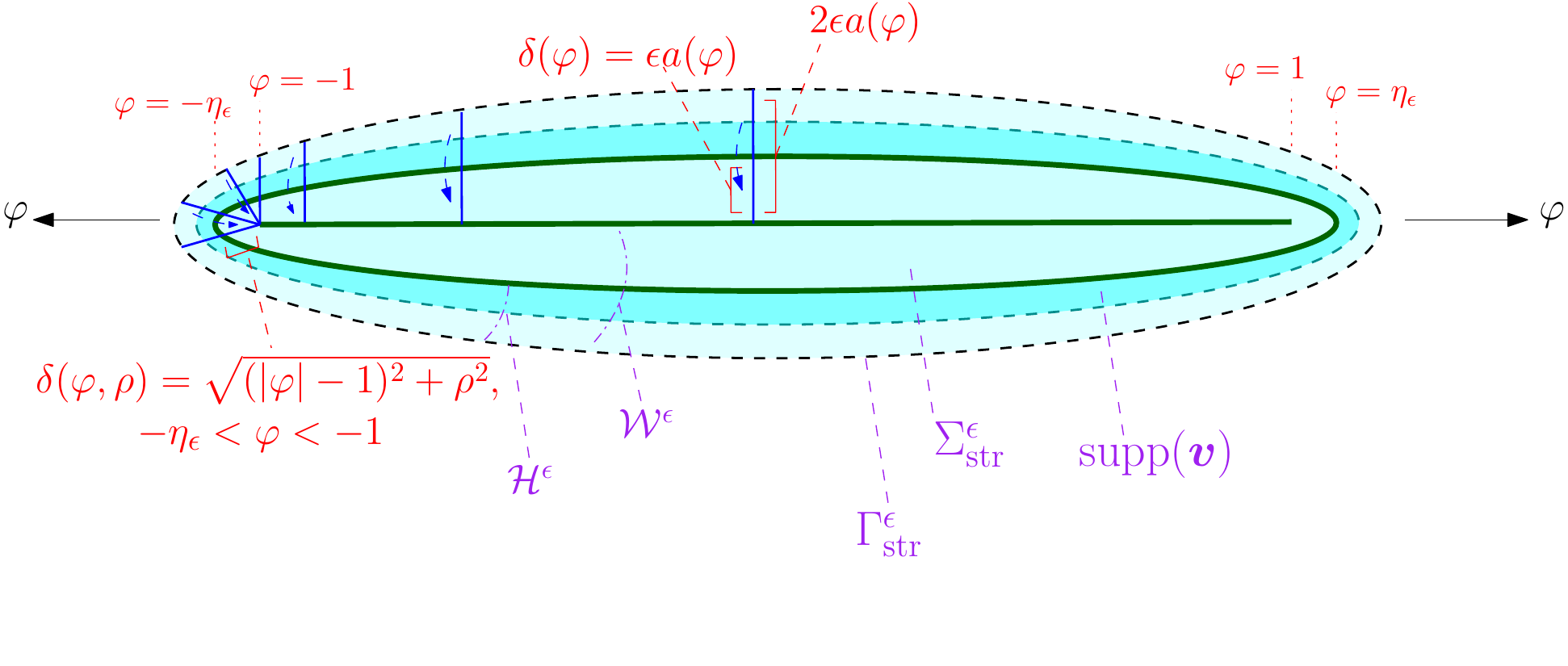}\\
\caption{Consider $\bv(\varphi,\theta,\rho)\in H^1(\mc{H}^\epsilon)$ supported only at $(\varphi,\theta,\rho)\in \mc{H}^\epsilon$ satisfying $\delta(\varphi,\rho) \le \frac{5}{3}\delta(\varphi^*,\epsilon a(\varphi^*))$. We construct the basic extension operator $E: H^1(\mc{H^\epsilon})\to H^1(\mc{W}^\epsilon)$ by reflecting $\bv\in H^1(\mc{H}^\epsilon)$ across the straight line distance to the effective centerline of the fiber ($-1\le \varphi\le 1$). }
\label{fig:E_pic}
\end{figure}

{\bf Step 2: Gradient estimates for $E$:} \\
We now aim to prove the following estimate on the (Euclidean) gradient of the reflection $E$ within $\Sigma_\text{str}^\epsilon$. 

\begin{proposition}\label{Egrad_est}
Let $E$ be as defined in \eqref{extE}. For each $\bx=(\varphi,\theta,\rho)\in \Sigma_{\text{str}}^\epsilon$, we have
\begin{equation}\label{nablaE_bd}
\abs{\nabla(E\bv)(\varphi,\theta,\rho)} \le  C\begin{cases}
\abs{\nabla \bv(\varphi,\theta,2\epsilon a(\varphi^*)-\rho)}, & \abs{\varphi}\le 1 \\
\abs{\nabla \bv(2\varphi^*-\varphi-1,\theta,2\epsilon a(\varphi^*)-\rho)}, & 1<\varphi \le \eta_\epsilon \\
\abs{\nabla \bv(2\varphi^*-\varphi-1,\theta,2\epsilon a(\varphi^*)-\rho)}, & -\eta_\epsilon \le \varphi <-1.
\end{cases} 
\end{equation}
\end{proposition}

\begin{proof}[Proof of Proposition \ref{Egrad_est}]
For $\bx\in \Sigma^\epsilon_{\text{str}}$ and $\abs{\varphi}\le 1$, we differentiate equation 2 of \eqref{extE} to obtain
\begin{equation}\label{nablaE}
\begin{aligned}
\abs{\nabla(E\bv)(\varphi,\theta,\rho)} &\le \abs{\frac{1}{\rho}\frac{\p\bv}{\p\theta}} + \abs{2\epsilon a'(\varphi) - 1} \abs{\frac{\p\bv}{\p\rho}} + \abs{\frac{\p\bv}{\p\varphi}} \\
& \le \abs{\frac{C}{2\epsilon a(\varphi)-\rho}}\abs{\frac{\p\bv}{\p\theta}} + \bigg(\frac{\epsilon \bar c_a}{a(1)}+1\bigg) \abs{\frac{\p\bv}{\p\rho}} + \abs{\frac{\p\bv}{\p\varphi}} \\
&\le C\abs{\nabla\bv(\varphi,\theta,2\epsilon a(\varphi)-\rho)},
\end{aligned}
\end{equation}
where, in the second inequality, we used that supp($E\bv$) extends only to $\rho \ge \epsilon a(\varphi)/3$ to bound the $\p/\p\theta$ term, and Definition \ref{admissible_a} to bound the $\p/\p\rho$ term. Note that all functions on the right hand side in \eqref{nablaE} are evaluated at $(\varphi,\theta,2\epsilon a(\varphi)-\rho)$. \\

Similarly, for $\bx\in \Sigma^\epsilon_{\text{str}}$ and $1< \varphi <\eta_\epsilon$, differentiating equation 3 of \eqref{extE} and using \eqref{star_derivs}, the gradient of the extension $\nabla(E\bv)$ satisfies 
\begin{equation}\label{nablaE_end}
\begin{aligned}
\abs{\nabla(E\bv)(\varphi,\theta,\rho)} &\le \abs{\frac{2\epsilon a(\varphi^*)-2(\varphi^*-1)}{\rho+\epsilon |a'(\varphi^*)|(\varphi-1)} -1}\abs{\frac{\p\bv}{\p\varphi}} + \abs{\frac{1}{\rho}\frac{\p\bv}{\p\theta}} \\
&\qquad + \abs{\frac{2\epsilon^2 a(\varphi^*)a'(\varphi^*) + 2\epsilon |a'(\varphi^*)|(\varphi^*-1)}{\rho+\epsilon |a'(\varphi^*)|(\varphi-1)} -1}\abs{\frac{\p\bv}{\p\rho}}.
\end{aligned}
\end{equation}

Here each function on the right hand side of \eqref{nablaE_end} is evaluated at $(2\varphi^*-\varphi-1,\theta,2\epsilon a(\varphi^*)-\rho)$. We must now bound each of the above coefficients to obtain a bound as in \eqref{nablaE}. Using Definition \ref{admissible_a}, since $\varphi^*\ge 1$ we have $|a'(\varphi^*)|\ge \frac{C}{\epsilon}$. Thus
 \[ \rho+\epsilon\abs{a'(\varphi^*)}(\varphi-1) \ge C\big(\rho +(\varphi-1) \big).\]
 
 Furthermore, using that $\bv(\varphi,\theta,\rho)$ is only supported at $(\varphi,\theta,\rho)\in \mc{H}^\epsilon$ satisfying $\delta(\varphi,\rho)<\frac{5}{3} \delta(\varphi^*,\epsilon a(\varphi^*))$, we have that $\nabla(E\bv)$ is only supported at $(\varphi,\theta,\rho)\in \Sigma_{\text{str}}^\epsilon$ satisfying
 \begin{equation}\label{supp_ineq}
\rho + (\varphi-1) \ge \sqrt{\rho^2+(\varphi-1)^2} \ge \frac{1}{3}\delta(\varphi^*,\epsilon a(\varphi^*)) =\frac{1}{3}\sqrt{\epsilon^2a^2(\varphi^*)+(\varphi^*-1)^2}.
 \end{equation}
Note that this implies that for any pair $(\rho,\varphi)$, 
 \[ \frac{\varphi^*-1}{\rho+(\varphi-1)} \le \frac{\sqrt{\epsilon^2a^2(\varphi^*)+(\varphi^*-1)^2}}{\rho+(\varphi-1)} \le 3. \]
Also, by \eqref{star_derivs}, we have $\frac{\p \varphi^*}{\p\rho}\le 0$, and therefore $\frac{\varphi^*-1}{\rho+(\varphi-1)}$ is largest when $\rho=0$. Thus
\begin{equation}\label{bound_star}
\frac{\varphi^*-1}{\varphi-1} \le 3.
\end{equation}

Then, using \eqref{supp_ineq}, we can address the first coefficient in \eqref{nablaE_end} as
 \[ \abs{\frac{2\epsilon a(\varphi^*)-2(\varphi^*-1)}{\rho+\epsilon |a'(\varphi^*)|(\varphi-1)} -1} \le C\bigg( \frac{\epsilon a(\varphi^*)+(\varphi^*-1)}{\sqrt{(\varphi^*-1)^2+\epsilon^2a^2(\varphi^*)}} +1\bigg) \le C. \]
Similarly, by \eqref{supp_ineq} and \eqref{bound_star}, we can address the third coefficient in \eqref{nablaE_end} as
\begin{align*}
\abs{\frac{2\epsilon^2 a(\varphi^*)a'(\varphi^*) + 2\epsilon |a'(\varphi^*)|(\varphi^*-1)}{\rho+\epsilon |a'(\varphi^*)|(\varphi-1)} -1} &\le \frac{2\epsilon^2 a(\varphi^*)|a'(\varphi^*)|}{\sqrt{\epsilon^2a^2(\varphi^*)+(\varphi^*-1)^2}} + 2\frac{\varphi^*-1}{\varphi-1} +1 \\
&\le \frac{2\bar c_a\epsilon^2}{\sqrt{\epsilon^2a^2(\varphi^*)+(\varphi^*-1)^2} }+ 7.
\end{align*}
Now, if $0<(\varphi^*-1)\le (\eta_\epsilon-1)/2$, then by the spheroidal endpoint requirement of Definition \ref{admissible_a}, $a(\varphi^*) \ge C\epsilon$. If $(\varphi^*-1)\ge (\eta_\epsilon-1)/2$, then by \eqref{eta_epsilon}, $(\varphi^*-1)\ge C\epsilon^2$. Therefore 
\begin{align*}
\abs{\frac{2\epsilon^2 a(\varphi^*)a'(\varphi^*) + 2\epsilon |a'(\varphi^*)|(\varphi^*-1)}{\rho+\epsilon |a'(\varphi^*)|(\varphi-1)} -1} \le \frac{2\bar c_a\epsilon^2}{\sqrt{(\varphi^*-1)^2+\epsilon^2a^2(\varphi^*)} }+ 7 \le C\frac{\epsilon^2}{\epsilon^2}+7 \le C.
\end{align*}

Finally, using \eqref{varphi_star} and \eqref{bound_star}, we can address the middle $\frac{1}{\rho}$ coefficient of \eqref{nablaE_end}. We have 
\[ \abs{2\epsilon a(\varphi^*)-\rho} = \abs{2\frac{\varphi^*-1}{\varphi-1}-1}\rho \le 7\rho,\]
and therefore
\[ \frac{1}{\rho} \le \frac{C}{2\epsilon a(\varphi^*)-\rho}. \]

Altogether, for $\bx\in \Sigma^\epsilon_{\text{str}}$ and $1< \varphi <\eta_\epsilon$, we have that $\nabla(E\bv)$ satisfies 
\begin{equation}\label{nablaE2}
\begin{aligned}
\abs{\nabla(E\bv)(\varphi,\theta,\rho)} &\le C\bigg(\abs{\frac{\p\bv}{\p\varphi}} + \abs{\frac{1}{2\epsilon a(\varphi^*)-\rho}}\abs{\frac{\p\bv}{\p\theta}} + \abs{\frac{\p\bv}{\p\rho}} \bigg) \\
&\le C\abs{\nabla \bv(2\varphi^*-\varphi-1,\theta,2\epsilon a(\varphi^*)-\rho)}.
\end{aligned}
\end{equation}
A similar bound holds for $-\eta_\epsilon < \varphi < -1$. 
\end{proof}

{\bf Step 3: Cutoff and definition of $T$:} \\
To make use of the extension $E$ for $H^1(\mc{H}^\epsilon)$ functions with arbitrary support, we must define a cutoff function $\psi\in \mc{C}^1(\mc{H}^\epsilon)$. In particular, we take 
\begin{equation}\label{free_psi_def}
\psi(\varphi,\rho) = \begin{cases}
 \begin{cases}
1 & \text{ if } \delta(\varphi,\epsilon a(\varphi)) \le \rho \le \frac{4\delta(\varphi,\epsilon a(\varphi))}{3} \\
0 & \text{ if } \rho > \frac{5\delta(\varphi,\epsilon a(\varphi))}{3}
\end{cases} & \text{ if } \abs{\varphi} \le 1 \\
 \begin{cases}
1 & \text{ if } \delta(\varphi^*,\epsilon a(\varphi^*)) \le \sqrt{(\abs{\varphi}-1)^2+\rho^2} \le \frac{4\delta(\varphi^*,\epsilon a(\varphi^*))}{3} \\
0 & \text{ if } \sqrt{(\abs{\varphi}-1)^2+\rho^2} > \frac{5\delta(\varphi^*,\epsilon a(\varphi^*))}{3}
\end{cases} & \text{ if } 1\le \abs{\varphi} \le \eta_\epsilon
\end{cases}
\end{equation}
with smooth transition between 0 and 1, and $\mc{C}^1$ transition from $\abs{\varphi} \le 1$ to the endpoint sections. Here $\varphi^*$ is as in \eqref{varphi_star}. \\

For $\abs{\varphi}<1$, we require that the decay rate $\p\psi/\p\rho$ is such that $\nabla\psi$ satisfies 
\begin{equation}\label{grad_psi_bd1}
\abs{\nabla\psi(\varphi,\rho)} \le \abs{\frac{\p\psi}{\p\rho}}+\abs{\frac{\p\psi}{\p\varphi}} \le C\bigg(\frac{1}{\epsilon a(\varphi)} + \epsilon |a'(\varphi)|\bigg) \le \frac{1}{\epsilon a(\varphi)}C\big(1 + \epsilon^2\bar c_a \big) \le \frac{C}{\epsilon a(\varphi)}
\end{equation}
for $\epsilon$ sufficiently small. Similarly, for $1\le \abs{\varphi}<\eta_\epsilon$, we require that $\nabla\psi$ satisfies 
\begin{equation}\label{grad_psi_bd2}
\begin{aligned}
\abs{\nabla\psi(\varphi,\rho)} &\le \frac{C}{\delta(\varphi^*,\epsilon a(\varphi^*))} \le C\begin{cases}
\frac{1}{\epsilon(1-\epsilon^2c_a)\sqrt{\eta_\epsilon^2 - (\eta_\epsilon+1)^2/4}}, & 1 \le \varphi^* < \frac{\eta_\epsilon+1}{2} \\
\frac{2}{\eta_\epsilon-1}, & \frac{\eta_\epsilon+1}{2}\le \varphi^*<\eta_\epsilon
\end{cases} \\ 
&\le C\begin{cases}
\frac{1}{(\eta_\epsilon-1)\sqrt{c_{\eta,0}}(1-\epsilon^2c_a)}, & 1 \le \varphi^* < \frac{\eta_\epsilon+1}{2} \\
\frac{2}{\eta_\epsilon-1}, & \frac{\eta_\epsilon+1}{2} \le \varphi^*<\eta_\epsilon
\end{cases}\\
& \le \frac{C}{\eta_\epsilon-1},
\end{aligned}
\end{equation}
where we have used Definition \ref{admissible_a} and \eqref{eta_epsilon} to rewrite the bound. \\

We then define our preliminary extension operator $T: H^1(\mc{H}^\epsilon) \to H^1(\mc{W}^\epsilon)$ by
\begin{equation}\label{T_prelim}
 T\bu = E(\psi \bu) + (1-\psi)\bu. 
 \end{equation}

Note that the $L^2$ norm of the extension satisfies 
\begin{equation}\label{firstE_ineq}
\norm{T\bu}_{L^2(\mc{W}^\epsilon)} \le C \norm{\bu}_{L^2(\mc{H}^\epsilon)},
\end{equation}
where the constant $C$ is independent of $\epsilon$. \\

Furthermore, using \eqref{nablaE_bd}, \eqref{grad_psi_bd1}, and \eqref{grad_psi_bd2}, the symmetric gradient $\E(T\bu)$ satisfies 
\begin{equation}\label{extension1_bound}
\begin{aligned}
\abs{\E(T\bu)} &\le 2\abs{\nabla(E(\psi\bu))} + \abs{\E(\bu)} + 2\abs{\nabla\psi}\abs{\bu} \\
& \le C\big(\abs{\nabla(\psi\bu)}+\abs{\E(\bu)} \big)+ C\begin{cases}
\frac{1}{\epsilon a(\varphi)}\abs{\bu}, & \abs{\varphi}<1 \\
\frac{1}{\eta_\epsilon-1}\abs{\bu}, & 1\le \abs{\varphi}<\eta_\epsilon
\end{cases}
\end{aligned}
\end{equation}
for $C$ independent of $\epsilon$. \\

{\bf Step 4: Auxiliary Korn-type lemmas:} \\
Now, we would like to be able to adapt this extension operator $T$ to satisfy Lemma \ref{extension_free}. However, the $L^2(\mc{W}^\epsilon)$ bound for $\E(T\bu)$ is problematic (see the last term of \eqref{extension1_bound}). To address this issue, we follow a similar construction to \cite{mazya1997differentiable}, Chapter 3, which develops an extension operator for a thin, infinite cylinder with gradient bounds independent of $\epsilon$, as well as \cite{closed_loop}, Section 3.2, which adapts these arguments for the symmetric gradient. The new difficulty here is the change in the radius function $a(\varphi)$ from $O(1)$ to 0 along the filament, as well as the introduction of endpoints. \\

The key to removing the $\epsilon$-dependence in the estimate \eqref{extension1_bound} will involve dividing the fiber into many small segments with length roughly equal to radius on each segment. We will then construct an extension operator that makes use of the Korn-Poincar\'e inequality (Lemma \ref{diffeo_KP}) along with homogeneous rescaling (Corollary \ref{korn_poincare_free}) to get rid of the problematic last term of \eqref{extension1_bound}. This construction will involve working with a series of domains that are very similar to each other but not exactly identical. The main goal of this step is therefore to show that the Korn-Poincar\'e inequality (Lemma \ref{diffeo_KP}) holds with a uniform constant on domains that are similar but not exactly the same. To arrive at this result, we must first prove a series of Korn-type inequalities with a uniform constant over slightly deformed domains. \\
 
To this end, for a bounded, Lipschitz domain $\D\subset\R^3$ and neighborhood $\mc{N}$ of $\D$, we will work with $\mc{C}^2$ maps $\Psi: \mc{N}\to \R^3$ satisfying
\begin{equation}\label{diffeo_bound}
\Psi: \mc{N}\to \Psi(\mc{N}) \text{ is a diffeomorphism}; \quad \norm{\Psi}_{\mc{C}^2(\mc{N})} + \norm{\Psi^{-1}}_{\mc{C}^2(\Psi(\mc{N}))} \le M
\end{equation}
for some $M>0$. Note that $\Psi(\D)$ is also a bounded, Lipschitz domain.\\

The reason we consider Lipschitz domains rather than smoother domains is that we will be using the following results in a truncated cylinder domain; in particular, $\D$ must be allowed to have corners. However, the mappings $\Psi$ that we consider will be $\mc{C}^2$ in the sense that second derivatives of $\Psi$ are uniformly continuous up to $\p\D$. \\

We proceed to summarize the necessary tools to show Lemma \ref{diffeo_KP}. We begin by recalling the following important result from elasticity theory, which does not depend on the domain. The proof may be found in \cite{duvaut1976inequalities}.
 \begin{lemma}\emph{(Rigid motion)}\label{rigid_motion_free}
Let $\D \subset \R^3$ be any domain. If $\bu: \D \to \R^3$ with $\nabla \bu \in L^2(\D)$ satisfies $\nabla \bu +(\nabla \bu)^{\rm T}=0$, then $\bu$ is a rigid body motion: $\bu(\bx) = \bm{A}\bx+{\bm b}$ for some constant, antisymmetric $\bm{A}\in \R^{3\times3}$ and constant ${\bm b}\in \R^3$. 
\end{lemma}

Now, for a bounded, Lipschitz domain $\D$, we let $\sR$ denote the space of rigid motions:
\[ \mathcal{R} = \{ \bv\in H^1(\D) \ts : \ts \bv = \bm{A}\bx + \bm{b} \text{ for some } \bm{A}= -\bm{A}^{\rm T} \in \R^{3\times 3} \text{ and } \bm{b}\in \R^3\}.\]
For any $\bu\in H^1(\D)$, we define the $L^2$ projection $P_\sR\bu$ onto the space $\sR$:
\[ P_{\mathcal{R}}\bu = \bv\in \mathcal{R} \text{ such that } \|\bu-\bv\|_{L^2(\D)} \le \|\bu - \bw\|_{L^2(\D)}\quad  \text{for all } \bw\in \mathcal{R}. \]
By Lemma \ref{rigid_motion_free}, for any $\bv\in \sR$, we have $\E(\bv)=0$. \\

The first varying-domain result that we will need is a uniform Ne\v{c}as inequality for slightly deformed domains. Here we verify just the uniformity of the constant for small domain deformations; the proof of the inequality itself can be found in \cite{duvaut1976inequalities,boyer2012mathematical}.

\begin{lemma}[Varying domain Ne\v{c}as inequality]\label{necas_deform}
Let $\D\subset\R^3$ be a bounded, Lipschitz domain and consider any diffeomorphism $\Psi$ as in \eqref{diffeo_bound} for fixed $M>0$. Then for any $\bw\in H^{-1}(\Psi(\D))$ with $\nabla\bw\in H^{-1}(\Psi(\D))$, we in fact have $\bw\in L^2(\Psi(\D))$ and $\bw$ satisfies
\begin{equation}\label{necas}
\norm{\bw}_{L^2(\Psi(\D))} \le C(M) \big( \norm{\bw}_{H^{-1}(\Psi(\D))} + \norm{\nabla\bw}_{H^{-1}(\Psi(\D))} \big),
\end{equation}
where the constant $C(M)$ is uniform for any $\Psi$ satisfying \eqref{diffeo_bound}.
\end{lemma}

\begin{proof}
Let $\wt\bw=\bw\circ\Psi^{-1}\in L^2(\D)$, and let $\wt C$ denote the constant for which \eqref{necas} holds for $\wt \bw$. Note that for any $\bm{\phi}\in H^1_0(\D)$ with $\norm{\nabla \bm{\phi}}_{L^2(\D)}=1$, we have
\begin{align*}
\abs{ \int_{\D}\bm{\phi} \cdot \wt\bw } &= \bigg| \int_{\Psi(\D)}(\bm{\phi}\circ \Psi) \cdot \bw \abs{\det\nabla\Psi} \bigg| = \bigg| \int_{\Psi(\D)}\overline{\bm{\phi}} \cdot \bw \bigg| \\
&= \|\nabla\overline{\bm{\phi}}\|_{L^2(\Psi(\D))} \bigg| \int_{\Psi(\D)}\frac{\overline{\bm{\phi}}\cdot \bw}{\|\nabla\overline{\bm{\phi}}\|_{L^2(\Psi(\D))}} \bigg|,
\end{align*}
where we have defined $\overline{\bm{\phi}} := (\bm{\phi}\circ \Psi) \abs{\det\nabla\Psi}$. Taking the supremum over all such $\bm{\phi}$, we obtain
\[\norm{\wt\bw}_{H^{-1}(\D)} \le C(M) \norm{\bw}_{H^{-1}(\Psi(\D))}.\]
Likewise,
\[ \norm{\nabla\wt\bw}_{H^{-1}(\D)} \le C(M)  \norm{\nabla\bw}_{H^{-1}(\Psi(\D))}. \]
Furthermore,
\[ \int_{\Psi(\D)} \abs{\bw}^2 = \int_{\D} \abs{\wt\bw}^2\abs{\det\nabla\Psi^{-1}} \le M\int_{\D} \abs{\wt\bw}^2.\]
Altogether,
\begin{align*}
 \norm{\bw}_{L^2(\Psi(\D))} &\le M\norm{\wt\bw}_{L^2(\D)} \le M\wt C\big(\norm{\wt\bw}_{H^{-1}(\D)} + \norm{\nabla\wt\bw}_{H^{-1}(\D)} \big) \\
 &\le M \wt C\big(C(M)\norm{\bw}_{H^{-1}(\Psi(\D))} + C(M)\norm{\nabla\bw}_{H^{-1}(\Psi(\D))} \big).
 \end{align*}
\end{proof}

Using Lemma \ref{necas_deform}, we immediately obtain the following result.
\begin{lemma}[Varying domain Korn-type inequality, 1]\label{korn1_deform}
Let $\D\subset \R^3$ be a bounded, Lipschitz domain and let $\Psi$ be as in \eqref{diffeo_bound} for fixed $M>0$. Then $\bv\in H^1(\Psi(\D))$ satisfies
\begin{equation}
\norm{\bv}_{H^1(\Psi(\D))} \le C(M) \big( \norm{\E(\bv)}_{L^2(\Psi(\D))} + \norm{\bv}_{L^2(\Psi(\D))} \big),
\end{equation}
where the constant $C(M)$ is uniform for any $\Psi$ satisfying \eqref{diffeo_bound}. 
\end{lemma}

\begin{proof}
Define the set $K = \{\bv\in L^2(\Psi(\D)) \ts : \ts \E(\bv)\in L^2(\Psi(\D))\}$. Then $K$ is a Hilbert space with norm 
\[ \norm{\bv}_K^2 = \norm{\E(\bv)}_{L^2(\Psi(\D))}^2 + \norm{\bv}_{L^2(\Psi(\D))}^2.\]
Letting $\bx=(x_1,x_2,x_3)$, $\bv=(v_1,v_2,v_3)$, and $\E_{ik}(\bv) = \frac{1}{2}\big(\frac{\p v_i}{\p x_k} + \frac{\p v_k}{\p x_i}\big)$, we have that for $\bv\in K$, $\frac{\p\E_{ik}(\bv)}{\p x_j} \in H^{-1}(\Psi(\D))$ for each $1\le i,j,k\le 3$. Note that
\[ \frac{\p^2 v_k}{\p x_i\p x_j} = \frac{\p\E_{jk}(\bv)}{\p x_i} + \frac{\p\E_{ik}(\bv)}{\p x_j} - \frac{\p\E_{ij}(\bv)}{\p x_k}, \]
and therefore $\norm{\nabla^2\bv}_{H^{-1}(\Psi(\D))} \le 3 \norm{\nabla\big(\E(\bv)\big)}_{H^{-1}(\Psi(\D))}$. Then by Lemma \ref{necas_deform}, 
\begin{align*}
\norm{\nabla\bv}_{L^2(\Psi(\D))} &\le C(M) \big( \norm{\nabla\bv}_{H^{-1}(\Psi(\D))} + \norm{\nabla^2\bv}_{H^{-1}(\Psi(\D))} \big) \\
&\le C(M) \big( \norm{\nabla\bv}_{H^{-1}(\Psi(\D))} + \norm{\nabla\big(\E(\bv)\big)}_{H^{-1}(\Psi(\D))} \big) \\
&\le C(M) \big( \norm{\bv}_{L^2(\Psi(\D))} + \norm{\E(\bv)}_{L^2(\Psi(\D))} \big).
\end{align*}
\end{proof}

We now use Lemma \ref{korn1_deform} to show that the following pair of varying-domain Korn-type inequalities hold with a uniform constant.
\begin{corollary}\emph{(Varying domain Korn-type inequalities, 2)}\label{korn_type}
Consider a bounded, Lipschitz domain $\mc{D}\subset\R^3$ along with any diffeomorphism $\Psi$ satisfying \eqref{diffeo_bound} for fixed $M>0$. Suppose $\bv\in H^1(\Psi(\D))$ satisfies at least one of the following conditions:
\begin{enumerate}
\item $\bv \perp \sR$, the space of rigid motions on $\Psi(\D)$, or
\item $\bv$ vanishes on an open set of $\p(\Psi(\D))$ containing four non-coplanar points. 
\end{enumerate}
Then there exists a constant $C(M)>0$ such that $\bv$ satisfies
\begin{equation}\label{korn_type0}
 \|\nabla\bv\|_{L^2(\Psi(\D))} \le C(M)\|\E(\bv)\|_{L^2(\Psi(\D))}.
 \end{equation}
\end{corollary}

\begin{proof}
We first prove case 1. Assume that \eqref{korn_type0} does not hold. Then there is a sequence of $\mc{C}^2$-diffeomorphisms $\Psi_k$ and functions $\bv_k\in H^1(\Psi_k(\D))$ with $\bv_k\perp\sR$ and $\norm{\bv_k}_{L^2(\Psi_k(\D))} =1$ satisfying
\[ \norm{\nabla\bv_k}_{L^2(\Psi_k(\D))} > k \norm{\E(\bv_k)}_{L^2(\Psi_k(\D))}. \]

Let $\wt\bv_k=\bv_k\circ\Psi_k^{-1} \in H^1(\D)$. Then 
\[ 1=\int_{\Psi(\D)}\abs{\bv_k}^2 = \int_{\D}\abs{\wt\bv_k}^2\abs{\det\nabla\Psi_k^{-1}} \le M\int_{\D}\abs{\wt\bv_k}^2, \]
so $\norm{\wt\bv_k}_{L^2(\D)} \ge \frac{1}{\sqrt{M}} >0$. \\

Using Lemma \ref{korn1_deform}, we also have
\[ \norm{\E(\bv_k)}_{L^2(\Psi_k(\D))} < \frac{1}{k}\norm{\nabla\bv_k}_{L^2(\Psi_k(\D))} \le \frac{1}{k}\norm{\bv_k}_{H^1(\Psi_k(\D))} \le \frac{C(M)}{k}\big( \norm{\E(\bv_k)}_{L^2(\Psi_k(\D))} + 1 \big).\]
 Taking $k$ sufficiently large, we have 
\[ \bigg(1- \frac{C(M)}{k}\bigg)\norm{\E(\bv_k)}_{L^2(\Psi_k(\D))}< \frac{C(M)}{k} \to 0\]
as $k\to \infty$. Then, by definition of $\wt\bv_k$ and Lemma \ref{korn1_deform}, for $k$ sufficiently large we have
\[ \norm{\wt\bv_k}_{H^1(\D)} \le C(M) \norm{\bv_k}_{H^1(\Psi_k(\D))} \le C(M)\bigg(\frac{C(M)}{k-C(M)} +1 \bigg). \]
Passing to a subsequence (which we continue to denote by $\wt\bv_k$) and using Rellich compactness, $\wt\bv_k\to \wt\bv_\infty$ in $L^2(\D)$. Note also that by \eqref{diffeo_bound}, $\Psi_k\to \Psi_\infty$ in $\mc{C}^1$. \\

Now, choose any rigid motion $\bm{A}\bx+\bm{b} \in \sR$ on $\Psi_\infty(\D)$, where $\bm{A}$ is a constant, antisymmetric matrix and $\bm{b}$ is a constant vector. Note that this is also a well-defined rigid motion on $\Psi_k(\D)$ for any $k$. Then we have 
\begin{equation}\label{contrad_eq}
\begin{aligned}
0 &= \int_{\Psi_k(\D)} \bv_k\cdot(\bm{A}\bx+\bm{b}) = \int_{\D}\wt\bv_k \cdot(\bm{A}\Psi^{-1}_k(\bx)+\bm{b})\abs{\det\nabla\Psi^{-1}_k} \\
&\quad \to \int_{\D}\wt\bv_\infty \cdot(\bm{A}\Psi^{-1}_\infty(\bx)+\bm{b}) \abs{\det\nabla\Psi^{-1}_\infty} =  \int_{\Psi_\infty(\D)}\bv_\infty \cdot(\bm{A}\bx +\bm{b}).
\end{aligned}
\end{equation}
Thus $\bv_\infty\perp \sR$ on $\Psi_\infty(\D)$. However, since $\liminf_k\norm{\E(\bv_k)}_{L^2(\Psi_k(\D))}\ge \norm{\E(\bv_\infty)}_{L^2(\Psi_\infty(\D))}$, we have $\E(\bv_\infty)=0$, and by Lemma \ref{rigid_motion_free}, $\bv_\infty\in \sR$. Thus $\bv_\infty=0$, so $\bv_k\to 0$, contradicting $\norm{\bv_k}_{L^2(\Psi(\D))}\ge \frac{1}{\sqrt{M}}>0$. \\

To show case (2.) of Corollary \ref{korn_type}, we instead take $\bv_k\in H^1(\Psi_k(\D))$ vanishing on an open set of $\p(\Psi(\D))$ containing four non-coplanar points. The proof proceeds as above, except instead of arriving at a contradiction via \eqref{contrad_eq}, we use that each $\wt\bv_k$ vanishes on an open set of $\p\D$ and thus $\wt\bv_\infty$ does as well. Then $\bv_\infty$ also vanishes on an open set of $\p(\Psi_\infty(\D))$, but $\E(\bv_\infty)=0$, and therefore $\bv_\infty=0$. 
\end{proof}

Finally, using Corollary \ref{korn_type}, we can show the following: 
\begin{lemma}[Varying domain Korn-Poincar\'e inequality]\label{diffeo_KP}
Let $\mc{D}\subset\R^3$ be a bounded, Lipschitz domain and consider any diffeomorphism $\Psi$ satisfying \eqref{diffeo_bound} for fixed $M>0$. Then there exists a constant $C(M)>0$ such that 
any $\bv\in H^1(\Psi(\D))$ with $\bv\perp\sR$, the space of rigid motions on $\Psi(\D)$, satisfies
\begin{equation}
\norm{\bv}_{L^2(\Psi(\D))} \le C(M) \norm{\E(\bv)}_{L^2(\Psi(\D))}.
\end{equation}
\end{lemma}
 
\begin{proof}
Suppose that Lemma \ref{diffeo_KP} does not hold. Then there exist a sequence of diffeomorphisms $\Psi_k$ and functions $\bv_k\in H^1(\Psi_k(\D))$ with $\bv_k\perp\sR$ and $\norm{\bv_k}_{L^2(\Psi_k(\D))} =1$ such that 
\[ 1 = \norm{\bv_k}_{L^2(\Psi_k(\D))} > k\norm{\E(\bv_k)}_{L^2(\Psi_k(\D))}. \]

As in the proof of Corollary \ref{korn_type}, let $\wt\bv_k = \bv_k\circ \Psi_k^{-1}\in H^1(\D)$. Then $\norm{\wt\bv_k}_{L^2(\D)} \ge \frac{1}{\sqrt{M}} >0$. Furthermore, using Corollary \ref{korn_type}, we have 
\begin{align*}
 \int_{\D} \abs{\nabla\wt\bv_k}^2 &\le \int_{\Psi_k(\D)} \abs{\nabla\bv_k}^2\abs{\nabla\Psi_k^{-1}}^2\abs{\det\nabla \Psi_k} \le M^3\int_{\Psi_k(\D)} \abs{\nabla\bv_k}^2 \\
 &\le C(M) \int_{\Psi_k(\D)} \abs{\E(\bv_k)}^2\le \frac{1}{k^2} \to 0
 \end{align*}
as $k\to\infty$. Passing to a subsequence (which we continue to denote $\wt\bv_k$) and using Rellich compactness, we have that $\wt\bv_k\to\wt\bv_\infty$ in $L^2(\D)$. Also, since $\liminf_k\norm{\nabla\wt\bv_k}_{L^2(\D)} \le \norm{\nabla\wt\bv_\infty}_{L^2(\D)}$, we have $\nabla\wt\bv_\infty=0$, so $\wt\bv_\infty=\bm{b}$ for some constant vector $\bm{b}$. \\

But then, using that $\Psi_k\to\Psi_\infty$ in $\mc{C}^1$, by \eqref{diffeo_bound}, we have $\bv_\infty = \wt\bv_\infty\circ\Psi_\infty = \bm{b}$, and $\bv_\infty\perp\sR$, so $\bm{b}=0$. Thus $\wt\bv_\infty=0$, which contradicts $\norm{\wt\bv_k}_{L^2(\D)} \ge \frac{1}{\sqrt{M}} >0$.
 \end{proof}
 
 Lemma \ref{diffeo_KP} under homogeneous rescaling then gives rise to the following corollary.
 \begin{corollary}[Korn-Poincar\'e rescaling]\label{korn_poincare_free}
Let $\mc{D}\subset\R^3$ be a bounded, Lipschitz domain and $\Psi$ a diffeomorphism as in \eqref{diffeo_bound} for fixed $M>0$. Consider $\bv\in H^1(\Psi(\mc{D}))$ with $\bv\perp\sR$, the space of rigid motions on $\Psi(\D)$. Under homogeneous rescaling of the domain $\Psi(\mc{D})\to \lambda \Psi(\mc{D})$, $\lambda\in \R_+$, we have 
\begin{equation}\label{korn_poincare2}
\norm{\bv}_{L^2(\lambda \Psi(\mc{D}))} \le \lambda C(M)\norm{\E(\bu)}_{L^2(\lambda \Psi(\mc{D}))},
\end{equation}
where the constant $C(M)$ is uniform for all $\Psi$ satisfying \eqref{diffeo_bound}.
\end{corollary}


{\bf Step 5: Decomposition of domain $\mc{W}^\epsilon$:}\\
Now, we would like to make use of the scaling in Corollary \ref{korn_poincare_free} to design an extension operator that removes the $\epsilon$-dependence coming from the lower part of the norm in \eqref{extension1_bound}. To do so, we consider the slender body as the union of many tiny segments, each sufficiently small to be able to make use of the scaling in Corollary \ref{korn_poincare_free}. We consider the positive half of the fiber ($\varphi\ge0$) here; the negative half ($\varphi\le0$) follows by the same arguments. \\

Recalling the constant $\delta_a$ given by Definition \ref{admissible_a}, we partition the segment $[0,1]$ in the following way. We define a collection of points $\{p_n\}$ by
\begin{equation}\label{p_defs}
\begin{aligned}
p_1 = 1;\quad  p_{n+1} = \begin{cases}
p_n-\epsilon a(p_n) & \text{ if } p_n > 1-\delta_a \\
p_n-\epsilon & \text{ if } 0\le p_n \le 1-\delta_a.
\end{cases}
\end{aligned}
\end{equation}
Note that we continue iterating this procedure until we reach $N$ such that $p_{N+1}<0$. It will also be necessary to define $p_{N+2}=p_{N+1}-\epsilon$. We also let $N_\delta$ denote 
\begin{equation}\label{N_delta}
N_\delta := \min \{ n\ge 0 \ts: \ts p_n \le 1-\delta_a \}.
\end{equation}

Using the collection of points $\{p_n\}$, we define an open cover $\{Q_n\}$ of the interval $(0,\eta_\epsilon)$ as follows: for $\varphi\in (p_{N+2},2\eta_\epsilon-1)$, define
\begin{equation}\label{cover_Q}
\begin{aligned}
Q_0 = \big\{\varphi \ts &: \ts  p_1 <\varphi<2\eta_\epsilon-1 \big\}, \quad Q_1 = \big\{\varphi \ts : \ts  p_2 <\varphi<(\eta_\epsilon+1)/2 \big\}, \\ 
\quad Q_n &= \big\{\varphi \ts : \ts  p_{n+1} <\varphi<p_{n-1} \big\}, \quad n\ge2.
\end{aligned}
\end{equation}

Using the cover $\{Q_n\}$, we subdivide the domain $\mc{W}^\epsilon$ as (see Figure \ref{fig:pn_pic})
\begin{equation}\label{cover_W}
\mc{W}_n^\epsilon = \big\{ \bx=\bx(\varphi,\theta,\rho) \in \mc{W}^\epsilon \ts : \ts \varphi\in Q_n \big\}.
\end{equation}

\begin{figure}[!h]
\centering
\includegraphics[scale=0.5]{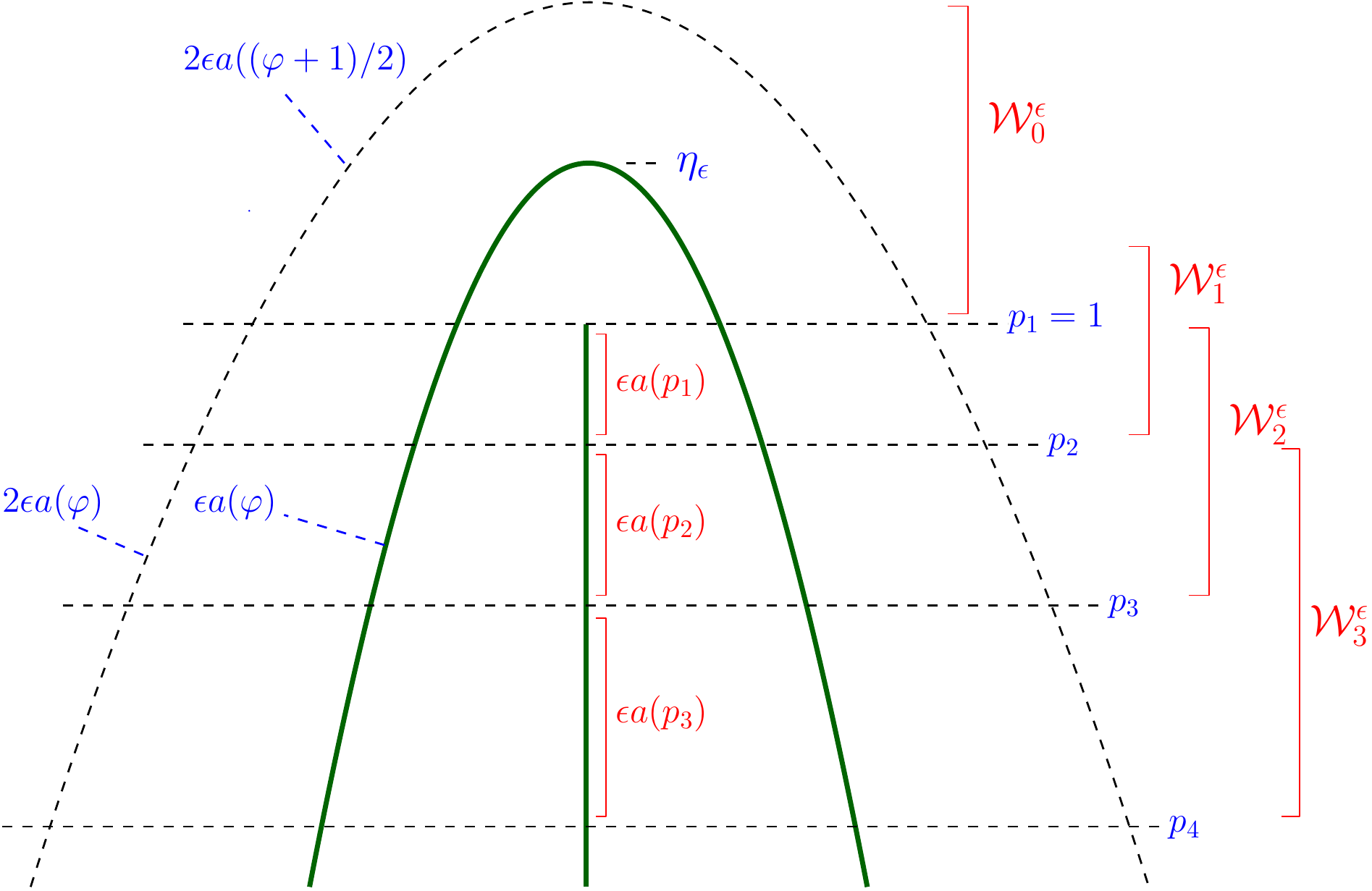}\\
\caption{We subdivide the domain $\mc{W}^\epsilon$ using the points $\{p_n\}$ defined via the process in \eqref{p_defs}. }
\label{fig:pn_pic}
\end{figure}

We also define
\begin{equation}\label{Hn_epsilon}
\mc{H}_n^\epsilon:=\mc{W}_n^\epsilon \cap \mc{H}^\epsilon. 
\end{equation}

Now, using Definition \ref{admissible_a} -- in particular, the endpoint monotonicity requirement -- along with \eqref{p_defs}, for each $n\in [2,N_\delta]$ we can write
\begin{align*}
 a(p_n) &= a(p_{n-1})- \epsilon \overline a(p_n,p_{n-1}) \quad \text{ for some function } \abs{\overline a} \le \bar c_a, 
 \end{align*}
where $\bar c_a$ and $\delta_a$ are as defined in Section \ref{geometry_section}. Then the ratios $a(p_n)/a(p_{n-1})$ and $a(p_{n-1})/a(p_n)$ satisfy 
\begin{equation}\label{ratio_bounds}
0 < \frac{a(p_n)}{a(p_{n-1})}, \frac{a(p_{n-1})}{a(p_n)} \le 1+ \epsilon \bar c_a/a(1), \quad 2\le n\le N_\delta.
\end{equation}
Using the spheroidal endpoint requirement of Definition \ref{admissible_a}, we have that these ratios are bounded independent of $\epsilon$. \\

Furthermore, using \eqref{eta_epsilon} and Definition \ref{admissible_a}, we have that the ratios $(\eta_\epsilon-1)/(\epsilon a(1))$ and $\epsilon a(1)/(\eta_\epsilon-1)$ satisfy
\begin{equation}\label{ratio_bound_end}
0<\frac{\eta_\epsilon-1}{\epsilon a(1)} \le \frac{c_{\eta}}{\sqrt{2c_{\eta,0}}- \epsilon c_a}, \quad 0<\frac{\epsilon a(1)}{\eta_\epsilon-1} \le \frac{\sqrt{2c_\eta}+\epsilon c_a}{c_{\eta,0}}
\end{equation}
and are thus also bounded independent of $\epsilon$. \\

Therefore we can can define the sets $\mc{W}_n^*$ as follows:
\begin{equation}\label{Wn_0}
 \mc{W}_n^\epsilon=\begin{cases}
\epsilon a(p_1)\mc{W}_0^*, &n=0 \\
\epsilon a(p_n)\mc{W}_n^*, & 1\le n\le N_\delta\\
\epsilon \mc{W}_n^*, & N_\delta+1 \le n\le N+1
\end{cases}
\end{equation} 
where each $\mc{W}_n^*$ is contained in a cylinder of radius 
\[ \begin{cases}
2, &n=0 \\
2\frac{a(p_{n+1})}{a(p_n)}, & 1\le n \le N_\delta \\
2, & N_\delta+1 \le n\le N+1
\end{cases} \] 
and height 
\[ \begin{cases}
2\frac{\eta_\epsilon-1}{a(1)}, & n=0 \\
1+\frac{\eta_\epsilon-1}{2a(1)}, & n=1 \\
1+ \frac{a(p_{n-1})}{a(p_n)}, & 2\le n\le N_\delta \\ 
2, & N_\delta+1 \le n \le N+1.
\end{cases} \]
See Figure \ref{fig:W0_pic} for a depiction of $\mc{W}_n^*$ for $2\le n\le N_\delta$. In particular, due to \eqref{ratio_bounds} and \eqref{ratio_bound_end}, the volume of each slice $\mc{W}_n^*$ is bounded independent of $\epsilon$. By the same arguments, we can also write $\mc{H}_n^\epsilon$ \eqref{Hn_epsilon} as
\begin{equation}\label{Hn_0}
 \mc{H}_n^\epsilon= \begin{cases}
\epsilon a(p_1)\mc{H}_0^*, & n=0 \\
\epsilon a(p_n)\mc{H}_n^*, & 1\le n \le N_\delta \\
\epsilon \mc{H}_n^*, & N_\delta+1\le n\le N+1,
\end{cases}
\end{equation} 
where the volume of each $\mc{H}_n^*$ is bounded independent of $\epsilon$. \\

\begin{figure}[!h]
\centering
\includegraphics[scale=0.6]{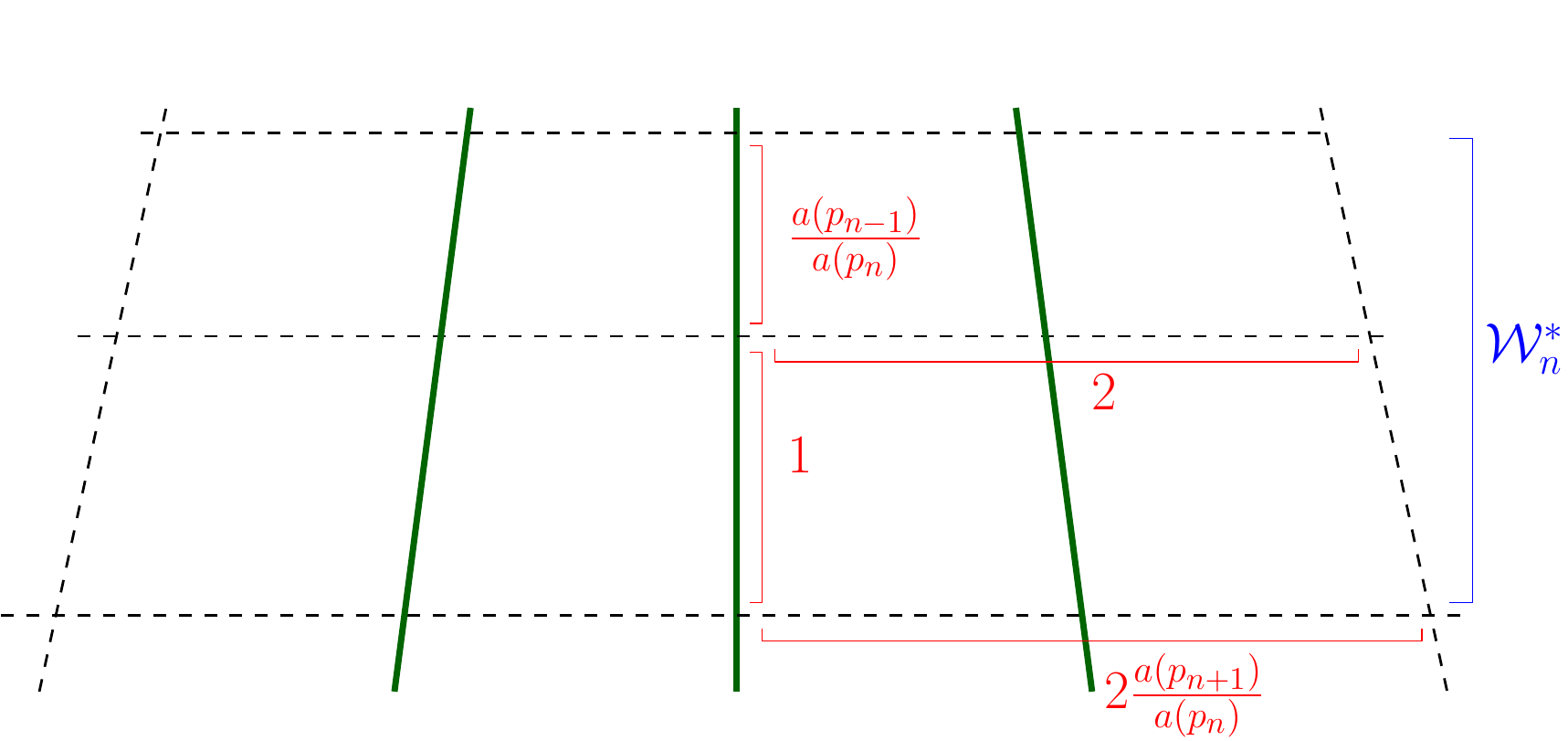}\\
\caption{For $2\le n \le N_\delta$, we may think of the domain $\mc{W}_n^\epsilon$ as $\mc{W}_n^\epsilon=\epsilon a(p_n)\mc{W}_n^*$, a rescaled version of the domain $\mc{W}_n^*$, the volume of which is bounded independent of $\epsilon$. A similar result holds on the endpoint segments $n=0,1$ and on the segments $N_\delta+1\le n\le N+1$. }
\label{fig:W0_pic}
\end{figure}

\begin{remark}\label{KP_domain_dep}
The domains $\mc{H}^*_n$, $1\le n\le N_\delta$, are each slightly different in size and shape; however, each can be deformed via a diffeomorphism satisfying \eqref{diffeo_bound} for $M$ small to a cylinder of height 2 and radius 2 with a cylindrical hole of radius 1. For any $1\le n\le N_\delta$, we may parameterize $\mc{H}^*_n$ with respect to cylindrical coordinates as 
\[ \mc{H}^*_n = \bigg\{ (\varphi,\theta,\rho) \ts :\ts 0\le \varphi\le 1+\frac{a(p_{n-1})}{a(p_n)}, \ts 0\le \theta<2\pi, \ts \frac{a(p_{n+1}-\varphi)}{a(p_n)} \le \rho \le 2\frac{a(p_{n+1}-\varphi)}{a(p_n)} \bigg\}. \]
Then, using \eqref{ratio_bounds}, the transformation
\[ \Psi_n: (\varphi,\theta,\rho) \mapsto \bigg(\frac{a(p_n)}{a(p_n)+a(p_{n-1})}\varphi,\theta,\frac{a(p_n)}{a(p_{n+1}-\varphi) }\rho \bigg) \]
gives the desired diffeomorphism. Note that then, by Lemma \ref{diffeo_KP}, the Korn-Poincar\'e inequality over each of the domains $\mc{H}^*_n$, $1\le n\le N+1$, has a uniform constant. \\

The endpoint segment ($n=0$) must be treated separately from the above construction. Instead, using \eqref{ratio_bound_end}, we may deform $\mc{H}^*_0$ to a hemisphere of radius 2. Since the endpoint shape implicitly depends on $\epsilon$ via Definition \ref{admissible_a}, this will provide us with a uniform Korn-Poincar\'e inequality (Lemma \ref{diffeo_KP}) on the endpoint segment for any value of $\epsilon$. \\

We may then take the maximum of the two different Korn-Poncar\'e constants to serve as a uniform Korn-Poincar\'e constant for all segments of the fiber.  
\end{remark}

Now, for $\bu$ defined in $\mc{H}^\epsilon$, let $\bu_n$ denote the restriction $\bu\big|_{\mc{H}_n^\epsilon}$. On each $\mc{H}^\epsilon_n$, we consider the projection $P_{\sR}\bu_n$ onto rigid motions, and write $P_{\sR}\bu_n = \bm{A}_n\bx+\bm{b}_n$ for some $\bm{A}_n=-\bm{A}_n^{\rm T}\in \R^{3\times 3}$ and $\bm{b}_n\in \R^3$. We omit the dependence of the projection $P_{\sR}$ on $n$ and $\epsilon$. We take
\[ \overline P_{\sR}\bu_n =  \bm{A}_n\bx+\bm{b}_n, \quad \bx\in \Sigma_n^\epsilon\]
to be the extension of $P_{\sR}\bu_n$ to all of $\mc{W}_n^\epsilon$. Note that on each $\mc{W}_n^\epsilon$, we have 
\begin{equation}\label{barPR_bound}
\norm{\overline P_{\sR}\bu_n}_{L^2(\mc{W}_n^\epsilon)} \le C\norm{P_{\sR}\bu_n}_{L^2(\mc{H}_n^\epsilon)}.
\end{equation}

Let $\{\Phi_n \}$ be a partition of unity subordinate to the cover $\{Q_n\}$. We require each $\Phi_n$ to satisfy 
\begin{equation}\label{psiN_bd}
\abs{\nabla\Phi_n}\le C\begin{cases}
\frac{1}{\epsilon a(1)}, & n=0,1 \\
\frac{1}{\epsilon a(p_{n-1})}, & 2\le n\le N_\delta \\
\frac{1}{\epsilon}, & N_\delta+1\le N+1
\end{cases}
\end{equation}
where the constant $C$ is independent of $\epsilon$. Note that due to \eqref{ratio_bound_end}, the bound for $n=0,1$ is equivalent to requiring that $\abs{\nabla\Phi_n}\le C/(\eta_\epsilon-1)$, $n=0,1$. \\

{\bf Step 6: Definition of extension operator $T_\epsilon$:}\\
We now define the extension operator $T_\epsilon: H^1(\mc{H}^\epsilon) \to H^1(\mc{W}^\epsilon)$ that satisfies Lemma \ref{extension_free}. For $\bx\in \Sigma^\epsilon_\text{str}$, we take
\begin{equation}\label{extension_op_free}
\begin{aligned}
T_\epsilon\bu(\bx) &= \bm{T}_1(\bx)+\bm{T}_2(\bx); \\
 \bm{T}_1(\bx) &:= \sum_{n=0}^{N+1} \Phi_n(\varphi) \bigg( \overline P_{\sR}\bu_n\bigg)(\bx) \\
 \bm{T}_2(\bx) &:= \sum_{n=0}^{N+1} \Phi_n(\varphi) \bigg(T(\bu_n-P_{\sR}\bu_n)\bigg)(\bx),
\end{aligned}
\end{equation}
where the operator $T$ is as in \eqref{T_prelim}. Note that $T_\epsilon\bu\big|_{\mc{H}^\epsilon} = \bu$. \\

It remains to estimate the symmetric gradient $\E(T_\epsilon\bu)$ of the extension \eqref{extension_op_free}. We begin by estimating $\norm{\E(\bm{T}_1)}_{L^2(\mc{W}^\epsilon)}$, which is simpler. \\

Define a finer partition $\{\wt Q_n\}$ of the interval $(0,\eta_\epsilon)$:
\begin{equation}\label{part_wtQ}
\wt Q_0 = \big\{\varphi \ts : \ts  p_1 <\varphi<\eta_\epsilon \big\}, \quad \wt Q_n = \big\{\varphi \ts : \ts  p_{n+1} <\varphi\le p_n \big\}, \quad n\ge2.
\end{equation}
For each $n\ge 0$, define the slices $\wt {\mc{W}}_n^\epsilon$ of the fiber as follows:
\begin{equation}\label{tildeW_def}
\wt{\mc{W}}_n^\epsilon = \big\{ \bx = \bx(\varphi,\theta,\rho)\in \mc{W}^\epsilon \ts : \ts \varphi \in \wt Q_n \big\}.
\end{equation}
Accordingly, we define $\wt {\mc{H}}_n^\epsilon = \wt {\mc{W}}_n^\epsilon\cap  \mc{H}^\epsilon$. \\

Note that $\wt{\mc{W}}_n^\epsilon\subset \mc{W}_n^\epsilon$ and $\wt{\mc{W}}_n^\epsilon\subset \mc{W}_{n+1}^\epsilon$ for each $n$, and $\Phi_n(\varphi)+\Phi_{n+1}(\varphi)=1$ on $\wt{\mc{W}}_n^\epsilon$. Then on each $\wt{\mc{W}}_n^\epsilon$, we may write 
\[  \bm{T}_1(\bx) = \overline P_\sR\bu_n + \Phi_{n+1}(\varphi) (\overline P_\sR\bu_{n+1}-\overline P_\sR\bu_n).\]

Using the bounds \eqref{barPR_bound} and \eqref{psiN_bd}, for $2\le n\le N_\delta$ we have 
\begin{align*}
\norm{\E(\bm{T}_1)}_{L^2(\wt{\mc{W}}_n^\epsilon)} &\le \frac{C}{\epsilon a(p_{n-1})}\norm{\overline P_\sR\bu_{n+1}-\overline P_\sR\bu_n}_{L^2(\wt{\mc{W}}_n^\epsilon)} \le \frac{C}{\epsilon a(p_{n-1})}\norm{P_\sR\bu_{n+1}- P_\sR\bu_n}_{L^2(\wt{\mc{H}}_n^\epsilon)} \\
& \le \frac{C}{\epsilon a(p_{n-1})}\bigg(\norm{\bu_{n+1}-P_\sR\bu_{n+1}}_{L^2(\mc{H}_{n+1}^\epsilon)}+\norm{\bu_n -P_\sR\bu_n}_{L^2(\mc{H}_n^\epsilon)} \bigg),
\end{align*}
since $\wt{\mc{H}}_n^\epsilon\subset \mc{H}_n^\epsilon$ and $\wt{\mc{H}}_n^\epsilon\subset \mc{H}_{n+1}^\epsilon$. Then, using Remark \ref{KP_domain_dep}, by \eqref{Hn_0} and Lemma \ref{korn_poincare_free} we have
\begin{align*}
\norm{\E(\bm{T}_1)}_{L^2(\wt{\mc{W}}_n^\epsilon)} &\le C\frac{a(p_n)}{a(p_{n-1})} \bigg(\norm{\E(\bu_{n+1})}_{L^2(\mc{H}_{n+1}^\epsilon)}+\norm{\E(\bu_n)}_{L^2(\mc{H}_n^\epsilon)} \bigg), \quad 2\le n\le N_\delta.
\end{align*}

Following the same arguments on the endpoint segments $n=0,1$ and on segments $n\ge N_\delta+1$, we obtain 
\begin{align*}
\norm{\E(\bm{T}_1)}_{L^2(\wt{\mc{W}}_n^\epsilon)} &\le 
C \bigg(\norm{\E(\bu_{n+1})}_{L^2(\mc{H}_{n+1}^\epsilon)}+\norm{\E(\bu_n)}_{L^2(\mc{H}_n^\epsilon)} \bigg), & n=0,1 \text{ or } n\ge N_\delta+1.
\end{align*}

Squaring and summing over $n$, using orthogonality from the nearly-disjoint supports of each $\bu_n$ along with \eqref{ratio_bounds}, we thus have
\begin{equation}\label{T1_bound}
\norm{\E(\bm{T}_1)}_{L^2(\mc{W}^\epsilon)}^2\le C\sum_{n=0}^{N+1} \norm{\E(\bm{T}_1)}_{L^2(\mc{H}_n^\epsilon)}^2 \le C\norm{\E(\bu)}_{L^2(\mc{H}^\epsilon)}^2
\end{equation}
for $C$ independent of $\epsilon$. \\

Next we estimate $\bm{T}_2$. Using \eqref{psiN_bd}, we have 
\begin{equation}\label{T2_bound0}
\begin{aligned}
\norm{\E(\bm{T}_2)}_{L^2(\mc{W}^\epsilon)}^2 &\le \sum_{n=0}^1 \bigg(\frac{C}{\epsilon a(1)}\norm{T(\bu_n- P_{\sR}\bu_n)}_{L^2(\mc{W}_n^\epsilon)}^2 + \norm{\E\big(T(\bu_n-P_{\sR}\bu_n)\big)}_{L^2(\mc{W}_n^\epsilon)}^2 \bigg) \\
&\quad + \sum_{n=2}^{N_\delta} \bigg(\frac{C}{\epsilon a(p_{n-1})}\norm{T(\bu_n- P_{\sR}\bu_n)}_{L^2(\mc{W}_n^\epsilon)}^2 + \norm{\E\big(T(\bu_n-P_{\sR}\bu_n)\big)}_{L^2(\mc{W}_n^\epsilon)}^2 \bigg) \\
&\quad + \sum_{n=N_\delta+1}^{N+1} \bigg(\frac{C}{\epsilon}\norm{T(\bu_n- P_{\sR}\bu_n)}_{L^2(\mc{W}_n^\epsilon)}^2 + \norm{\E\big(T(\bu_n-P_{\sR}\bu_n)\big)}_{L^2(\mc{W}_n^\epsilon)}^2\bigg).
\end{aligned}
\end{equation}

Now, using \eqref{firstE_ineq} along with Remark \ref{KP_domain_dep}, the scaling \eqref{Hn_0}, and Corollary \ref{korn_poincare_free}, we have 
\begin{equation}\label{T2_bound1}
\begin{aligned}
\norm{T(\bu_n- P_{\sR}\bu_n)}_{L^2(\mc{W}_n^\epsilon)} &\le C\norm{\bu_n- P_{\sR}\bu_n}_{L^2(\mc{H}_n^\epsilon)} \\
&\le C\begin{cases}
\epsilon a(p_1)\norm{\E(\bu_0)}_{L^2(\mc{H}_0^\epsilon)}, & n=0 \\
\epsilon a(p_n)\norm{\E(\bu_n)}_{L^2(\mc{H}_n^\epsilon)}, & 1\le n\le N_\delta \\
\epsilon \norm{\E(\bu_n)}_{L^2(\mc{H}_n^\epsilon)}, & N_\delta+1\le n \le N+1.
\end{cases}
\end{aligned}
\end{equation}

Furthermore, noting that $\E(\bu_n-P_\sR\bu_n) = \bu_n$ and using \eqref{extension1_bound}, we have
\begin{align*}
\norm{\E\big(T(\bu_n-P_{\sR}\bu_n)\big)}_{L^2(\mc{W}_n^\epsilon)} &\le C \big(\norm{\E(\bu_n)}_{L^2(\mc{H}_n^\epsilon)}+\norm{\nabla(\psi \bu_n-\psi P_{\sR}\bu_n)}_{L^2(\mc{H}_n^\epsilon)} \big) \\
&\qquad + C\begin{cases}
\frac{1}{\eta_\epsilon-1} \norm{\bu_n-P_\sR\bu_n}_{L^2(\mc{H}_n^\epsilon)}, & n=0,1 \\
\frac{1}{\epsilon a(p_{n-1})} \norm{\bu_n-P_\sR\bu_n}_{L^2(\mc{H}_n^\epsilon)}, & 2\le n \le N_\delta \\
\frac{1}{\epsilon a_0} \norm{\bu_n-P_\sR\bu_n}_{L^2(\mc{H}_n^\epsilon)}, & N_\delta+1 \le n,
\end{cases}
\end{align*}
where $a_0$ is as in Definition \ref{admissible_a}. \\

Now, due to the cutoff function $\psi$, we have that $\psi \bu_n-\psi P_{\sR}\bu_n$ vanishes for $\delta(\varphi,\rho)> \frac{5}{3}\delta(\varphi^*,\epsilon a(\varphi^*))$, and thus by Corollary \ref{korn_type}, case (2.), we obtain
\begin{align*}
\norm{\nabla(\psi\bu_n-\psi P_{\sR}\bu_n)}_{L^2(\mc{H}_n^\epsilon)} &\le C\norm{\E(\psi\bu_n-\psi P_{\sR}\bu_n)}_{L^2(\mc{H}_n^\epsilon)} \\
&\le C\norm{\E(\bu_n)}_{L^2(\mc{H}_n^\epsilon)}+ C\begin{cases}
\frac{1}{\eta_\epsilon-1} \norm{\bu_n-P_\sR\bu_n}_{L^2(\mc{H}_n^\epsilon)}, & n=0,1 \\
\frac{1}{\epsilon a(p_{n-1})} \norm{\bu_n-P_\sR\bu_n}_{L^2(\mc{H}_n^\epsilon)}, & 2\le n \le N_\delta \\
\frac{1}{\epsilon a_0} \norm{\bu_n-P_\sR\bu_n}_{L^2(\mc{H}_n^\epsilon)}, & N_\delta+1 \le n.
\end{cases} 
\end{align*}

Then, using Remark \ref{KP_domain_dep}, the scaling \eqref{Hn_0}, and Corollary \ref{korn_poincare_free}, we have 
\begin{equation}\label{T2_bound2}
\norm{\E\big(T(\bu_n-P_{\sR}\bu_n)\big)}_{L^2(\mc{W}_n^\epsilon)} \le C\norm{\E(\bu_n)}_{L^2(\mc{H}_n^\epsilon)}\begin{cases}
1+ \frac{\epsilon a(1)}{\eta_\epsilon-1}, & n=0,1 \\
1+ \frac{a(p_n)}{a(p_{n-1})}, & 2\le n \le N_\delta \\
1+ \frac{1}{a_0}, & N_\delta+1 \le n.
\end{cases} 
\end{equation}

Using the estimates \eqref{T2_bound1} and \eqref{T2_bound2} in \eqref{T2_bound0}, and again using orthogonality from the nearly-disjoint supports of each $\bu_n$, we have 
\begin{equation}\label{T2_bound}
\norm{\E(\bm{T}_2)}_{L^2(\mc{W}^\epsilon)}^2 \le C\sum_{n=0}^{N-1}\norm{\E(\bu_n)}_{L^2(\mc{H}^\epsilon_n)}^2 \le C\norm{\E(\bu)}_{L^2(\mc{H}^\epsilon)}^2
\end{equation}
where, by the ratio bounds \eqref{ratio_bounds} and \eqref{ratio_bound_end}, $C$ is independent of $\epsilon$. Combining \eqref{T1_bound} and \eqref{T2_bound}, we obtain the $\epsilon$-independent estimate (2.) of Lemma \ref{extension_free}. 
\end{proof}


\subsubsection{Korn and Sobolev inequalities}\label{korn_stab}
Using the extension operator guaranteed by Corollary \ref{extension_free2}, we can easily show that the Korn and Sobolev inequalities (Lemmas \ref{korn_free} and \ref{sob_free}, respectively) hold in $\Omega_\epsilon$ with constants that are independent of $\epsilon$. 


\begin{proof}[Proof of Lemma \ref{korn_free}:]
Recall that for any $\bv\in D^{1,2}(\R^3)$, the Korn inequality over all of $\R^3$ holds with constant $c_K=\sqrt{2}$: 
\begin{align*}
\int_{\R^3} |\E(\bv)|^2 \ts d\bx &= \int_{\R^3} \bigg(2|\nabla \bv|^2 + 2\nabla\bv:(\nabla\bv)^{\rm T}\bigg) \ts d\bx \\ 
&= \int_{\R^3} 2|\nabla \bv|^2 \ts d\bx + 2\int_{\R^3} |\dive \ts \bv|^2 \ts d\bx \ge \int_{\R^3} 2|\nabla \bv|^2 \ts d\bx.
\end{align*}
Here we have used integration by parts twice on the second integral term. \\

Then, using the extension operator $\wt{T}_\epsilon$ from Corollary \ref{extension_free2}, we have 
\begin{align*}
\norm{\nabla\bu}_{L^2(\Omega_\epsilon)} &\le \norm{\nabla (\wt{T}_\epsilon \bu)}_{L^2(\R^3)} \le \sqrt{2}\norm{\E (\wt{T}_\epsilon \bu)}_{L^2(\R^3)} \le C\norm{\E(\bu)}_{L^2(\Omega_\epsilon)}.
\end{align*}
\end{proof}

Similarly, we can use the extension operator $\wt{T}_\epsilon$ to show that the Sobolev inequality \eqref{sob_free} holds in $\Omega_\epsilon$ with a constant independent of $\epsilon$.

\begin{proof}[Proof of Lemma \ref{sob_free}]
First note that, using the Korn inequality on $\R^3$, the extension $\wt{T}_\epsilon$ from Corollary \ref{extension_free2} satisfies
\begin{align*}
\norm{\nabla(\wt{T}_\epsilon\bu)}_{L^2(\R^3)} \le \sqrt{2}\norm{\E(\wt{T}_\epsilon\bu)}_{L^2(\R^3)} \le C\norm{\E(\bu)}_{L^2(\Omega_\epsilon)} \le C\norm{\nabla\bu}_{L^2(\Omega_\epsilon)}.
\end{align*}

We then have 
\begin{align*}
\norm{\bu}_{L^6(\Omega_\epsilon)} \le \norm{\wt{T}_\epsilon\bu}_{L^6(\R^3)} \le c_R\norm{\nabla(\wt{T}_\epsilon\bu)}_{L^2(\R^3)} \le C\norm{\nabla\bu}_{L^2(\Omega_\epsilon)},
\end{align*}
where $c_R$ is the Sobolev constant over all of $\R^3$. Since $c_R$ is a universal constant independent of the slender body geometry, we do not note its dependence in the statement of Lemma \ref{sob_free} or elsewhere where the Sobolev inequality is used.
\end{proof}

\subsubsection{Pressure estimate}\label{press_stab}
In this section we prove Lemma \ref{press_est_free} detailing the $\epsilon$-independence of the pressure constant $c_P$. The proof is essentially the same as in the closed loop setting \cite{closed_loop}, but we recap the arguments here to make note of slight differences due to the free endpoint geometry. 

\begin{proof}[Proof of Lemma \ref{press_est_free}:]
Consider the region $\mc{O}$ \eqref{mc_O} of the slender body. Recall that by \cite{galdi2011introduction}, Theorem III.3.1, given $P\in L^2(\mc{O})$, there exists $\bm{B}\in H^1_0(\mc{O})$ satisfying 
\begin{equation}\label{constant_CB}
 \dive \ts\bm{B}=P \ts \text{ in }\mc{O}, \quad \norm{\nabla \bm{B}}_{L^2(\mc{O})} \le c_B\norm{P}_{L^2(\mc{O})},
 \end{equation}
 where the constant $c_B$ has a useful explicit representation provided that $\mc{O}$ can be written as a finite union of star-shaped domains. \\
 
 We verify that the region $\mc{O}$ can indeed be written as a finite union of domains star-shaped with respect to balls of uniform radius. Note that by construction, since $\epsilon a \le \epsilon \le \frac{r_{\max}}{4}$ and the fiber is non-self-intersecting \eqref{no_intersect}, the region $\mc{O}$ satisfies a uniform interior sphere condition with radius $r_{\max}/2$. In particular, we can consider $\mc{O}$ as the infinite union of balls of radius $r_{\max}/2$. \\
 
As in \cite{closed_loop}, we can follow the construction in Lemma 2, Chapter 1.1.9 of  \cite{maz2013sobolev} to show that this uniform sphere condition implies that 
 \[ \mc{O} = \bigcup_{k=1}^N \mc{O}_k \]
 where each $\mc{O}_k$ is star-shaped with respect to balls of radius $r_{\max}/2$ and $N$ depends only on $\kappa_{\max}$ and $c_\Gamma$. Then, according to \cite{galdi2011introduction}, equation III.3.27, the constant $c_B$ in \eqref{constant_CB} satisfies 
\[ c_B \le c_0 \bigg(\frac{\delta(\mathcal{O})}{r_{\max}} \bigg)^3\bigg(1+ \frac{\delta(\mathcal{O})}{r_{\max}} \bigg) \]
where $\delta(\mathcal{O})$ is the diameter of $\mathcal{O}$ and $c_0$ depends on the diameters of each $\mathcal{O}_k$, which are bounded independent of $\epsilon$. Using this form of $c_B$, the $\epsilon$-independence of the constant $c_P$ follows by exactly the same construction as in the closed loop case (see \cite{closed_loop}, Appendix A.2.5), which we do not repeat here. 
\end{proof}


\bibliography{SBT_free.bib}{}
\bibliographystyle{siam}


\end{document}